\documentclass[a4paper,twosided,12pt]{amsart}

\usepackage{fullpage}
\usepackage{amssymb,amsmath}
\usepackage{float}
\usepackage{graphicx}
\usepackage{epsf}
\usepackage{ psfrag}
\usepackage{subcaption}
\usepackage{cleveref}

\usepackage[all]{xy}

\usepackage{caption}

\usepackage{lscape} 
\usepackage{enumerate,enumitem}

\usepackage{inputenc}
\usepackage[usenames, dvipsnames]{color}

\theoremstyle{plain}
\newtheorem{lemma}{Lemma}[section]
\newtheorem{prop}[lemma]{Proposition}
\newtheorem{cor}[lemma]{Corollary}
\newtheorem{conj}[lemma]{Conjecture}
\newtheorem{thm}[lemma]{Theorem}
\newtheorem{de}[lemma]{Definition}

\newtheorem{rem}[lemma]{Remark}

\newtheorem*{thm*}{Theorem}
\newtheorem*{cor*}{Corollary}

\newtheorem{innercustomthm}{Theorem}
\newenvironment{customthm}[1]
  {\renewcommand\theinnercustomthm{#1}\innercustomthm}
  {\endinnercustomthm}
  
\newtheorem{innercustomcor}{Corollary}
\newenvironment{customcor}[1]
  {\renewcommand\theinnercustomcor{#1}\innercustomcor}
  {\endinnercustomcor}
\begin{document}

\title{Tensor diagrams and Chebyshev polynomials }
\author[L. Lamberti]{Lisa Lamberti}
\address{Mathematical Institute \\
University of Oxford \\
Oxford\\
OX2 6GG\\
United Kingdom 
}
\email{Lisa.Lamberti@maths.ox.ac.uk}

\begin{abstract}
In this paper, we describe a class of elements in the
ring of $\mathrm{SL}(V)$-invariant polynomial functions on the space of configurations of
vectors and linear forms of a 3-dimensional vector space $V.$ These elements are 
related to one another by an induction formula using
Chebyshev polynomials. 
We also investigate the relation between these polynomials and G.\ Lusztig's dual
canonical basis in tensor products of representations of $U_q(\mathfrak{sl}_3(\mathbb C)).$
\end{abstract}
\maketitle
\tableofcontents

\addtocontents{toc}{\setcounter{tocdepth}{1}}%

\let\thefootnote\relax\footnotetext{This project was supported by the Early Mobility Research Grant sponsored by the Swiss National Science Foundation (grant P2EZP$2_1$48747) and the Mary Ewart Junior Research Fellowship from Somerville College, University of Oxford.}

\section{Introduction} \label{sect:intro}

Given a complex vector space $V,$ consider the ring~$R_{a,b}(V)=\mathbb C[(V^\ast)^a\times V^b]^{\mathrm{SL}(V)}$ of 
polynomial functions on the space of configurations of $a$ vectors and $b$
covectors, which are invariant under the natural action of~$\mathrm{SL}(V).$
Rings of this type play a central role in representation theory, and their study
dates back to D.\ Hilbert. Over the last three decades,
bases of~$R_{a,b}(V)$ with remarkable properties were found
by G.\ Lusztig \cite{MR1180036,MR1227098} and 
further studied by many researchers, including G.\ Kuperberg 
\cite{Kuperberg}, B.\ Fontaine, J.\ Kamnitzer and G.\ Kuperberg 
\cite{fontaine_kamnitzer_kuperberg_2013} and 
M.\ Gross, P.\ Hacking, S.\ Keel and M.\ Kontsevich \cite{2014arXiv1411.1394G}.
To explicitly construct, as well as to compare, some of these bases remains 
a challenging problem, already open when $V$ is 3-dimensional. 
In the latter case, new perspectives in the study of canonical bases for $R_{a,b}(V)$
were suggested by S.\ Fomin and P.\ Pylyavskyy \cite{FP} who 
established that~$R_{a,b}(V)$ has (several) cluster algebra structures. 
These structures provide an approach to describing
canonical bases in~$R_{a,b}(V)$ by comparison with
other cluster algebras for which dual canonical bases are defined.
With this approach it follows, for example, that the set of cluster monomials forms the dual canonical
basis of~$R_{a,b}(V)$ when~$a=0$ and $b\leq8$
and (conjecturally) a subset of the latter for all
other values of~$a$ and~$b.$\\

The main goal of this paper is 
to describe and study~$\mathrm{SL}(V)$-invariants in~$R_{a,b}(V)$ which 
can be defined naturally by two different families of bivariated
Chebyshev polynomials. These invariants are 
of interest for the following reasons. 
We know of other algebras for which Lusztig's dual canonical bases are defined, and where basis elements can be defined by recursions of this type, 
see the contributions of B.\ Leclerc, P.\ Lampe, A.\ Berenstein and A.\ Zelevinsky, Ding and  Xu in \cite{Leclerc, MR2817684, MR3180605, MR2916330}. 
Moreover, the findings of B.\ Fontaine, J.\ Kamnitzer and G.\ Kuperberg, see \cite[Thm.\ 5.16]{fontaine_kamnitzer_kuperberg_2013},
suggest that there could be connections between one family of bivariated Chebyshev polynomials and Lusztig's semi-canonical basis \cite{MR1758244}.
In addition, we give a graphical descriptions of 
these recursively defined~$\mathrm{SL}(V)$-invariants in 
the language of {\em tensor diagrams}, 
or spiders, developed for rank 2 Lie algebras by G.\ Kuperberg in \cite{Kuperberg}.
In the setting of $R_{a,b}(V)$
tensor diagrams are finite bipartite graphs with a fixed 
proper coloring of their vertices in two colors, black and white, and 
with a fixed partition into boundary and internal vertices. 
Each internal vertex is trivalent and comes with a 
fixed cyclic order of the edges incident to it, see Definition \ref{tensordiag}.
If the boundary of $D$ consists of $a$ white and $b$ black vertices, one 
says that $D$ has type $(a,b).$ The coloring of the boundary of $D$ also determines a 
binary cyclic word, called the {\em signature of $D$} and denoted by $\sigma(D).$ For a 
fixed signature $\sigma,$ all invariants of $R_{\sigma}(V)\cong R_{a,b}(V)$, $a,b\in\mathbb Z_{\geq 0}$, 
can be described by tensor diagrams of type $(a,b)$ and signature $\sigma.$ Moreover, all tensor diagrams define 
$\mathrm{SL(V)}$-invariants. For example, the tensor diagrams illustrated in Figure ~\ref{tensor1} 
can be interpreted as the Weyl generators of $R_{a,b}(V)$. There is a 
linear basis of $R_{a,b}(V)$ spanned by invariants 
defined by {\em non-elliptic webs,} that is planar tensor diagrams whose internal faces have at least six sides. 
This basis was found in 1994 by G.\ Kuperberg in \cite{Kuperberg} and is known as Kuperberg's web basis.
The elements of this basis are called {\em web invariants}.\\

\begin{figure}[t]
      \psfragscanon
      \psfrag{c}[][]{$[W]^3$}
      \psfrag{t}[][]{$[\mathrm{thick}_3(W)]$}
      \psfrag{b}[][]{$[\mathrm{brac}_3(W)]$}
      \psfrag{e}[][]{$[\mathrm{band}_3(W)]$}
    \psfrag{sym}[][][0.80]{$\frac{1}{3\textrm{!}}\sum\limits_{\sigma\in \mathfrak{S}_3} \sigma$}
\includegraphics[width=14cm]{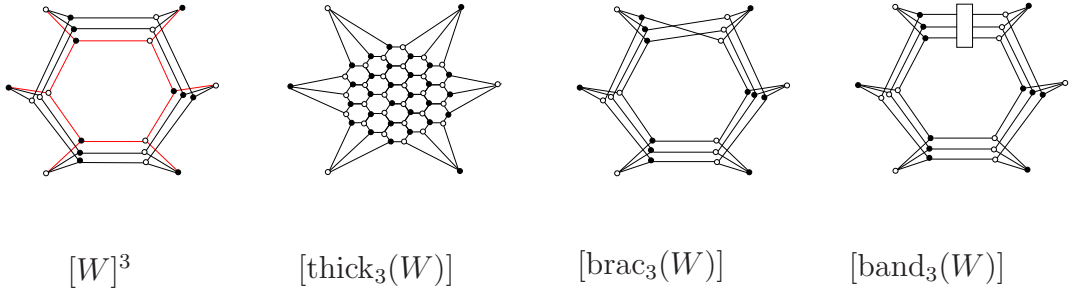}
\caption{Tensor diagrams defining
the invariants $[W]^3,$  $[\mathrm{thick}_3(W)],$ $[\mathrm{brac}_3(W)]$ and $[\mathrm{band}_3(W)]$ in $R_{\sigma(W)}(V)$ when $W$ is the non-elliptic web drawn in red in the left most diagram. }\label{twist}
\end{figure} 

Let us now turn to the main results of this paper. 
Let $[W_n]\in R_{\sigma(W_n)}(V)$ be a web invariant 
defined by an arbitrary non-elliptic web $W_n$
with a single internal face bounded by $n\geq 6$ sides.
In this paper we consider various types of operations on $W_n$. Let $k\in\mathbb Z_{\geq0}$.
The first operation is the {\em $k$-bracelet operation} defining the invariant 
$[\mathrm{brac}_k(W_n)]\in R_{\sigma(W_n)}(V)$ 
in the way shown in Figure ~\ref{twist} and more carefully described
in Definition \ref{de:bracelet}. 
The second operation is the {\em $k$-band operation} defining the 
invariant $[\mathrm{band}_k(W_n)]\in R_{\sigma(W_n)}(V)$ 
described in Definition \ref{def band} and obtained as the last tensor diagram 
in Figure ~\ref{twist}.
Let $[b(W_n)]=\frac{1}{2}([W_n]^2-[\mathrm{brac}_2(W_n)])\in R_{\sigma(W_n)}(V)$.
\begin{customthm}{3.19}\label{thmC1}
The $k$-bracelet operation of $W_n$ transforms the invariants as follows:
$$
[\mathrm{brac}_k(W_n)]=T_k([W_n],[b(W_n)])
$$
where $T_k(x,y)$ is the rescaled Chebyshev polynomial of 
the first kind.
\end{customthm}
\begin{customthm}{3.28}\label{thmC2}
The $k$-band operation of $W_n$ transforms the invariants as follows:
$$
[\mathrm{ band}_k(W_n)]=U_k([W_n],[b(W_n)])
$$
where $U_k(x,y)$ is the rescaled Chebyshev polynomial of 
the second kind.
\end{customthm}
Later in the text, these results are presented in a slightly more general form. The Chebyshev polynomials $(T_{k})_{k\in\mathbb Z_{\geq0}}$ 
and $(U_{k})_{k\in\mathbb Z_{\geq0}}$
are as in Definition \ref{deC1} and Definition \ref{ceC1}.
Another significant aspect of the band and bracelet operations is that they can be seen
as combinatorial deformations of the power transformation $[W_n]\mapsto[W_n]^k\in R_{\sigma(W_n)}(V)$ 
represented by superimposing $k$-copies of~$W_n$, see the example on the left most diagram in Figure ~\ref{twist}.
By superposition we mean that one repeats all vertices and edges of the $k$ copies of $W_n,$ except for the boundary 
vertices of $W_n$ which are not repeated. In Theorem~\ref{thm10.10} we show that
$[W_n]^k\in R_{\sigma(W_n)}(V)$ can equivalently be described by the web invariant $[\mathrm{thick}_k(W_n)]$
obtained by the {\em $k$-thickening operation} of $W_n$ given in Definition \ref{thickening}.
An example of the 3-thickening operation on a minimal  is given
in the second tensor diagram in Figure ~\ref{twist}.
This operation was introduced and studied by S.\ Fomin and P.\ Pylavskyy in \cite{FP}. Moreover,
the transformations we give in Theorem \ref{thmC1} and Theorem \ref{thmC2}
are potentially significant since they 
are natural adaptations to 
$\mathrm{SL(V)}$-invariant rings of the
topological characterizations of Chebyshev polynomials 
in Riemann surface cluster algebras as established by 
G.\ Musiker, R.\ Schiffler and L.\ Williams in \cite{MSW} and in
skein algebras as given by D.\ Thurston in \cite{MSW}. See also work of G.\ Dupont \cite{MR2912471} 
and references therein.\\

In the second part of this paper, we aim to compare
the $\mathrm{SL}(V)$-invariants described as in Theorem ~\ref{thmC1} and Theorem ~\ref{thmC2}  with elements of the dual of Lusztig's canonical basis.
More precisely, we consider the invariant space 
$\mathrm{Inv}(V^\sigma)=\mathrm{Hom}_{U_q(\mathfrak{sl}_3(\mathbb C))}(V^\sigma,\mathbb C(q^{1/2})$,
where  $V^\sigma$ is an arbitrary tensor product of the 
natural representation (of type +1) of~$U_q(\mathfrak{sl}_3(\mathbb C))$ and its dual. 
Varying the signature $\sigma$ one obtains multiple spaces of this type which specialize to the multilinear components of $R_{a,b}(V)$ at the classical limit at $q=1$, see Section \ref{subsec:backto} for more details.
Moreover, Kuperberg's web basis extends to a 
linear basis for each space $\mathrm{Inv}(V^\sigma)$ and any tensor product $V^\sigma$, see \cite{Kuperberg}. 
Hence, the $U_q(\mathfrak{sl}_3(\mathbb C))$-invariant spaces and 
the ring of $SL(V)$-invariants $R_{a,b}(V)$ share the 
diagrammatic representation of their elements.
In addition, Lusztig's dual canonical basis for the $U_q(\mathfrak{sl}_3(\mathbb C))$-invariant spaces $\mathrm{Inv}(V^\sigma)$
was studied in 1991 by M.\ Khovanov and G.\ Kuperberg in \cite{KK}.
In the same paper the authors show that Lusztig's dual canonical basis 
and Kuperberg's basis of $\mathrm{Inv}(V^\sigma)$ agree for 
all tensor products $V^{\sigma}$ up to 12 factors but are
different in general.

To understand this difference consider the two non-elliptic
webs on the right of Figure ~\ref{cheb2:intro}.
This figure can be interpreted as the invariant 
$[\mathrm{Thick}_2(W)]-[B(W)]\in \mathrm{Inv}(V^{\sigma(\mathrm{Thick}_2(W))})$, where
$V^{\sigma(\mathrm{Thick}_2(W))}$ indicates a certain tensor product of
12 factors given by copies of the natural representation and its dual.
The order of these factors is determined by the arguments of $[\mathrm{Thick}_2(W)]$.
Khovanov and Kuperberg show that this invariant defines  
an element in Lusztig's dual canonical basis for $\mathrm{Inv}(V^{\sigma(\mathrm{Thick}_2(W))})$, which we denote by $\ell(\mathrm{Thick}_2(W))$.
Since $\ell(\mathrm{Thick}_2(W))$ decomposes as a difference of two 
web-invariants, it follows that Kuperberg's basis and Lusztig's dual canonical basis are different.

This surprising conclusion can be related to the first part
of the paper by observing that the invariant
$\ell(\mathrm{Thick}_2(W))$ can 
be obtained from the band operation described above. 
To see this, let $W$ be the non-elliptic web consisting of a 
hexagon with an edge attached to each of its six vertices, 
drawn in red in Figure \ref{twist}. 
The 2-band operation on $W$ produces the tensor diagram
$\mathrm{band}_2(W)$. The corresponding invariant decomposes as $[\mathrm{thick}_2(W)]-[b(W)]$ in Kuperberg's basis for $R_{\sigma(W)(V)}$, where $\mathrm{thick}_2(W)$ and $b(W)$ are the non-elliptic webs on the left in Figure~\ref{cheb2:intro}. 
The non-elliptic webs on the right of Figure~\ref{cheb2:intro}
are then obtained from $\mathrm{thick}_2(W)$ and $b(W)$ by
separating all edges connected to a shared boundary point 
and repeating the boundary points. We denote the non-elliptic webs
obtained in this way by $\mathrm{Thick}_2(W)$ and $B(W).$

\begin{figure}[ht]
\centering
\psfragscanon
\psfrag{-}{$-$}
\psfrag{B_1}{$\ell(\mathrm{Thick}_2(W))=([\mathrm{Thick}_2(W)]-[B(W)]$}
\psfrag{B_2}{$[\mathrm{band}_2(W)]=[\mathrm{thick}_2(W)]-[b(W)]$}
\includegraphics[height=3cm]{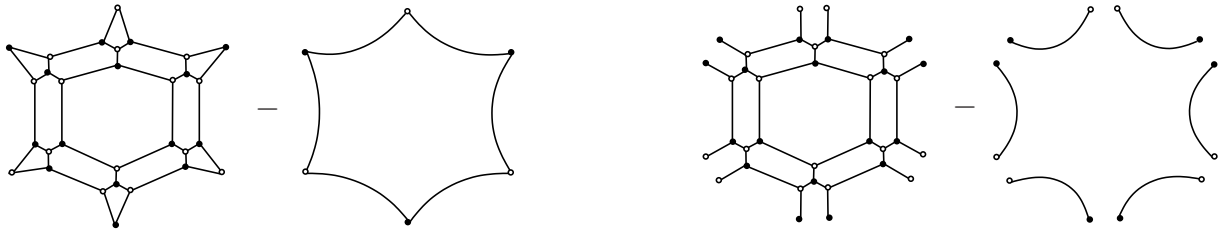}
\caption{The invariants $[\mathrm{band}_2(W)]$ ( on the left) and
$\ell(\mathrm{Thick}_2(W))$ (on the right) expressed in Kuperberg's web basis.}\label{cheb2:intro}
\end{figure}

In recent years, there has been an increasing amount of literature 
suggesting that this relation between the invariants $\ell(\mathrm{Thick}_2(W))$ 
and $[\mathrm{band}_2(W)]$ generalizes and that 
other dual canonical basis elements can be 
obtained by the band operation.
A first result in this direction is given in Theorem \ref{band2cd:2} below.
To state the result, let $W_n$ be as above. Let $\mathrm{Thick}_k(W)), k\in\mathbb Z_{\geq 0},$ be the non-elliptic web obtained by separating all edges 
connected at a same boundary point in the $k$-thickening of $W_n$ and repeating the boundary points.
Let $[\mathrm{Thick}_k(W))]\mathrm{Inv}(V^{\sigma(\mathrm{Thick}_k(W_n))})$ be the corresponding invariant. We then show:
\begin{customthm}{6.3}\label{band2cd:2}
\label{band2cd}
The invariant $[\mathrm{Thick}_2(W_n)]-[B(W_n)]$
belongs to Lusztig's dual canonical basis for $\mathrm{Inv}(V^{\sigma(\mathrm{Thick}_2(W_n))}).$
\end{customthm}

Also in this case, the difference $[\mathrm{Thick}_2(W_n)]-[B(W_n)]$ is obtained from the 2-band operation on $W_n$. We then explore how the above examples 
generalize to higher values of $k$ and ask if the 5-band operation on the hexagonal non-elliptic web $W$ produces a dual canonical basis element in a corresponding
$U_q(\mathfrak{sl}_3(\mathbb C))$-invariant space.
We then find an unexpected discrepancy as we explain in the next result:
\begin{customcor}{6.10}\label{contra}
Assume $\ell(\mathrm{Thick}_5(W))\in\mathrm{Inv}(V^{\sigma(\mathrm{Thick}_5(W))})$ has 
integer coefficients when expanded in Kuperberg's web basis.
Then, the invariants  
$[\mathrm{band}_3(W)]$ and
$[\mathrm{band}_5(W)]$
of $R_{\sigma(W)}(V)$ 
can not simultaneously belong to the specialization of Lusztig's dual canonical basis at $q=1$.
\end{customcor}

Summing up, in this paper we have identified two interesting families of invariants which can naturally be described with two recursions. To test the properties of these recursions and prove the above results we combine various recent developments in the theory of Lusztig's canonical bases and provide new considerations which we believe are relevant for further studies aimed at comparing the different bases for $R_{a,b}(V)$ and quantizations thereof.

\subsection{Acknowledgments} I would like to thank S.\ Fomin and P.\ Pylyavskyy 
for introducing me to this beautiful subject during the workshop on 
cluster algebras at MSRI, Berkeley, in  2012. I thank both of them 
for the always interesting and stimulating discussions we had. 
Among others, I wish to thank G.\ Kuperberg, P.\ Lampe, L.-H.\ Robert and D.\ Tubbenhauer  
for clarifying email conversations. I am also indebted to the anonymous referees for carefully reading this work and providing many helpful suggestions.
I also thank the Department of Mathematics of the University of Michigan  
for their hospitality and support during my research stay there.

\section{Preliminaries on the $\mathrm{SL(V)}$-invariant space $R_{\sigma}(V)$}\label{sect2}
We begin this section by introducing the ring of invariant $\mathrm{SL(V)}$-polynomial functions.
We then explain their relation to multi homogeneous invariant functions and multilinear symmetric tensors.
After this, we recall the important graphical characterization of the elements of the space $R_{a,b}(V)$ in terms of tensor diagrams, also called tensor networks in various other areas of science. This presentation of the invariants in $R_{a,b}(V)$ will enable us to define combinatorial operations on tensor diagrams and  establish links to other well understood rings. At the end of this section we introduce two families of Chebyshev polynomials and summarize
some of their properties.

The main references for this section are G. Kuperberg's work \cite{Kuperberg}, as well as
S.\ Fomin and P.\ Pylavskyy's contribution \cite[\S.~4]{FP} and the book of R.\ Goodman and N.W.\ Wallach \cite[\S.~5]{MR2522486}.


\subsection{$\mathrm{SL(V)}$-invariant polynomials}
For $V=\mathbb C^3$ and $V^\ast=\mathrm{Hom}_{\mathbb C}(V,\mathbb C)$ 
let $X=(V^\ast)^a\times V^b$ be the direct product of
$a$-copies of~$V^\ast$ and $b$-copies of~$V,$ for $a,b\in\mathbb Z_{\geq0}.$
The special linear group $\mathrm{SL(V)}$ acts on $V$ by left multiplication
and it induces a left action on $V^*$ given by
$
g\cdot u^\ast(v)=u^\ast(g^{-1}\cdot v)
$
for $g\in \mathrm{SL(V)},$ $v\in V$ and $u^\ast\in V^\ast.$
This action of $\mathrm{SL(V)}$ extends to an action on $X$ by defining
$$
g\cdot (\textbf{u}^\ast,\textbf{v})=
(g\cdot u ^\ast_1,\dots,g\cdot u^\ast_a,g\cdot v_1,\dots,g\cdot v_b)
$$
for $g\in\mathrm{SL(V)}$ and
$(\textbf{u}^\ast,\textbf{v})=(u^\ast_1,\dots,u^\ast_a,v_1,\dots,v_b)\in X.$

To describe $\mathrm{SL(V)}$-invariant polynomial functions on $X$
consider the standard basis for $V$ and  $V^\ast$ which determines
a basis for $X.$ A function $f:X\rightarrow \mathbb C$ is called {\em polynomial}
if it is given by a polynomial
in the basis coordinates of $X.$ The definition is
independent of the choice of a basis. 
Let $\mathbb{C}[X]$ be the {\em coordinate ring of $X$}, that is the $\mathbb C$-algebra
of polynomial functions on $X$ with $a$ covector arguments (elements of $V^\ast$)
and $b$ vector arguments (elements of $V$).

Let $\mathbb{C}[X]$ be the {\em coordinate ring of $X$}, 
that is the $\mathbb C$-algebra of polynomial functions on $X$ with $a$ covector arguments (elements of $V^\ast$)
and $b$ vector arguments (elements of $V$).
On 
$\mathbb{C}[X],$ the action of $\mathrm{SL(V)}$
is given by
$$
(g\cdot f)(\textbf{u}^\ast , \textbf{v})=
f( g^{-1} \cdot(\textbf{u}^\ast ,\textbf{v}))
$$
for $f\in \mathbb{C}[X], g\in \mathrm{SL(V)}.$

\begin{de}
A function $f\in \mathbb C[X]$ is called $\mathrm{SL(V)}$-invariant
if $f(g(\textbf{u}^\ast , \textbf{v}))=f(\textbf{u}^\ast , \textbf{v})$ for all
$g \in \mathrm{SL}(V)$ and $(\textbf{u}^\ast , \textbf{v})\in X.$
Let $R_{a,b}(V)=\mathbb C[X]^{\mathrm{SL(V)}}$ be the subalgebra of $\mathbb C[X]$
of $\mathrm{SL(V)}$-invariant polynomial functions on $X.$ 
\end{de}
The {\em signature} of an $\mathrm{SL(V)}$-invariant polynomial function $f$ is a word in
the alphabet $\{\circ,\bullet\}$ reflecting the order of the $a+b$ arguments
of $f.$

Being $R_{a,b}(V)$ a $\mathbb C$-algebra implies that $R_{a,b}(V)$ 
has a $\mathbb C$-vector space structure compatible with the ring multiplication of $\mathbb C[X].$ 

Main examples of elements of $R_{a,b}(V)$ are:
Pl\"ucker coordinates, given by $\mathrm{det}(v_i,v_j,v_k)$ 
for any tuple $(v_i,v_j,v_k)$ of vectors in $X$;
dual Pl\"ucker coordinates, given by $\mathrm{det}(u^\ast_i,u^\ast_j,u^\ast_k)$
for any tuple $(u^\ast_i,u^\ast_j,u^\ast_k)$ of covectors in $X$; and bilinear
pairings $ \langle u^\ast,v\rangle$ of covectors with vectors of $X.$
\begin{thm}[First Fundamental Theorem for $\mathrm{SL(V)}$] \
The ring $R_{a,b}(V)$ is generated by $\langle u^\ast,v\rangle$ and $\mathrm{det}(v_i,v_j,v_k)$ and
$\mathrm{det}(u^\ast_i,u^\ast_j,u^\ast_k).$
\end{thm}
These three classes of generators of $R_{a,b}(V)$ are called {\em Weyl generators}.


\subsection{Multi homogeneous invariants}

Let $X=(V^\ast)^a\times V^b$ be as before.
A polynomial function $f\in\mathbb C[X]$ is multi homogeneous of degree 
$[\textbf{p},\textbf{q}]=(p_1,\dots,p_a,q_1,\dots,q_{b})\in\mathbb Z_{\geq0}^a\times \mathbb Z_{\geq0}^b$ 
if for all 
$(t_1,\dots,t_{a+b}) \in(\mathbb C^{\times})^{a}\times(\mathbb C^{\times})^{b}$ we have
$$
f(t_1u^\ast_1,\dots,t_au^\ast_a,t_{a+1}v_1,\dots,t_{a+b}v_b)=t_1^{p_1}t_2^{p_2}\dots t_{a+b}^{q_{b}}(u^\ast_1,\dots,u^\ast_a,v_1,\dots,v_b).
$$
Clearly, a multi homogeneous function is multilinear if and only if
it is of degree $(1,1,\dots, 1).$

Every polynomial function $f$ in $\mathbb C[X]$ decomposes in a unique 
way into a sum of multi homogeneous functions:
$f=\sum f_{[\textbf{p},\textbf{q}]},$ where $f_{[\textbf{p},\textbf{q}]}$ are called multi homogeneous components of $f.$ Letting
$
\mathbb C [X]_{[\textbf{p},\textbf{q}]}=\{f\in\mathbb C[X]: f \textrm{ has multi homogeneous degree } [\textbf{p},\textbf{q}]\}
$
gives rise to the decomposition
\begin{align*}
\mathbb C [X]
=\displaystyle\bigoplus_{\textbf{p}\in\mathbb Z_{\geq0}^{a}} \displaystyle\bigoplus_{\textbf{q}\in\mathbb Z_{\geq0}^{b}} \mathbb C [X]_{[\textbf{p},\textbf{q}]}.
\end{align*}
If $f$ is $\mathrm{SL}(V)$-invariant then every multi homogeneous component
of $f$ is also $\mathrm{SL}(V)$-invariant and one has
\begin{align}\label{dec}
R_{a,b}(V)&=\mathbb C [X]^{\mathrm{SL}(V)}=\displaystyle\bigoplus_{\textbf{p}\in\mathbb Z_{\geq0}^{a}} \displaystyle\bigoplus_{\textbf{q}\in\mathbb Z_{\geq0}^{b}} \mathbb C [X]^{\mathrm{SL}(V)}_{[\textbf{p},\textbf{q}]}.
\end{align}


\subsection{Multilinear invariants}\label{Invariant}

In the following, we present a second realization of the ring $R_{a,b}(V)$ in terms of symmetric tensors. 
This result  is important since it will later enable us to link $R_{a,b}(V)$ to the space of $U_q(\mathfrak{sl}_3(\mathbb C))$-invariants.

For $\textbf{p}\in\mathbb Z_{\geq0}^{a}$ and $\textbf{q}\in\mathbb Z_{\geq0}^{b}$ consider $({V^\ast})^{\otimes \textbf{p}}\otimes V^{\otimes \textbf{q}}$
with the $\mathrm{SL}(V)$ action given by left multiplication with $g$ on each factor $V$ and by left multiplication with
$g^{-1}$ on each factor $V^\ast,$ as before. For $|\textbf{p}|=\sum p_i$ and $|\textbf{q}|=\sum q_i$
let $\mathfrak{S}_{\textbf{p}}=\mathfrak{S}_{p_1}\times \mathfrak{S}_{p_2}\times \dots \times \mathfrak{S}_{p_a}$ be acting
as a group of permutations of $\{1,\dots, |\textbf{p}|\},$ where the symmetric group $\mathfrak{S}_{p_1}$ permutes
the factors in position $1$ up to $p_1,$ $\mathfrak{S}_{p_2}$ permutes
$p_1+1,\dots, p_1+p_2,$ and so on. Then $\mathfrak{S}_{\textbf{p}}\times \mathfrak{S}_{\textbf{q}}$ acts
on $({V^\ast})^{\otimes \textbf{p}}\otimes V^{\otimes \textbf{q}}$ and 
the actions of $\mathfrak{S}_{\textbf{p}}\times \mathfrak{S}_{\textbf{q}}$ and $\mathrm{SL}(V)$ commute. 

In view of the following result, let
$S^{\textbf{p}}(V^{\ast})=S^{p_1}(V^{\ast})\otimes\dots\otimes S^{p_a}(V^{\ast})$ and
$S^{\textbf{q}}(V)=S^{q_1}(V)\otimes\dots\otimes S^{q_b}(V),$ where
$S^k(V),$ resp.~$S^k(V^\ast)$ is the space of symmetric $k$-tensors over $V,$ resp.\ over $V^\ast.$

\begin{lemma}\cite[Lemma 5.4.1]{MR2522486}\label{decop}
Let $\textbf{p}\in\mathbb Z_{\geq0}^{a}$ and $\textbf{q}\in\mathbb Z_{\geq0}^{b}.$ There is a linear isomorphisms:
\begin{align*}
R_{a,b}(V)&\cong\displaystyle\bigoplus_{\textbf{p}\in\mathbb Z_{\geq0}^{a}} \displaystyle\bigoplus_{\textbf{q}\in\mathbb Z_{\geq0}^{b}} 
\Big[\big((V^{\ast})^{\otimes |\textbf{p}|}\otimes V^{\otimes |\textbf{q}|}\big)^{\mathfrak{S}_{\textbf{p}}\times \mathfrak{S}_{\textbf{q}}}\Big]^{\mathrm{SL}(V)}\\
&\cong\displaystyle\bigoplus_{\textbf{p}\in\mathbb Z_{\geq0}^{a}} \displaystyle\bigoplus_{\textbf{q}\in\mathbb Z_{\geq0}^{b}} 
\Big[S^{\textbf{p}}(V^{\ast})\otimes S^{\textbf{q}}(V)\Big]^{\mathrm{SL}(V)}.
\end{align*}
\end{lemma}
The space $R_{a,b}$ can also be seen as a quotient space of 
$
\big((V^\ast)^{\otimes |\textbf{p}|}\otimes V^{\otimes |\textbf{q}|}\big)^{\mathrm{SL}(V)}
$
under the quotient homomorphism:
$$
(V^{\ast})^{\otimes |\textbf{p}|}\otimes V^{\otimes |\textbf{q}|}\twoheadrightarrow \big((V^\ast)^{\otimes |\textbf{p}|}\otimes V^{\otimes |\textbf{q}|}\big)^{\mathfrak{S}_{\textbf{p}}
\times \mathfrak{S}_{\textbf{q}}}.
$$ 

These two presentations of $R_{a,b}(V)$
will later be implicitly assumed when we recall the 
graphical presentation of invariants in $R_{a,b}(V).$


\subsection{Tensor diagrams}

Let $id: V^\ast\times V\rightarrow \mathbb C$ 
be the {\em identity tensor} given by
$(u^\ast,v)\mapsto \langle u^\ast,v\rangle$
for all $u^\ast\in V^\ast$ and $v\in V.$
Let $\mathrm{vol}\in\Lambda^3V^\ast$ be a basis vector of the 
one-dimensional vector space $\Lambda^3V^\ast.$ Let
$\mathrm{vol}^\ast\in\Lambda^3V$ be
defined by $$\mathrm{vol}(v_1,v_2,v_3)\mathrm{vol}^\ast(u^\ast_1,u^\ast_2,u^\ast_3)=
\mathrm{det}(\langle u_j^\ast, v_i \rangle)_{\substack{1\leq i \leq 3 \\1\leq j\leq 3}}.$$

Every element in $R_{a,b}(V)$ 
can be obtained from $\mathrm{vol},$ $\mathrm{vol}^\ast$
and $id$ by the operations of tensor product
and contraction and taking linear combinations with scalar coefficients,
see \cite{Kuperberg, FP}.

To obtain a diagrammatic representation of
these invariant tensors, one associates 
to $\mathrm{vol},$ $\mathrm{vol}^\ast$
and $id$ certain bicolored planar graphs
as represented in Figure ~\ref{tensor1}.
Each such graph is embedded inside a disc (not drawn) passing through the 
univalent vertices. The latter represent the arguments of the corresponding tensor.
To distinguish the type of arguments, we indicate vector arguments by $\bullet,$ 
and covectors arguments by $\circ.$ The trivalent vertices are colored with the opposite
color so that the result is indeed a bicolored graph. 
To ensure that these diagrams represent well defined tensors, one also
specifies the clockwise ordering of the three edges meeting
at the internal trivalent vertex. 

\begin{figure}
\includegraphics[scale=0.5]{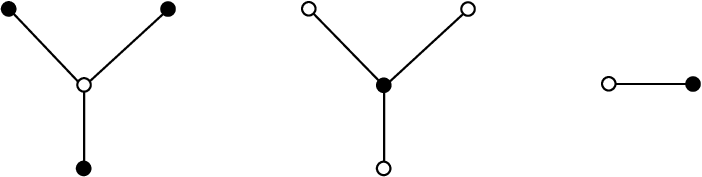}
\centering
\caption{Tensor diagrams associated with the three types of Weyl generators of $R_{a,b}(V).$  }\label{tensor1}
\end{figure}

For multilinear invariants, the operation of tensor product is 
represented by juxtaposing the corresponding tensor diagrams (thus ensuring that the result is again multilinear). 
Contraction of two tensors with respect to a pair of arguments 
of opposite valence is represented by connecting two adjacent boundary points 
with opposite color by an arc (and removing the boundary points of the joint
edges). Contraction with the tensor diagram representing the identity tensor
does not change the tensor one produces. Hence
tensor diagrams don't depend on the number or locations
of the pieces corresponding to the identity tensor.

Iterating these operations, bigger tensor 
diagrams representing multilinear tensors can be obtained.  
To preserve well definiteness one imposes that the 
cyclic ordering of the edges incident to each interior vertex 
matches the clockwise orientation of the disc.

\begin{de}\label{tensordiag}
A {\em tensor diagram $D$} is a finite bipartite planar graph 
drawn in a disc with a fixed 
proper coloring of its vertices into two colors, black and white, and 
with a fixed partition of its vertex set into boundary and internal vertices of $D,$ 
satisfying the following conditions:
\begin{itemize}
\item Each internal vertex is trivalent;
\item for each internal vertex, a cyclic order on the edges incident to it is fixed.
\end{itemize}
\end{de}
Unless specified otherwise, tensor diagrams will always be
considered up to isotopy fixing the boundary of the disc. 

In the following, we say that a tensor diagram $D$ is {\em planar} if the
edges of $D$ don't cross, otherwise we say that $D$ is {\em non-planar}.
Therefore, tensor diagrams are not necessarily 
embeddings of graphs into a disc.
For example, in Figure ~\ref{twist} all but the second are non-planar tensor diagrams, in the sense of the above definition.
The boundary points of a tensor diagram indicate the arguments of the corresponding
tensor, and as such they determine its signature.
If the boundary of 
$D$ consists of $a$ white vertices and $b$ black ones, then we say that
$D$ has type $(a,b).$


\subsection{Clasping and unclasping}

In the previous section, we gave a diagrammatic description of 
multilinear invariants of $R_{a,b}(V).$ But not all invariants of $R_{a,b}(V)$ are of this type.
To extend this graphical description  to all invariants of $R_{a,b}(V)$ we also have
to consider multi homogeneous invariants.
To do so we now recall the {\em restitution}, 
resp.~{\em partial restitution}, operators defined by H.\ Kraft and C.\ Procesi in \cite[\S.4.6]{MR1044584}.
Let $d=(d_1,d_2,\dots, d_r)\in\mathbb Z_{\geq0}^r$ and $\vert d\vert=\sum _i^r d_i.$
The restitution of a multilinear invariant 
$F(v_1,\dots,v_{|d|})$ is the multi homogeneous 
invariant $\mathcal{R}F(v_1,\dots,v_r)$ 
of degree $(d_1,d_2,\dots,d_r)$ obtained 
from $F$ by setting
$$
\mathcal{R}F(v_1,\dots,v_r):=F(\underbrace{v_1,\dots, v_1}_{d_1},\underbrace{v_2,\dots, v_2}_{d_2},\dots,\underbrace{v_r,\dots, v_r}_{d_r})
$$
where $v_i$ are vector or covector arguments of $F$ and $1\leq  i\leq \vert d \vert.$

Let $\sigma_{F}=(s_1,\dots, s_{|d|})$ be the signature of $F.$ 
Then the signature of $\mathcal{R}F,$ denoted by $\sigma_{\mathcal{R}F},$  is obtained from
$\sigma_F$ by replacing the $i$-th substring $(s_i,\dots, s_i)$ of lengths $d_i$ with $s_i,$ so that
$$
\sigma_{\mathcal{R}F}=(s_1,\dots,s_r)=(\underbrace{s_1,\dots,s_1}_{d_1},\underbrace{s_2,\dots, s_2}_{d_2},\dots,\underbrace{s_r,\dots, s_r}_{d_r}).
$$
Restitution operation with respect to some variables $v_i$ of $F,$ 
leaving the other variables unchanged will be called partial restitution.

On the level of tensor diagrams, restitution is described as follows.
Let $D$ be the tensor diagram representing a multilinear invariant $F$ with signature $\sigma_F.$ This means in particular that the
boundary points of $D$ are connected to a single edge.
Then $\mathcal{R}F$ is represented by the tensor diagram with signature $\sigma_{\mathcal{R}F}$ obtained from $D$
by joining all edges formerly incident to the single vertex $s_i.$
In this way, the multi degree of the $i$-th argument of the invariant $\mathcal{R}F$ is equal
to the degree of the boundary vertex $s_i,$ i.e.\ the number of edges that connect to it. 
In the following, we call the (partial) restitution operation on invariants or on the corresponding tensor diagrams simply {\em clasping}. 

\begin{rem}
The restitution presented above, are special cases of Kuperberg's 
clasping operations defined in \cite{Kuperberg}.  
\end{rem}

The {\em full unclasping} of a tensor diagram $D$ with boundary vertices of a given multi degree $d=(d_1,d_2,\dots, d_r),$ 
is the tensor diagram obtained from $D$ by replacing
each boundary vertex of color $s_i$ and of degree $d_i,$ by $d_i$ distinct successive and equally colored vertices serving as endpoints of the edges formerly incident to $s_i.$ 
In the unclasping, we also assume that no additional edge crossings will be created.
There are in general many  different multilinear invariants
clasping to the same multi homogeneous invariant.


\subsection{Diagrammatic description of the multiplication in $R_{a,b}(V)$}\label{multiplication}

The calculus of tensor diagrams yields a diagrammatic
interpretation of the addition and multiplication operations of $\mathrm{SL(V)}$-polynomial functions in $R_{a,b}(V).$
Addition is obtained formally by allowing linear combinations of
tensor diagrams. Multiplication in $R_{a,b}(V)$ can be modeled using 
{\em superimposition} of diagrams defined in the following way: if $D$ is the union of 
sub-diagrams $D_1,D_2,\dots$ connected only at the boundary vertices,
then $[D]=[D_1][D_2]\dots.$ Notice that the invariant $[D]$
does not depend on the vertical order in which the superimposition of $D_1, D_2, \dots$ is made, since the ring $R_{a,b}(V)$ is commutative. On the left of Figure ~\ref{twist} we illustrate the multiplication $[W]^3$ in $R_{a,b}(V).$


\subsection{Skein relations for tensor diagrams}

The diagrammatic description of invariants in $R_{a,b}(V)$ given above is
not unique. 
To obtain uniqueness one has to consider a number of local relations satisfied by invariants in $R_{a,b}(V)$ and
illustrated on the tensor diagrams in Figure ~\ref{skein}. For example, relation $(f)$ in Figure ~\ref{skein} might be
interpreted as the trivial invariant $\mathrm{det}(v,v,v)=\mathrm{det}(u^\ast,u^\ast,u^\ast)=0$.
Relation $(e)$ represents a trivial invariant with a symmetric and antisymmetric pair of arguments, but since we consider the case of symmetric tensors, we let corresponding invariant be the trivial invariant. 

These relations, called {\em skein relations}, allow one to transform a small fragment
$F$ of the diagram $D$ into linear combinations of other diagrams $F=c_iF_i, c_i \in \mathbb Z_{\geq0},$
where the $F_i$  are tensor diagrams of the same type as $F.$
Whenever $F$ 
defines the same invariant as $c_iF_i,$ one has that
$[D]=\sum_ic_i[D_i], $ where $D_i$ indicates the tensor diagram obtained from $D$
by replacing the fragment $F$ with the other pieces $F_i$ and keeping the rest of the diagram unchanged.

Quantum versions of relations (a), (c) and (d) were first given by 
G.\ Kuperberg \cite{doi:10.1142/S0129167X94000048}, the remaining ones
were introduced by S.\ Fomin and P.\ Pylyavskyy \cite{FP}. 

\begin{figure}[ht]
\psfragscanon
\psfrag{+}{$+$}
\psfrag{-}{$-$}
\psfrag{=}{$=$}
\psfrag{z}{-2}
\psfrag{x}{3}
\psfrag{v}{0}
\psfrag{a}{(a)}
\psfrag{b}{(b)}
\psfrag{c}{(c)}
\psfrag{d}{(d)}
\psfrag{e}{(e)}
\psfrag{f}{(f)}
\includegraphics[scale=0.5]{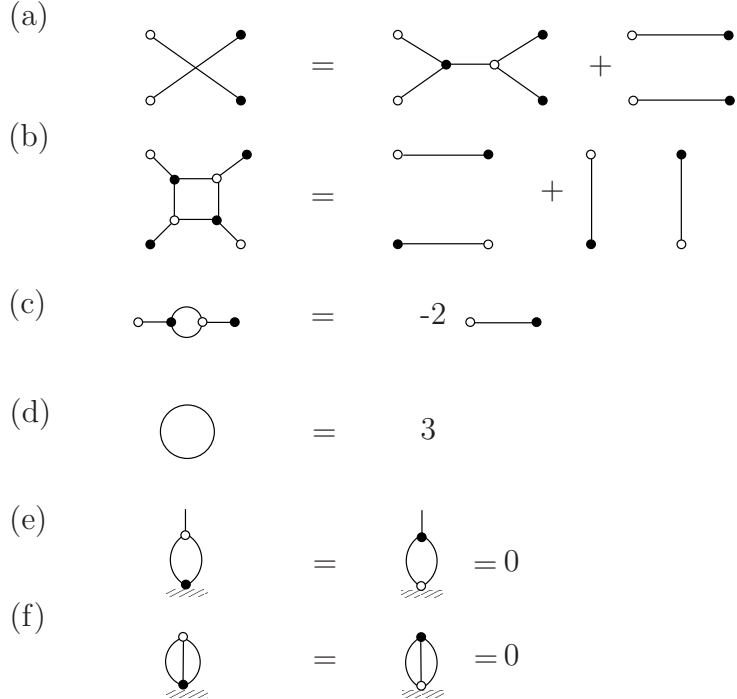}
\centering
\caption{Skein relations.
The relations (e) and (f) involve a vertex lying 
on the boundary, here represented with dashed lines.}\label{skein}\label{Skeinbound}
\end{figure}


\subsection{Web invariants}
In the following definition, we say that a face of a tensor diagram is {\em internal}
if the vertices of the edges bounding the face are all internal vertices, see Definition \ref{tensordiag}.
\begin{de}
A {\em web} $W$ is a tensor diagram embedded in an oriented disc so that its
edges do not cross or touch each other, except at endpoints. 
Each web is considered up to isotopy of the disc that fixes its boundary.
A web is {\em non-elliptic} if it has 
no internal face bounded by two or four edges.
The invariant $[W]\in R_{a,b}(V)$ associated with a non-elliptic 
web $W$ is called a {\em web invariant}.
\end{de}
A web with an internal face bounded by two or four edges is called {\em elliptic}.
The {\em signature} of a web $W,$ denoted by $\sigma(W),$ is defined as 
the word in the alphabet $\{\circ, \bullet\}$
obtained from the boundary of $W.$

\begin{thm} [Kuperberg \cite{doi:10.1142/S0129167X94000048}]\label{webbasis}
Web invariants with a fixed signature $\sigma$ of type $(a,b)$
form a $\mathbb C$-linear basis in the ring of invariants $R_\sigma(V)\cong R_{a,b}(V).$
\end{thm} 
The basis from Theorem ~\ref{webbasis}
is called the {\em web basis} for $R_{a,b}(V).$ From 
Theorem ~\ref{webbasis} it follows that skein relations are 
consistent local relations; hence it is irrelevant
in which order one applies them. 
From Theorem 5.3 in \cite{KK}
one might deduce that Kuperberg's web basis is in fact a $\mathbb Z$-basis
for $R_{a,b}(V).$

To conclude this section let us point out that $R_{a,b}(V)$ is a finitely generated ring,
and the dimension of the vector space
$R_{a,b}(V)$ is infinite for any choice of $a$ and 
$b.$ In general, the Weyl generators represent only a small 
subset of all possible basis vectors.
Nevertheless, every multi homogeneous component of $R_{a,b}(V)$ has a
finite dimensional web basis. All linearly independent web invariants 
spanning a six dimensional multi homogeneous component of $R_{a,b}(V)$ 
are represented in Example 5.3 in \cite{FP}.


\subsection{Two families of Chebyshev polynomials}\label{chebdef12}

In this section, we recall some basic facts about the Chebyshev recursions which become useful when we characterize invariants in $R_{a,b}(V)$ satisfying these recursions.

\begin{de}\label{deC1}
Let $k \in\mathbb Z_{\geq 0}.$
The rescaled 2-variables Chebyshev polynomial of 
the first kind $T_k(x,y)$ is defined by the recurrence
\begin{align*}
T_0(x,y)&=2,\\
T_1(x,y)&=x,\\
T_k(x,y)&=xT_{k-1}(x,y)-yT_{k-2}(x,y).
\end{align*}
\end{de}
\begin{de}\label{ceC1}
Let $k \in\mathbb Z_{\geq 0}.$
The rescaled 2-variables Chebyshev polynomial of 
the second kind $U_k(x,y)$ is defined by the recurrence
\begin{align*}
U_0(x,y)&=1,\\
U_1(x,y)&=x,\\
U_k(x,y)&=xT_{k-1}(x,y)-yT_{k-2}(x,y).
\end{align*}
\end{de}
In the following, we refer to the families $(T_{k})_{k\in\mathbb Z_{\geq0}}$ 
and $(U_{k})_{k\in\mathbb Z_{\geq0}}$
as Chebyshev polynomials.

\begin{prop}\cite[Prop.~2.35]{MSW}\label{monomialCheby}
For all $k\geq1,$  the following identities hold,
\begin{align*}
x^k&=T_k+\binom{k}{1}yT_{k-2}+\dots+\binom{k}{\frac{k-1}{2}}y^{\frac{k-1}{2}}T_1, \textrm{ if $k$ is odd;}\\
x^k&=T_k+\binom{k}{1}yT_{k-2}+\dots+\binom{k}{\frac{k-2}{2}} y^{\frac{k-2}{2}}T_2+\binom{k}{\frac{k}{2}}y^{\frac{k}{2}}T_0,
\textrm{ if $k$ is even.}
\end{align*}
In particular, $x^k$ can be written as a positive integer 
linear combination of the Chebyshev
polynomials of the first type $(T_{k})_{k\in\mathbb Z_{\geq0}}.$
\end{prop}
In a similar way, we deduce the following result for Chebyshev polynomials of the second kind.
\begin{prop}\label{monomialCheby2}
For all $k$ the following identities hold,
$$ U_k+\Bigg\{ \binom{k}{1}-\binom{k}{0}\Bigg\}yU_{k-2}+\dots+\Bigg\{ \binom{k}{\frac{k-1}{2}}-\binom{k}{\frac{k-1}{2}-1}\Bigg\}y^{\frac{k-1}{2}}U_1$$
if $k$ is odd; 
$$ U_k+\Bigg\{ \binom{k}{1}-\binom{k}{0}\Bigg\}yU_{k-2}+\dots+\Bigg\{ \binom{k}{\frac{k}{2}-1}-\binom{k}{\frac{k}{2}-2}\Bigg\}y^{\frac{k-2}{2}}U_2+
\Bigg\{ \binom{k}{\frac{k}{2}}-\binom{k}{\frac{k}{2}-1}\Bigg\}y^{\frac{k}{2}}U_0
$$
if $k$ is even. In particular, $x^k$ can be written as a positive integer
linear combination of the Chebyshev
polynomials of the second type $(U_{k})_{k\in\mathbb Z_{\geq0}}.$
 \qed
\end{prop}

\begin{table}[h]
\begin{align*}
\begin{split}
T_0(x,y)&=2 \\
T_1(x,y)&=x       \\
T_2(x,y)&=x^2-2y  \\
T_3(x,y)&=x^3-3xy \\
T_4(x,y)&=x^4-4x^2y+2y^2  \\
T_5(x,y)&=x^5-5x^3y+5xy^2 \\
\end{split}
\quad\quad
\begin{split}
U_0(x,y)&=1  \\
U_1(x,y)&=x  \\
U_2(x,y)&=x^2-y \\
U_3(x,y)&=x^3-2xy  \\
U_4(x,y)&=x^4-3x^2y+y^2   \\
U_5(x,y)&=x^5-4x^3y+3xy^2 \\
\end{split}
\end{align*}
\caption{The first five Chebyshev polynomials $(T_{k})_{k\in\mathbb Z_{\geq0}}$  and $(U_{k})_{k\in\mathbb Z_{\geq0}}$ in the two variables $x$ and $y.$}
\label{chebtable}
\end{table}

\begin{rem}
The usual Chebyshev polynomials of the first kind $\mathrm{Cheb}_k(x)$
defined by 
$$\mathrm{Cheb}_k(\cos x) = \cos(kx), \textrm{ for } k\in\mathbb Z_{\geq 0}.$$ 
relate to $(T_{k})_{k\in\mathbb Z_{\geq0}}$ by the transformation
$$T_k=2\mathrm{Cheb}_k(\frac{x}{2\sqrt{y}})y^{k/2}.$$
Moreover, the polynomials
$(U_{k})_{k\in\mathbb Z_{\geq0}}$ relate to $(T_{k})_{k\in\mathbb Z_{\geq0}}$ by 
$\frac{ d}{dx}T_k=kU_{k-1}.$ Hence, $U_k)$
relate to the usual Chebyshev polynomials by
$kU_{k-1}=\frac{d}{dx}\big( 2\mathrm{Cheb}_{k}(\frac{x}{2\sqrt{y}})y^{k/2}\big).$
\end{rem}



\section{$\mathrm{SL(V)}$-invariants and Chebyshev polynomials}\label{single-face-diag}
Our next goal is to explain how powers of certain web invariants 
can be written in the web basis of $R_{a,b}(V)$.
We then discuss how modifications of these power 
operations yield Chebyshev recursions.
The modifications we have in mind will be
described in terms of surgeries on tensor diagrams, 
see Sections~\ref{Bracelet and band operations}-~\ref{subsect:band}.
For each such surgery, we ask how the corresponding invariant decomposes in Kuperberg's web basis. The two decompositions we find involve a coefficient variable which we describe in Section~\ref{coeffsub}.\\

The main results of this section are Theorem~\ref{thm10.10} as well as Theorem~\ref{thmC1.2}  and Theorem~\ref{thmC2:2}.


\subsection{Arborization operation on single-face tensor diagrams}

Among the invariants with favorable properties for 
the power operation one finds
the ones defined by (not necessarily planar) 
tree diagrams, as well as the ones defined by
(non-necessarily planar) diagrams with a single internal 
face, see Definition~\ref{single-face} below. The invariants of the first
type were already investigate by Fomin-Pylyavskyy in \cite{FP},  while new results 
about the invariants of the second type will be given in Theorem~\ref{thm10.10}.

To describe these two sets of invariants we recall the operation of 
arborization, first introduced by S.\ Fomin and P.\ Pylyavskyy in \cite[Section 10]{FP}.
This is an algorithm which takes any tensor diagram as inputs and locally changes it by 
applying a sequence of skein transformations to it. These transformations don't 
change the value of the invariant in $R_{a,b}(V)$ but
yield a different presentation of it which facilitates computations, as we will see.

Below we reproduce Fomin-Pylyavskyy's description of the arborization algorithm.
\begin{de}
Let $D$ be a tensor diagram, $s_1$ and $s_2$ two of its internal vertices, and $e_1$
and $e_2$ two edges incident to $s_1$ and $s_2,$ respectively. We call vertices $s_1$ and 
$s_2$ {\em siblings} of each other (more precisely, ``siblings away from $e_1$ and $e_2$'') 
if the following happens. For $i\in\{1, 2\},$ let $B_i$ denote the subgraph of $D$ 
whose edge set consists of those edges which
can be reached from $s_i$ without going along $e_i$ or connecting through a boundary vertex.
(In particular, the edge $e_i$ is not in $B_i.$) We then want $B_1$ and $B_2$ to be isomorphic
binary trees having the same multi set of leaves on the boundary of the disc. Thus, $s_1$
and $s_2$ are siblings if they are obtained from the same multi set of boundary vertices by
the same sequence of taking pairwise joins.
\end{de}

\begin{de}\label{dearb}
Suppose that a tensor diagram $D$ contains a fragment which is:
\begin{itemize}
\item a quadrilateral with one vertex on the boundary, or
\item a four-edge path whose endpoints are siblings of each other, looking away from the edges of the path.
\end{itemize}
an {\em arborization step} is the transformation of such a diagram $D$ as shown in Figure ~\ref{arb}.
\end{de}
\begin{figure}[h]
\centering
\includegraphics[scale=0.3]{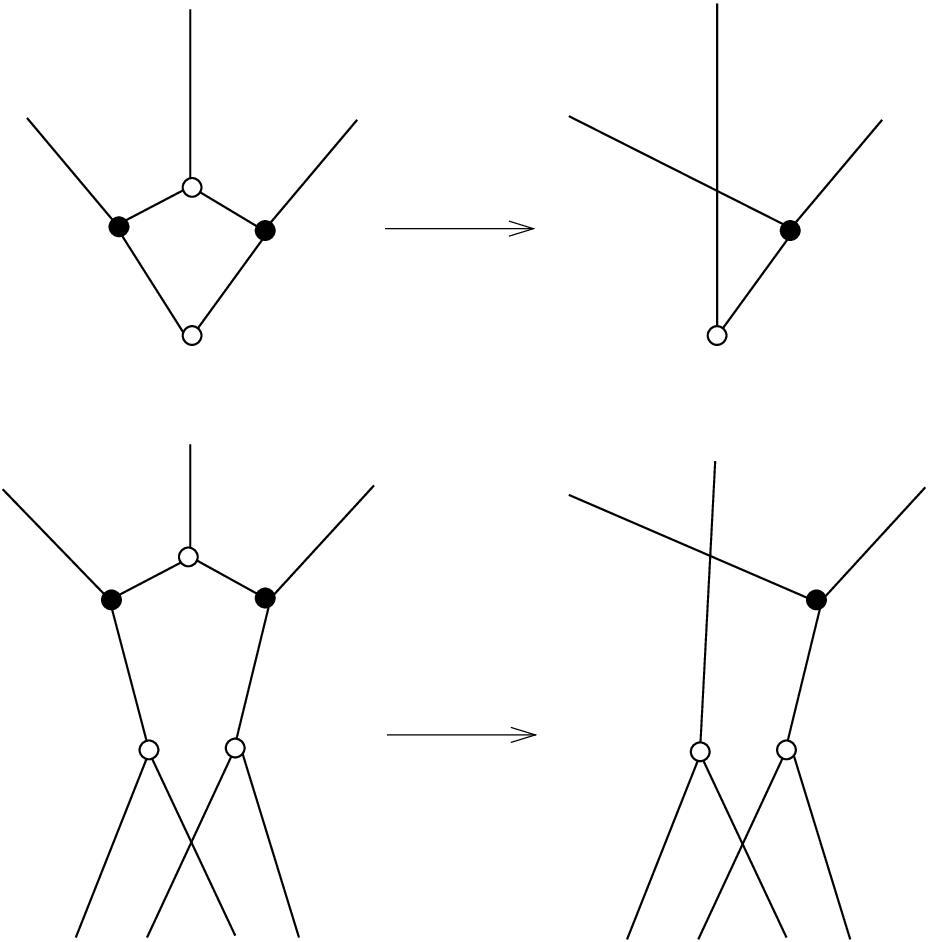}
\caption{Arborizing steps.}\label{arb}
\end{figure}
The {\em arborization algorithm} takes any tensor diagram as input and recursively applies the arborization steps. 
The algorithm stops if no more arborization steps can be performed. 
\begin{lemma}\cite[Lemma.~10.4]{FP}
An arborization step does not change the value of the invariant defined by the tensor diagram.
\end{lemma}
\begin{thm}\cite[Thm.~10.5]{FP}\label{confluent}
The arborization algorithm is confluent. That is, its output does not depend on 
the choice of an arborizing step made at each stage.
\end{thm}

We now turn our attention to two distinguished sets of outputs of the 
arborization algorithm.

\begin{de}\label{single-face}
A {\em single-face diagram} is a tensor diagram 
consisting of a single-face bounded by at least six edges. The vertices of this face
are all internal and possibly with intersecting
tree diagrams attached to it.
\end{de}
Single-face non-elliptic webs might have
further bounded faces passing through
boundary vertices (i.e.\ non-internal faces). 
But every single-face diagram $D$ has a unique
minimal (with respect to the number of vertices) single-face non-elliptic
web as a sub-diagram which we denote by $M=M(D).$ 

In Figure ~\ref{singlecycle}, we give three examples
of single-face non-elliptic webs. On the left, we represent the
minimal (with respect to the number of vertices) single-face non-elliptic. We will call this non-elliptic web the {\em hexagonal web}. 
The single-face non-elliptic web in the middle has the former as  sub-diagram $M$.
The non-elliptic web on the right has a non-internal face bounded by four edges. Of course,
all faces in the non-elliptic webs in the figure could be bounded by more edges. 

\begin{figure}[h]
\begin{center}
                \centering
                \includegraphics[width=13cm]{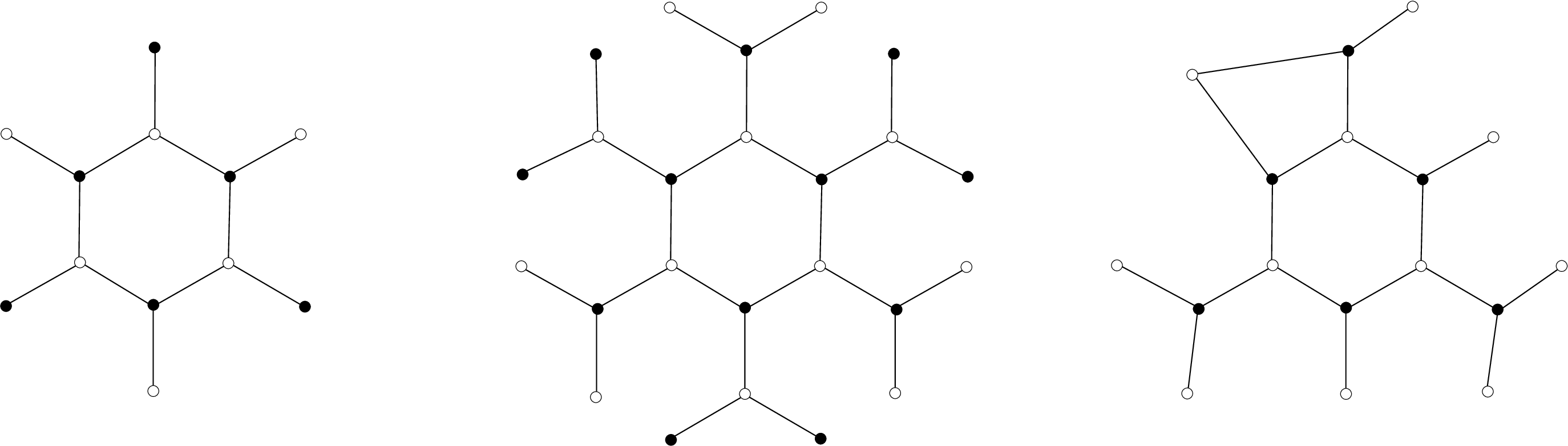}
                \caption{Examples of single-face non-elliptic webs. The arborization 
                algorithm is the identity the first two, while it outputs a
 tree diagram for the single-face web on the right.}\label{singlecycle}
\end{center}
\end{figure}

\begin{figure}[h]
                \centering
         \includegraphics[width=10cm]{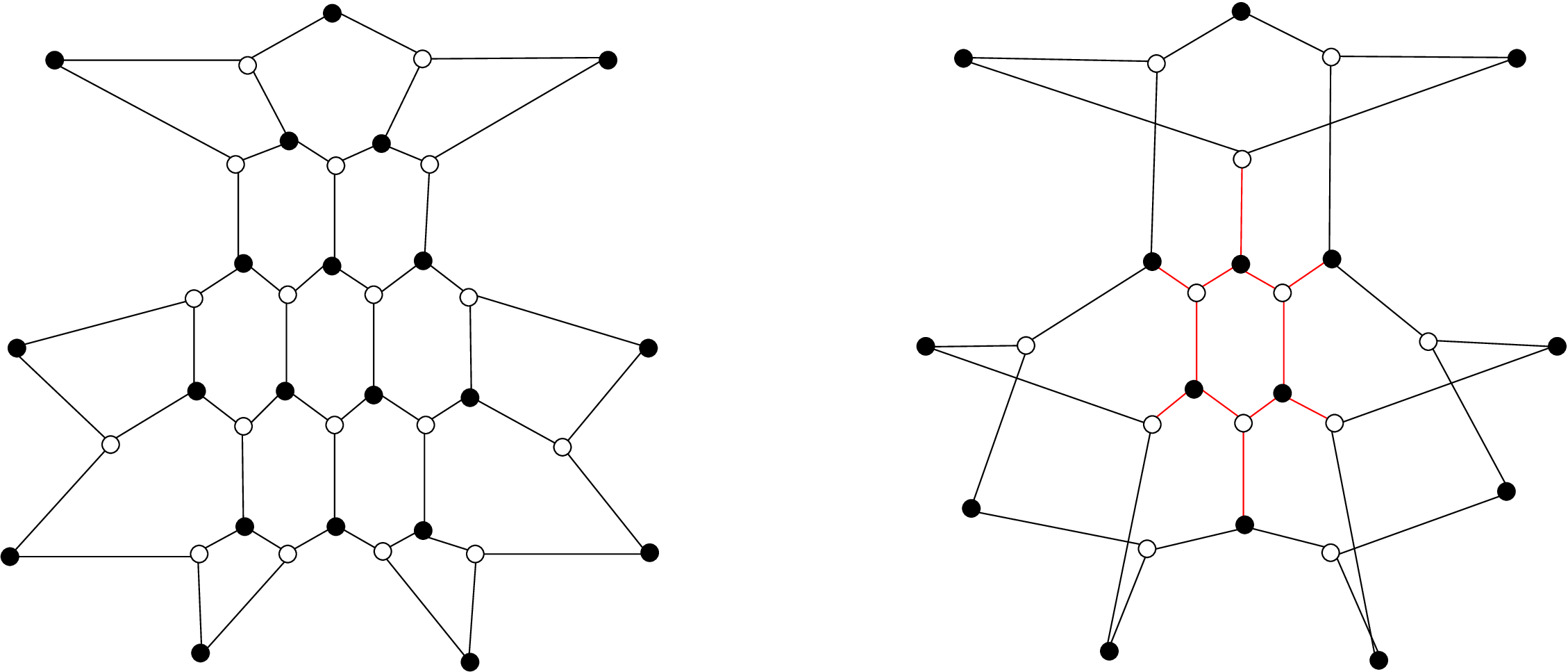}
         \caption{A non-elliptic web $W$ together with its
          arborized form as a single-face diagram $D.$ In red we 
          highlight the minimal single-face non-elliptic web $M$ inside $D.$}\label{arbex}
\end{figure}
Given these notions, we say that a non-elliptic web $W$ 
{\em arborizes to a tree diagram, or a single-face diagram} $D$
if the output of the arborization algorithm is a tree diagram, 
resp.~a single-face diagram. 
Let $\mathcal{T}$, resp.~$\mathcal{S}$, be the set of web invariants in
$R_{a,b}(V)$ defined by non-elliptic webs that arborize to tree, resp.~single-face diagrams.
From Theorem ~\ref{confluent} we deduce that these two sets are disjoint. 
In Figure ~\ref{arbex}, taken from \cite[Fig.\ 31]{FP}, we illustrate
a non-elliptic web (on the left) which arborizes to the single-face
diagram $D$ shown on the right.
\begin{rem}
By work of Chris Fraser \cite[Thm.8.10]{2017arXiv170200385F} we know that 
there are infinitely many web invariants of  $R_{3,9}(V)$  
belonging to the set $\mathcal{S}.$ 
On the other extreme, $\mathcal{S}$ is not empty if and only if $a\neq 0$ and $b \geq 9.$ 
Finally, although remarkably many web invariants in $R_{a,b}(V)$ belong to 
$\mathcal{T}$ or $\mathcal{S},$ not all invariants in
$R_{a,b}(V)$ can, in general, be described in this way. 
This was already observed in \cite[\S.11]{FP}.
\end{rem}


\subsection{Thickening of webs}

As mentioned above, a nice description of 
the expansion of powers $[T]^k,$ for  $[T]\in\mathcal{T}$ and $k\in\mathbb Z_{\geq0},$ in the web basis
was given by S.\ Fomin and P.\ Pylyavskyy \cite{FP}, see Theorem ~\ref{thm10.9} below.
The aim of this section is to show that the same description also applies
to powers of web invariants in $\mathcal{S}.$

The next definition is taken from \cite{FP}.
\begin{de}\label{thickening}
Let $k$ be a positive integer and $W$ a non-elliptic web.
The {\em $k$-thickening} of $W$ is obtained as follows:
\begin{itemize}
\item replace each internal vertex of $W$ by a ``honeycomb'' fragment $H_k$ of the appropriate
color,
as shown in Figure ~\ref{honey} (boundary vertices stay put);
\item replace each edge of $W$ by a $k$-tuple of edges 
connecting the corresponding honeycombs and / or boundary vertices.
\end{itemize}
\end{de}

\begin{figure}[h]
        \psfragscanon
\psfrag{1}{$H_1$}
\psfrag{2}{$H_2$}
\psfrag{3}{$H_3$}
\psfrag{4}{$H_4$}
\psfrag{5}{$H_5$}
                \includegraphics[width=10cm]{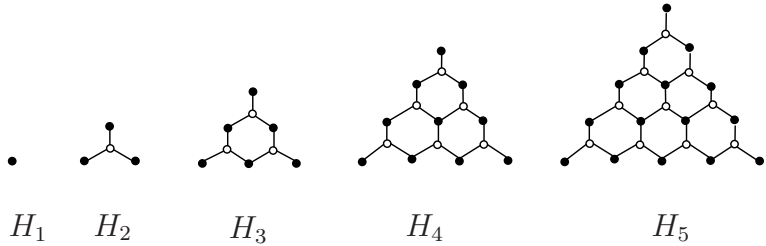}   
                \centering
        \caption{Honeycomb fragments for $k=1,2,3,4,5.$} \label{honey}
\end{figure}

In the sequel, we denote the $k$-thickening of $W$ by $\mathrm{thick}_k(W).$ It's
full unclasping will be denoted by $\mathrm{Thick}_k(W).$

\begin{thm}\cite[Thm.\ 10.9]{FP}\label{thm10.9}
Let $[W]$ be a web invariant in $\mathcal{T}.$
Then each power $[W]^k$ is a web invariant defined
by the $k$-thickening of $W.$
\end{thm}
In a similar fashion we show the following result.
\begin{thm}\label{thm10.10}
Let $[W]$ be a web invariant in $\mathcal{S}.$
Then each power $[W]^k$ is a web invariant defined
by the $k$-thickening of $W.$
\end{thm}
\begin{proof}
Assume $W$ arborizes to a single-face diagram $D.$
Let $M$ be the minimal single-face non-elliptic web of $D.$
To prove the result one can follow the same steps 
as in the proof of Theorem ~\ref{thm10.9}. That is:
one reverses the arborization process, starting with $D$ 
and step by step planarizing it. 
Notice that $M$ is planar by definition.
Hence, to planarize $D$ is the same as planarizing the
tree portions attached to $M$ patterned after some tree sub-diagrams of $D.$ 
To obtain $z^k$ one then
superimposes $k$ copies of $D,$ connected only at the boundary vertices. 
As before we denote this tensor diagram by $D^k.$ 
To planarize $D^k$ one follows the same steps as for $D$, 
together with planarizing copies of $M$. 
Since the order in which we use skein relations does not matter,
one can planarize $D^k$ moving from the $k$-copies
of internal vertices of the $k$-copies of $M$ towards the boundary.
Since this is the same as planarizing the $k$-tree portions attached 
to $k$-copies of $M$ the claim follows.

We illustrate how this works using the example in Figure ~\ref{kolk}.
The picture on the left shows the superimposition of
three copies of a fragment of a single-face diagram $D$ obtained through arborization.
The six white vertices at the bottom of the figure correspond to sibling vertices in $D.$ Boundary vertices are now drawn.
Tensor sub-diagram attached below those vertices are identical copies of subtrees attached to $D.$
In the second figure, copies of internal vertices
of $D$ and intersection points have been planarized.
At each stage of the planarization process, all but one 
term vanish, as in the proof of Theorem ~\ref{thm10.9}.
To clarify this point, notice that eventually the remaining terms have a square
bounded by to edges connecting to the same boundary vertex off the tensor diagram  $D^k.$
Hence the terms vanish by skein relation $(e).$
\end{proof}

From Theorem ~\ref{thm10.10}  and the definition of honeycomb fragments we deduce the following 
properties of web invariants in $\mathcal{T}\sqcup\mathcal{S}.$
\begin{cor}\label{corthick}
Let $[W]$ be a web invariant in $\mathcal{T}\sqcup\mathcal{S}$ defined by
the non-elliptic web $W.$
Then for all $k,l \in\mathbb Z_{\geq 0}$ the following identities hold,
$$
[W]^{kl}=[\mathrm{thick}_k(W)][\mathrm{thick}_l(W)]=[\mathrm{thick}_{kl}(W)]=[\mathrm{thick}_k(\mathrm{thick}_l(W))]
$$
in $R_{\sigma(W)}(V).$\qed
\end{cor}

It is tempting to believe that the proof of Theorem ~\ref{thm10.10} applies to all webs
that arborize to a diagram having a non-elliptic web as a sub-diagram.
But this is wrong, as one can see in  the example treated in Figure 33 in \cite{FP}.
Hence, Theorem ~\ref{thm10.9} and Theorem ~\ref{thm10.10} imply that the class of invariants in
$\mathcal{T}\sqcup\mathcal{S}$ are very distinguished from all other invariants.

\begin{figure}[h]
                \psfragscanon
                \psfrag{c}{$W^3$}
                \psfrag{g}{$\mathrm{thick}_3(W)$}
                \includegraphics[width=15cm]{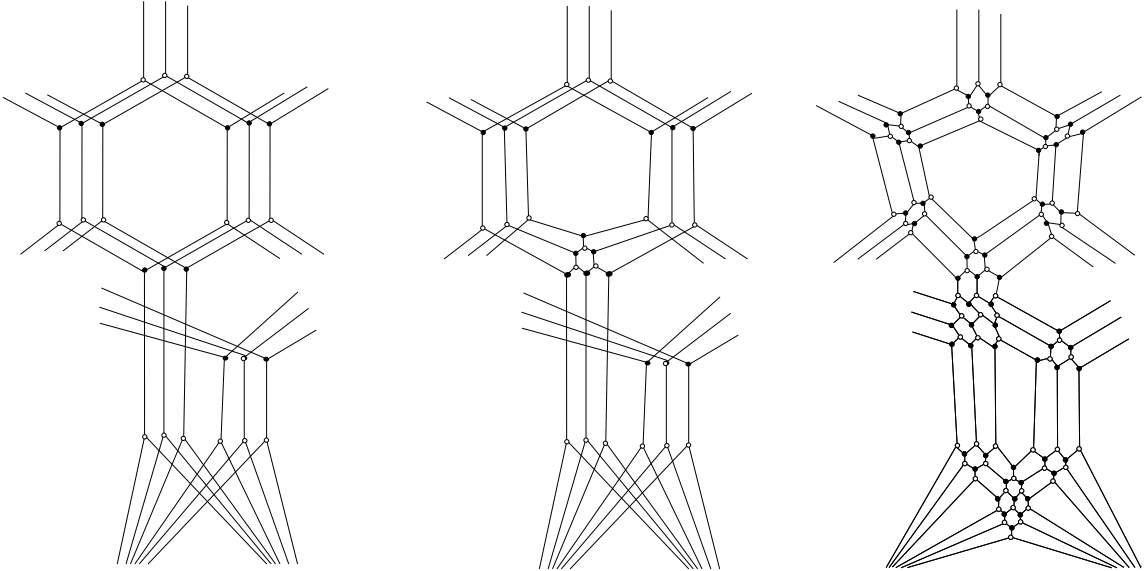}
                \centering
                \caption{Planarization of copies of a single-cycle diagram 
                outputted by the arborization algorithm.}\label{kolk}
\end{figure}


\subsection{A coefficient variable}\label{coeffsub}

Let $W$ be a non-elliptic web that arborizes to the single-face diagram $D.$
Let $M$ the minimal non-elliptic web inside $D.$

\begin{de} \label{part}
The non-elliptic web $b(W)$ is obtained from 
joining successive 
vertices serving as endpoints of $\mathrm{Thick}_2(M)$ inside
$\mathrm{thick}_2(W)$  and
then erasing $\mathrm{Thick}_2(M).$ 
We call $b(W)$ the {\em coefficient} of $W.$
\end{de}
The invariant defined by $b(W)$ will later play the role of 
a coefficient variable, see Theorem ~\ref{thmC1.2} and Theorem ~\ref{thmC2:2}.
The full unclasping of $[b(W)]$ is denoted by $[B(W)].$

In Figure ~\ref{coefffig} we illustrate Definition ~\ref{part} on two examples. The 
two coefficients are associated to the two single-face non-elliptic webs 
(not shown) represented on the left of Figure ~\ref{singlecycle}.
 
\begin{figure}[h]
\psfragscanon
\psfrag{=}{=}
 \includegraphics[width=10cm]{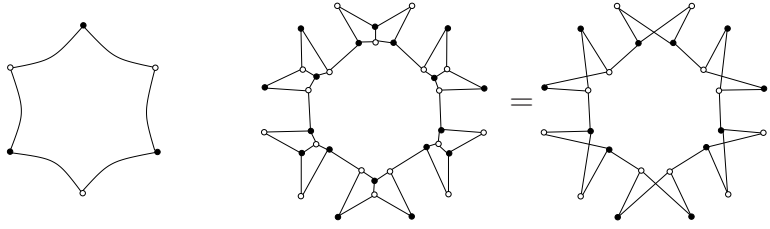}
 \centering
 \caption{Two coefficients $[b(W)]$.}\label{coefffig}
\end{figure}


\subsection{Surgeries on tensor diagrams}\label{Bracelet and band operations}

\begin{figure}[h]
      \psfragscanon
      \psfrag{t}[][]{$[\mathrm{brac}_3(W)]$}
      \psfrag{b}[][]{$[\mathrm{band}_3(W)]$}
\includegraphics[width=\textwidth]{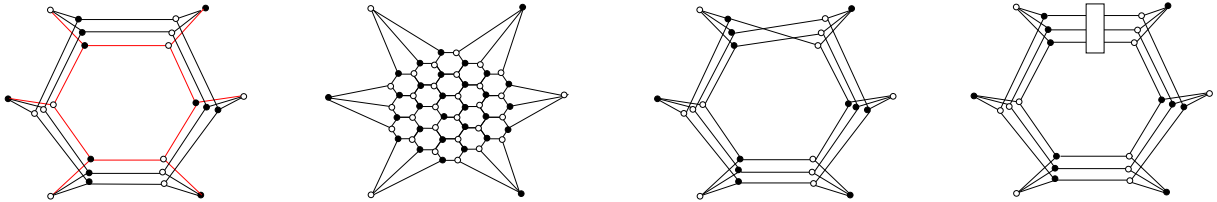}
\centering
\caption{Tensor diagrams defining
the invariants $[W]^3,$ $[\mathrm{thick}_3(W)],$ $[\mathrm{brac}_3(W)]$ and $[\mathrm{band}_3(W)]$ in $R_{\sigma(W)}(V)$
where $W$ is in red.}\label{twist2}
\end{figure} 

So far, we have focused on the properties
of powers of the two families of web invariants $\mathcal{T}\sqcup\mathcal{S}$. 
Our next goal is to introduce two new operations
obtained by distinguished local surgeries in the superimposition of copies 
of a single-face diagram $D$.

Throughout, let $k\in\mathbb Z_{\geq 0},$  let
$E$ be an edge bounding the internal face of $D.$
Let $D^k$ be the tensor diagram 
obtained by superimposing $k$-copies of $D.$ 

\begin{de}
A {\em surgery on $D^k$ at $E$} is an operation which simultaneously removes 
$k$ copies $E$ in $D^k$ and inserts another configuration of $k$-edges. The 
boundary points of E remain unchanged. 
\end{de}
Of course, surgeries  on $D^k$ could be defined at any edge, but for the applications
we have in mind, we restrict our self to the case where $E$ is always an edge bounding the 
single-face of $D.$ We then say that the surgery takes place 
in the thickening of the single-face of $D.$ The 
tensor diagram obtained after such surgery will again be
considered up to isotopy, fixing boundary points.
Moreover, a surgery in $D^{k}$ which inserts two disjoint fragments involving $l,$ resp.~ $k-l,$ $k\leq l$ copies of $E,$
can equally be seen, as a surgery inside the superimposition of $D^{l}$ with $D^{k-l},$ 
connecting only at the boundary vertices.


\subsection{Bracelet operations and Chebyshev polynomials of the first type} \label{subsect:brac}

\begin{figure}[h]
\begin{subfigure}[b]{0.3\textwidth}
                \psfragscanon
                \psfrag{i}{$\vdots$}
                  \centering
          \includegraphics[height=1.8cm,width=1.5cm]{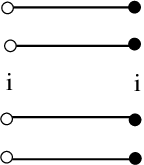}\\
                \caption{$k$- copies of an edge.}
\end{subfigure}
\begin{subfigure}[b]{0.3\textwidth}
        \psfragscanon
                \psfrag{i}{$\vdots$}
                  \centering
          \includegraphics[height=1.8cm,width=1.5cm]{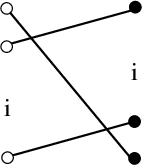}\\
                \caption{$(k-1)$-crossing.}
\end{subfigure}
\begin{subfigure}[b]{0.3\textwidth}
                \psfragscanon
                \psfrag{i}{$\vdots$}
                  \centering
          \includegraphics[height=1.8cm,width=1.5cm]{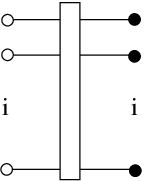}\\
                \caption{ $k$-band.}
\end{subfigure}
\caption{}\label{ins}
\end{figure}

\begin{de}\label{def crossing}
Let $s_1 \neq s_2\in\{\circ,\bullet\}, s_1\neq s_2.$ 
A {\em $(k-1)$-crossing} is a tensor diagram fragment
consisting of $k$ equally colored vertices, 
$(v_1,v_2,\dots,v_k)\in\{s_1\}^k,$ and $k$ equally colored vertices of the opposite color,
$(t_1,t_2,\dots,t_k)\in\{s_2\}^k,$
and edges connecting these vertices according to the following rule:
$v_i$ connects to $t_{i-1},$ $(\mathrm{mod} k)$ for all $1\leq i\leq k.$ 
\end{de}
In Figure ~\ref{ins}(B) an example of a $(k-1)$-crossing is provided.
Let us point out that in all the next
pictures, we always assume that
the illustrated tensor diagram fragments have 
an arbitrary number of pairs of
white and black vertices. The definitions and proofs are 
symmetric in the color of the vertices.

\begin{de}
The {\em $k$-bracelet operation} is the surgery on $D^k$ 
which replaces $k$-copies of $E$ with the $(k-1)$-crossing fragment 
whose boundary vertices are 
appropriately colored.
\end{de} 
The $k$-bracelet operation of $D$ will be denoted by $\mathrm{brac}_k(D).$
The full unclasping of $\mathrm{brac}_k(D)$ will be denoted by $\mathrm{Brac}_k(D).$ 
Then $\mathrm{brac}_1(D)=\mathrm{Brac}_k(D)=D$ and we set $\mathrm{brac}_0(D)=2.$
Clearly, the invariants $[\mathrm{brac}_k(D)]\in R_{\sigma(D)}(V),$ resp.~$[\mathrm{Thick}_k(D)]\in R_{\sigma(\mathrm{Brac}_k(D))}$  are 
independent on the choice of the edge $E$ bounding the single-face of $D$ and the order in which one
superimposes the $k$ copies of $D$ (with connected boundary vertices). 

\begin{de}\label{de:bracelet}
Let $W$ be a web that arborizes to a single-face diagram $D.$ Then
$$
\mathrm{brac}_k(W)=\mathrm{brac}_k(D).
$$
\end{de}
In Figure ~\ref{twist2} an example of the $3$-bracelet 
operation for the hexagonal web drawn on the
left of Figure ~\ref{singlecycle} is provided.

\begin{rem}
Definition ~\ref{de:bracelet} is an adaptation of 
the bracelet power defined in skein algebras by D.\ Thurston in 
\cite{Thurston} and in Riemann surfaces with boundary by G.\ Musiker, R.\ Schiffler and L.\ Williams in \cite{MSW}.
\end{rem}

\begin{lemma}\label{let}
The following local identities always hold:\\
\begin{itemize}
\item[(a)]\par
\begin{minipage}[t]{4cm}
 \end{minipage}

\begin{minipage}{\linewidth}
\centering
\psfragscanon
\psfrag{v}{\vdots}
\psfrag{=}{$=$}
\psfrag{+}{$+$}
\includegraphics[height=2cm]{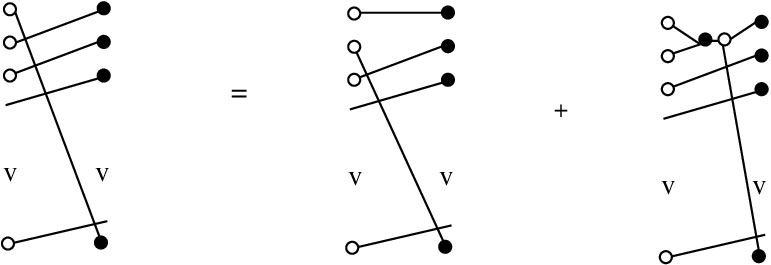}
\end{minipage}

\item[(b)]\par
\begin{minipage}[t]{4cm}
 \end{minipage}
 
\begin{minipage}{\linewidth}
\centering
\psfragscanon
\psfrag{v}{\vdots}
\psfrag{=}{$=$}
\psfrag{d}{\dots}
\psfrag{+}{$+$}
\includegraphics[height=2cm]{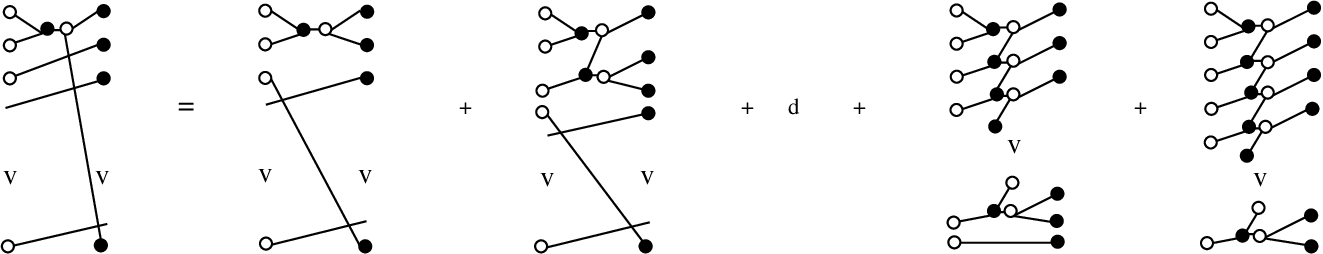}
\end{minipage}
\end{itemize}
\qed
\end{lemma}

To state the next results, we introduce a further
piece of notation. 
Consider the surgery of $D^k$ at $E,$ which removes $i$-parallel edges $E,$ for $2\leq i\leq k,$
and inserts the $(i-1)$-th summand on the right-hand side of relation $(b)$ in Lemma ~\ref{let}.
Let us denote the resulting tensor diagram by $B_{i}\overline{\mathrm{brac}}_{k-i}(W)$, where $B_{i}$
denotes the planar fragment with no crossing edges and $\overline{\mathrm{brac}}_{k-i}$ is the same transformation
as $\mathrm{brac}_{k-i}$, except that here $\overline{\mathrm{brac}}_{0}=1.$
For instance, $B_2$ is
obtained from $D^2$, by removing two parallel edges $E$ and inserting an 
$H$ tensor diagram piece (in red in Figure~\ref{s1}).
$B_3$ is obtained by inserting two $H$ pieces connected along one edge (in blue in Figure~\ref{s2}), and so on.

\begin{lemma}\label{let2}
Let $W$ be a web that arborizes to a single-face diagram $D.$
The following local identity holds:
\begin{itemize}
\item[(c)]
\par
\begin{minipage}[t]{4cm}
 \end{minipage}
\begin{minipage}{\linewidth}
\psfragscanon
\psfrag{v}{\vdots}
\psfrag{=}{$=$}
\psfrag{a}{$-2 b(W)$}
\psfrag{d}{\dots}
\includegraphics[height=2cm]{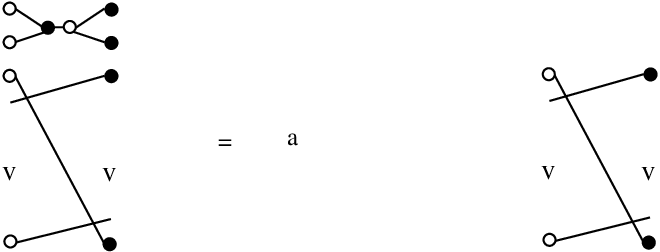}
\centering
\end{minipage}
\end{itemize}
when the surgery involves edges bounding the single-face of $D.$
\end{lemma}
\begin{proof} 
In the notation introduced before, relation $(c)$ reads as 
$$B_{2}\overline{\mathrm{brac}}_{k-2}(W)=-2 b(W)\overline{\mathrm{brac}}_{k-2}(W).$$
To prove this equality we show that $B_{2}=-2 b(W)$, and
argue with the example given in Figure ~\ref{s1}.

In Figure ~\ref{s1}, we represent a fragment of $B_2$ obtained from
$D^2$ by planarizing around the single-face of $D$. 
At each boundary vertex of the illustrated fragment $B_2$
there could be attached copies of tree tensor sub-diagrams of $D$, as in Figure \ref{kolk}. The subsequent figures represent intermediate
reduction steps obtained with skein relation $(b)$. 
At each reduction step, also a second non-elliptic arises (not drawn). 
This term always vanishes by the same argument used in the proof of Theorem ~\ref{thm10.10}.
Recall that by Theorem \ref{webbasis} we know that
the order in which one performs skein relations does not
change the result.

After having reduced all internal square faces one 
is left with the picture on the right in Figure~\ref{s1}.
Reducing with skein relation $(c)$ yields the coefficient $-2$ and the non-elliptic web 
$b(W)$ and proves the claim. 

It is easy to see that the above example
extends to single-face diagrams $D$ with an internal face bounded by more than six edges.
Moreover, one can argue in the same way in the presence of the term
$\overline{\mathrm{brac}}_{k-2}(W)$, since $B_{2}$ is then superimposed to $\overline{\mathrm{brac}}_{k-2}(W)$ and the order of superimposition of tensor diagrams does not matter, see Corollary~\ref{corthick}.

Finally, let us point out that a necessary condition to make use of skein relation $(c)$ in the last step of the proof is
is that the surgery of $D^2$ takes place at a copy of an edge bounding the single-face. 
\end{proof}

\begin{figure}[ht]
\psfragscanon
\psfrag{dots}{$\dots$}
\includegraphics[width=13cm]{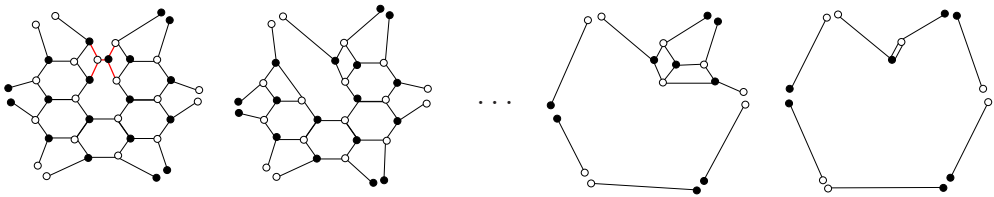}
\caption{Moving squares inside $D^2$.}\label{s1}
\end{figure} 

\begin{lemma}\label{let2d}
Let $W$ be a web that arborizes to a single-face diagram $D.$
The following local identities hold:
\begin{itemize}
\item[(d)]
\begin{minipage}[t]{4cm}
 \end{minipage} \par \par
\begin{minipage}{\linewidth}
\psfragscanon
\psfrag{v}{\vdots}
\psfrag{=}{$=$}
\psfrag{d}{\dots}
\psfrag{+}{$+$}
\psfrag{c}{$b(W)$}
\includegraphics[height=5cm]{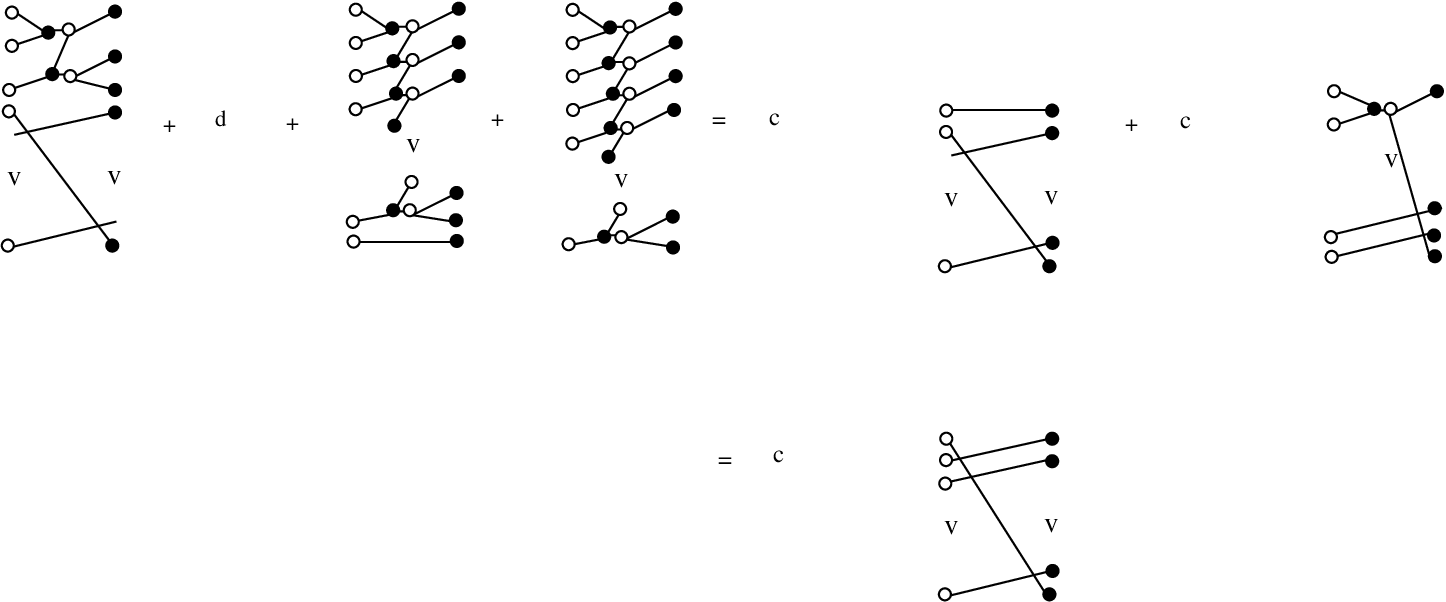}
\centering
\end{minipage}
\end{itemize}
when the surgery involves edges bounding the single-face of $D.$
\end{lemma}
\begin{proof}

To prove the first equality, we split the 
summands on the left and right of the equation
into two sets. The first sets consists of the first summand
$B_3\overline{\mathrm{brac}}_{k-3}(W)$ on the left and
the first summand on the right of the equality.
The second set contains the other terms.
We then first show that:
\begin{align}\label{eq:let2}
B_3\overline{\mathrm{brac}}_{k-3}(W)= b(W)W\overline{\mathrm{brac}}_{k-3}(W).
\end{align}
To prove this claim we use the example given in Figure~\ref{s2}.
On the left of it we represent a fragment of the tensor diagram $B_{3}$ associate to $D$ as described in the previous result. 
We then indicate intermediate
reduction steps obtained with skein relation $(b)$, in the subsequent figures. 
This then shows the equality in Equation \eqref{eq:let2}. 
As above, the order in which one performs skein relation $(b)$ does not matter. For more general single-face diagrams $D$ one then deduces the claim arguing as in Lemma~\ref{let2}.

Second, we consider the subsequent summands 
$B_{i}\overline{\mathrm{brac}}_{k-i}(W),$ $4\leq i\leq k$, in Equation $(d)$. 
Each non-elliptic web corresponding to $B_{i}$ in these summands has two square faces. 
Subsequently reducing the outer square face in each term
yields the following transformation:
$$
\sum_{4\leq i\leq k}B_{i}\overline{\mathrm{brac}}_{k-i}(W)
\mapsto
\sum_{4\leq i\leq k}b(W)B_{i-2}\overline{\mathrm{brac}}_{k-i-2}(W).$$
But 
$b(W)\sum_{2\leq j\leq k-2}B_{j}\overline{\mathrm{brac}}_{k-j}(W)$ 
is equal to the second summand, on the right of the first equality in Equation (d), involving $k-2$ univalent pairs of vertices,  by Lemma ~\ref{let} Part $(b)$. 
This thus proves the first equality in claim $(d)$.

The last equality in claim $(d)$ follows using Lemma ~\ref{let} Part $(a).$
\end{proof}
\begin{figure}[t]
\psfragscanon
\psfrag{dots}{$\dots$}
\includegraphics[width=\textwidth]{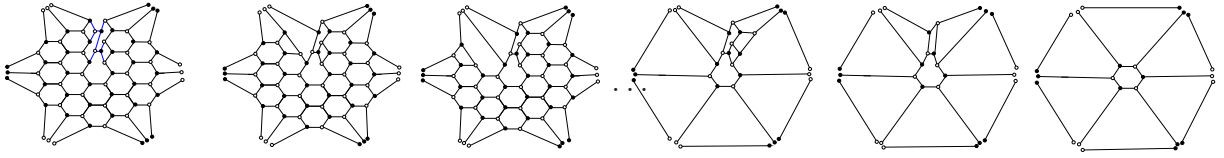}
\caption{Moving squares inside $D^3$.}\label{s2}
\end{figure} 

Given these results, Theorem \ref{thmC1.2} 
follows by an induction argument, together with Lemma ~\ref{let} part $(a)$ followed by part $(b)$ and followed by part $(c)$ and $(d)$ in Lemma  ~\ref{let2}.
\begin{thm}\label{thmC1.2}
Let $k\in\mathbb Z_{\geq 0}$ and $W$ be a non-elliptic web that arborizes to a single-face diagram. 
Then, the $k$-bracelet operation of $W$ transforms the invariants as follows:
$$
[\mathrm{brac}_k(W)]=T_k([W],[b(W)])
$$
where $(T_{k})_{k\in\mathbb Z_{\geq0}}$ is the rescaled Chebyshev polynomial of 
the first kind. \qed
\end{thm}
Keeping the assumptions of Thm. ~\ref{thmC1.2} the next result  
can be  deduced by Prop. ~\ref{monomialCheby}.
\begin{cor}\label{corcheb1}
Let $k\in\mathbb Z_{>0}.$ The following identities hold, 
\begin{align*}
[W]^k&=T_k+\binom{k}{1}[b(W)] T_{k-2}+\dots+\binom{k}{(k-1)/2}[b(W)]^{(k-1)/2}T_1,\textrm{ if $k$ is odd;}\\
[W]^k&=T_k+\binom{k}{1}[b(W)] T_{k-2}+\dots+\binom{k}{(k-2)/2}[b(W)]^{(k-2)/2}T_2+\binom{k}{k/2}[b(W)]^{k/2}
\end{align*}
if $k$ is even. 
In particular, $[W]^k$ can be written as a positive linear 
combination of the Chebyshev polynomials of the first kind
$(T_{k})_{k\in\mathbb Z_{\geq0}}$. 
\qed
\end{cor}

We wish to finish this section by pointing out that the results in this section were partially discussed in the author's thesis, see \cite{Lamberti}.

\subsection{Band operation and Chebyshev polynomials of the second type}\label{subsect:band}
Earlier we saw how inserting a 
crossing fragment inside $D^k$
affects the decomposition of the corresponding invariant
in the web basis.
Instead of considering one specific additional
crossing of $k$ edges in $D^k,$
one might consider all possible ways 
of crossings $k$ edges. This 
gives rise to the second transformation of $D^k,$ 
which we will now describe.

\begin{de}
Let $s_1\neq s_2\in\{\circ,\bullet\}.$
A {\em $k$-band} is an average of a formal linear combination of 
$k\textrm{!}$ tensor diagram fragments. Each summand represents a 
way of connecting 
$k$ equally colored vertices, 
$(v_1,v_2,\dots,v_k)\in\{s_1\}^k,$ with 
$k$ equally colored vertices of the opposite color,
$(t_1,t_2,\dots,t_k)\in\{s_2\}^k.$ 
\end{de}
We indicate the $k$-band by a box over $k$-edges, as drawn in Figure ~\ref{ins}(C) and
label the box by a number if we want to
specify the number of edges passing through the box.

\begin{de}
The {\em $k$-band operation} is the surgery of $D^k$ which replaces
$k$ copies of $E$ with a $k$-band fragment of the appropriate color.
\end{de}
The $k$-band operation of $D$ is denoted by $\mathrm{band}_k(D).$ 
The full unclasping of $\mathrm{band}_k(D),$ is denoted by $\mathrm{Band}_k(D).$ 
Then  $\mathrm{band}_1(D)=\mathrm{Band}_1(D)=D$ and we set $\mathrm{band}_0(D)=1.$ 
As for the bracelet operation also the 
invariants $[\mathrm{band}_k(D)]\in R_{\sigma(D)}(V)$ resp.~ $[\mathrm{Band}_k(D)]\in R_{\sigma(\mathrm{Thick}_k(D))}(V)$  are 
independent on the choice of the edge $E$ bounding the single-face of $D,$ and
on the order one superimposes the $k$ copies of $D$ (with connected boundary vertices). 

\begin{de}\label{def band}
Let $W$ be a web that arborizes to a single-face diagram $D.$
Then 
$$
\mathrm{band}_k(W):=\mathrm{band}_k(D).
$$
\end{de}
On the right of Figure ~\ref{twist2} an example the band operation is provided
for the hexagonal web on the left of Figure ~\ref{singlecycle} and for $k=3.$
\begin{rem}
The band operation we defined above is and adaptation of the
band power defined  in skein algebras, by D.\ Thurston in 
\cite{Thurston}. 
\end{rem}

\begin{lemma}\label{le_u1}
The following local identities hold:


\begin{itemize}
\item[(a)]\par
\begin{minipage}[t]{4cm}
\end{minipage}
 
\begin{minipage}{\linewidth}
\psfragscanon
\psfrag{n}[][c][0.7]{$n$}
\psfrag{m}[][c][0.7]{$m$}
\psfrag{=}{$=$}
\psfrag{+}{$+$}
\includegraphics[height=2cm]{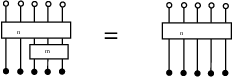}
\centering
\end{minipage}\vspace{0.5cm}

\item[(b)]\par
\begin{minipage}[t]{4cm}
\end{minipage}
\begin{minipage}{\linewidth}
\psfragscanon
\psfrag{a}{$\frac{1}{n}$}
\psfrag{n}[][r][0.7]{$n$}
\psfrag{n-1}[][r][0.7]{$n-1$}
\psfrag{b}{$\frac{n-1}{n}$}
\psfrag{c}{$\frac{1}{n}$}
\psfrag{d}{$\frac{n-1}{n}$}
\psfrag{m}[][r][0.7]{$m$}
\psfrag{=}{$=$}
\psfrag{+}{$+$}
\includegraphics[height=5cm]{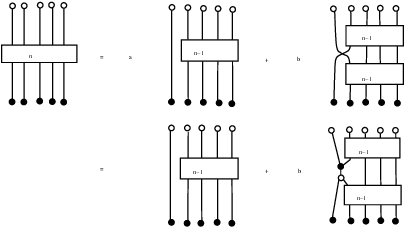}
\centering
\end{minipage}
\end{itemize}
\end{lemma}
\begin{proof}
The claim follows immediately from D.\ Thurston's identities for band powers in skein algebras, see \cite[Lemma 4.6]{Thurston} adapted to the context of tensor diagrams.

For part $(a)$: averaging once is the same as averaging twice.
For $(b)$: if one averages over $\mathfrak{S}_n,$ there is an edge
joining the first vertex above the box with the first vertex below the box with probability $1/n.$
The first vertex is connected with any other vertex below the box with probability
$(n-1)/n.$ These two probabilities
correspond to the coefficients
in front of the tensor diagrams on the right side of the equation. 
Applying skein relation $(a)$ for tensor diagrams
on the crossing yields the second equality.
\end{proof}
\begin{rem}
Part $(b)$ in Lemma ~\ref{le_u1} is the standard
definition of {\em boxes over oriented $n$-strands}
as described by T.\ Ohtsuki in \cite[Appendix B.2]{MR1881401} and by H.\ Wenzl in \cite{MR873400}. In particular,
the $k$-band is a special case of the $A_2$ 
internal clasp considered by G.\ Kuperberg in \cite[\S.7]{Kuperberg}.
With that definition, part $(a)$ follows from T.\ Ohtsuki and S.\ Yamada's work 
 \cite{MR1457194}, see also \cite[Appendix B.2,(B.8)]{MR1881401}.
\end{rem}

We now show how the last summand in Lemma ~\ref{le_u1} 
relation $(b)$ simplifies further.
For this relation to hold, a necessary condition is to view the above relation in 
a specific fragment of the thickening of a single-face non-elliptic web $W.$

\begin{lemma}\label{le_u2}
Let $W$ be a web that arborizes to a single-face diagram $D.$
Then the following local identities hold:\\

\begin{minipage}[t]{\linewidth}
\psfragscanon
\psfrag{n}[][r][0.7]{$n$}
\psfrag{n-1}[][r][0.7]{$n-1$}
\psfrag{d}{$\frac{n-1}{n}$}
\psfrag{m}{$m$}
\psfrag{=}{$=$}
\psfrag{a}{$+$}
\includegraphics[height=2cm]{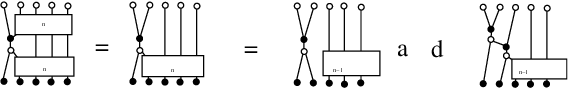}
\centering
\end{minipage}
in the thickening of the single-face of $D.$ \qed
\end{lemma}

\begin{thm}\label{thmC2:2}
Let $k\in\mathbb Z_{> 0}$ and let $W$ be a non-elliptic web that arborizes to a single-face diagram.
Then, the $k$-band operation of $W$ transforms the invariants as follows:
$$
[\mathrm{ band}_k(W)]=U_k([W],[b(W)])
$$
where $(U_{k})_{k\in\mathbb Z_{\geq0}}$ is the rescaled Chebyshev polynomial of 
the second kind.
\end{thm}
\begin{proof}
Proceed by induction. For small $n$ the claim can be checked directly.
For $n>2,$ the claim can be deduce from Lemma ~\ref{le_u1} part $(b),$ followed 
by part $(c)$ in Lemma ~\ref{le_u2} and solving the squares in the two summands of part $(c)$
following the reasoning of Lemma ~\ref{let} parts $(c)$ and $(d).$
\end{proof}

Keeping the assumptions of Theorem ~\ref{thmC2:2} we deduce the next result with 
Prop. ~\ref{monomialCheby2}.
\begin{cor}\label{corcheb2}
For all $k\geq 0$ the following identities hold,
$$ [W]^k=U_k+\Bigg\{ \binom{k}{1}-\binom{k}{0}\Bigg\}[b(W)]U_{k-2}+\dots+\Bigg\{ \binom{k}{\frac{k-1}{2}}-\binom{k}{\frac{k-1}{2}-1}\Bigg\}[b(W)]^{\frac{k-1}{2}}U_1$$
if $k$ is odd.
\begin{align*}
[W]^k=U_k+&\Bigg\{ \binom{k}{1}-\binom{k}{0}\Bigg\}[b(W)]U_{k-2}+\dots+\\
&\Bigg\{ \binom{k}{\frac{k}{2}-1}-\binom{k}{\frac{k}{2}-2}\Bigg\}[b(W)]^{\frac{k-2}{2}}U_2
+\Bigg\{\binom{k}{\frac{k}{2}}-\binom{k}{\frac{k}{2}-1}\Bigg\}[b(W)]^{\frac{k}{2}}U_0\\
\end{align*}
if $k$ is even. In particular, $[W]^k$ can be written as a positive linear 
combination of the Chebyshev polynomials of the second kind
$(U_{k})_{k\in\mathbb Z_{\geq0}}$\qed
\end{cor}

\section{Preliminaries on the $U_q(\mathfrak{sl}_3)$-invariant space $\mathrm{Inv}(V^{\sigma_n})$}\label{sect:inv}
In the rest of the paper, our goal is to test the canonical properties of the recursions described in Theorem ~\ref{thmC1.2} and Theorem ~\ref{thmC2:2}. 
To do so we link $(T_{k})_{k\in\mathbb Z_{\geq0}}$ and $(U_{k})_{k\in\mathbb Z_{\geq0}}$ to G.\ Lusztig's dual canonical basis. 
However, Lusztig's dual canonical basis is defined in the non-commutative setting of the
invariant space $\mathrm{Inv}(V^{\sigma_n})$ for the tensor product 
$V^{\sigma_n}$ of irreducible $U_q(\mathfrak{sl}_3)$-representations. 
To interpret our $\mathrm{SL}(V)$-invariants as $U_q(\mathfrak{sl}_3)$-invariants, we rely
on the shared diagrammatic representation of these invariants in terms of tensor diagrams, see also the discussion in
Section \ref{subsec:backto}. 

More precisely, in Section ~\ref{Invashom}, we present the invariant spaces $\mathrm{Inv}(V^{\sigma_n})$, $n\in\mathbb Z_{\geq0}$. In Section ~\ref{sect:growth}, we discuss how G.\ Kuperberg's web basis extents to a basis of $\mathrm{Inv}(V^{\sigma_n})$. We also explain how web invariants can be 
expressed in the induced tensor product basis of $\mathrm{Inv}(V^{\sigma_n})$.
A brief description of the combinatorics of flow lines, developed by M.\ Khovanov and G.\ Kuperberg's in \cite{KK}, and an explanation on
how this approach can be used to compute the change of basis between Kuperberg's web basis and 
the tensor product basis will also be provided.\\

The main reference for these sections is M.\ Khovanov and G.\ Kuperberg's beautiful work \cite{KK}. 
A good collection of additional background results and definitions can be found in the contributions \cite{Mackaay, LHRII, 20071503M}.


\subsection{Preliminaries on $U_q(\mathfrak{sl}_3)$-invariants}\label{Invashom}

Let $\mathbb C(q^{\frac{1}{2}})$ be the field of 
complex-valued rational functions in the indeterminate $q^{\frac{1}{2}}.$
Consider the deformation $U_q(\mathfrak{sl}_3)$ of the
universal enveloping algebra $\mathfrak{sl}_3.$
This is a Hopf algebra over $\mathbb C(q^{\frac{1}{2}})$ generated by
$E_i, F_i, K_i$ and $K^{-1}_i,$ $i=1,2$ and satisfying 
certain relations, see \cite[\S.2]{KK}. 
Throughout set $v=-q^{\frac{1}{2}}.$ 
Let $V^\circ$ be the canonical 3-dimensional representation of $U_q(\mathfrak{sl}_3)$. Let $\{e_1^\circ,e_0^\circ,e_{-1}^\circ\}$ be a basis for $V^\circ$
and $\{e_1^\bullet,e_0^\bullet,e_{-1}^\bullet\}$ the dual basis for  $V^\bullet=(V^{\circ})^\ast$. Here $\ast$ denotes the standard duality over $\mathbb C(v).$
Let $\sigma_n=(s_1,s_2,\dots, s_n)\in\{\circ,\bullet\}^n$ be a {\em signature}, that is a cyclic
word indexing the tensor product of copies of $V^\circ$ and its dual:
$
V^{\sigma_n}=V^{s_1}\otimes V^{s_2}\otimes\dots\otimes V^{s_n}.
$
If $\sigma$ is the empty sequence then $V^\sigma$
is the trivial representation $\mathbb C(v)$ with $U_q(\mathfrak{sl}_3)$-module structure given by the co-unit. 
Let  $J=(j_1,j_2,\dots, j_n)\in\{1,0,-1\}^n$ be a non-cyclic word in the alphabet $\{1,0,-1\}$ called a {\em state string}.
For any signature $\sigma_n$ and state string $J$ of length $n,$ let 
$
e_{J}^{\sigma_n}=e_{j_1}^{s_1}\otimes e_{j_2}^{s_2}\otimes \dots\otimes e_{j_n}^{s_n}
$
be a simple tensor in $V^{\sigma_n}.$
In the rest of the paper, we are interested in the invariant space
$$
\mathrm{Inv}(V^{\sigma_n})=
\mathrm{Hom}_{U_q(\mathfrak{sl}_3)}(V^{\sigma_n},\mathbb C(v))
$$
defined as in \cite{KK}.

Let
\begin{align*}
[Y_\bullet^{\circ\circ}]:V^\bullet&\rightarrow V^\circ\otimes V^{\circ}\\
e_1^{\bullet}&\mapsto e_1^{\circ}\otimes e_{0}^{\circ}+v^{-1} e_0^{\circ}\otimes e_{1}^{\circ} \\
e_0^{\bullet}&\mapsto e_1^{\circ}\otimes e_{-1}^{\circ}+v^{-1} e_{-1}^{\circ}\otimes e_{1}^{\circ}\\
e_{-1}^{\bullet}&\mapsto e_0^{\circ}\otimes e_{-1}^{\circ}+v^{-1} e_{-1}^{\circ}\otimes e_{0}^{\circ}.
\end{align*}

Consider also:
\begin{align*}
[\mho^{\circ\bullet}]:\mathbb C(v)&\rightarrow V^\circ\otimes V^{\bullet}  \\
1&\mapsto e^\circ_1\otimes e_{-1}^{\bullet} +v^{-1}e^\circ_0\otimes e_0^{\bullet} +v^{-2}e^\circ_{-1}\otimes e_{1}^{\bullet}
\end{align*}
and the pairing given by:
\begin{align*}
[\Omega^{\bullet\circ}]&:V^\circ\otimes V^{\bullet}\rightarrow\mathbb C(v),&&e^\circ_{-1}\otimes e_1^{\bullet}\mapsto 1&& \, e^\circ_{0}\otimes e_0^{\bullet}\mapsto v& \, e^\circ_{1}\otimes e_{-1}^{\bullet}\mapsto v^2.
\end{align*}
and where all other basis elements are sent to zero. Dually, one also defines
$[Y_\circ^{\bullet\bullet}]$ $[\mho^{\bullet\circ}]$ and
$[\Omega^{\circ\bullet}].$
Moreover, since
\begin{align}\label{isohominv1}
\mathrm{Hom}_{U_q(\mathfrak{sl}_3)}(V^{\sigma_l}, V^{\sigma_m})\cong
\mathrm{Inv}((V^{\sigma_l})^\ast\otimes V^{\sigma_m})
\end{align}
we view the above morphisms as invariants, see \cite{KK}. All other elements of $\mathrm{Inv}(V^{\sigma_n})$, $n\in\mathbb Z_{\geq 0}$, can be obtained by these invariants by contraction and tensor products, as well as the usual operations of addition and scalar multiplication \cite{Kuperberg}.   
M.\ Khovanov and G.\ Kuperberg described the above invariants diagrammatically as in Figure ~\ref{YL}.
Each such diagram is embedded into a half-space with fixed vertical orientation. The boundary vertices represent the arguments of the invariant: white points, $\circ,$ represent arguments in $V^\circ,$ while black points, $\bullet,$ represent arguments in $V^\bullet.$ 
For $[\Omega^{\bullet\circ}]$ and $[\Omega^{\circ\bullet}]$ the opposite convention holds.
Tensor product and contraction are given by disjoint union and gluing. In this non-commutative setting, invariants depend on the number and location of the pairings. This will be specified diagrammatically by drawing a point with vanishing slope of an imaginary tangent line at the location of the contraction (vanishing tangent points). For some purposes it will be convenient
to draw the boundary points of tensor diagrams on two horizontal lines to clarify the arguments
of the corresponding homomorphisms.

\begin{figure}[ht]
\psfragscanon
\psfrag{A}{$[Y_\bullet^{\circ\circ}]$}
\psfrag{B}{$[Y_\circ^{\bullet\bullet}]$}
\psfrag{C}{$[\lambda^{\circ}_{\bullet\bullet}]$}
\psfrag{D}{$[\lambda^{\bullet}_{\circ\circ}]$}
\psfrag{a}{$[\mho^{\circ\bullet}]$}
\psfrag{b}{$[\mho^{\bullet\circ}]$}
\psfrag{e}{$[Y_\bullet^{\circ\circ}]$}
\psfrag{f}{$[Y_\circ^{\bullet\bullet}]$}
\psfrag{c}{$[\Omega^{\bullet\circ}]$}
\psfrag{d}{$[\Omega^{\circ\bullet}]$}
 \includegraphics[scale=0.5]{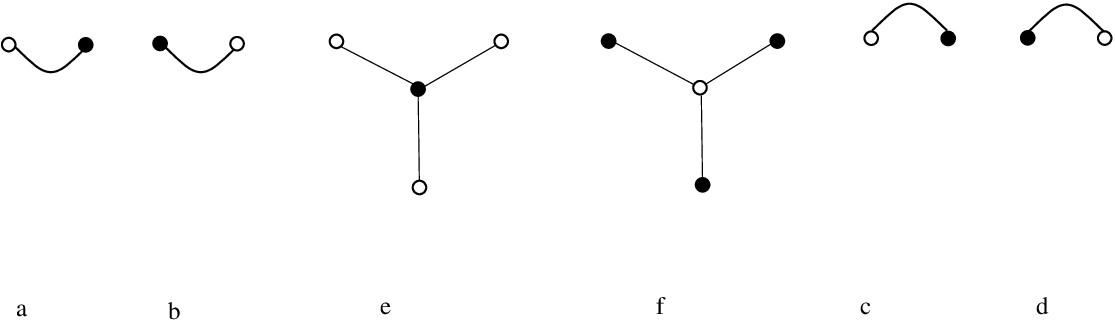}
\centering
\caption{}\label{tensor2} \label{YL}
 \end{figure}

\begin{de}
A $q$-tensor diagram $D$ is a tensor diagram embedded into an half-space with fixed vertical orientation
and with two types of crossings, only univalent boundary vertices (dawn on a horizontal line) and with a collection of vanishing tangent points.
\end{de}
In a $q$-tensor diagram one boundary point is distinguished as first (in green in the following figures).  
These $q$-tensor diagrams are considered up to an isotopy that fixes both
the boundary points as well as the vanishing tangent points. 
To these $q$-tensors diagrams invariants can be associated (not uniquely). These satisfy a number of local relations, called {\em (quantum) skein relations}, see Kuperberg's work \cite{Kuperberg}. 

\begin{de}
A non-elliptic $q$-web is a planar $q$-tensor diagram so that all internal faces have at least six sides.
Invariants defined by $q$-non-elliptic webs will be called {\em $q$-web invariants}.
\end{de}
In the following, we drop the prefix $q$- and denote web invariants defined by a non-elliptic web $W$ by $[W].$ The signature of $[W]$ can be deduced from the boundary of $W$ and is denoted by $\sigma(W)$ or simply by $\sigma.$


\subsection{Web basis and tensor product basis}\label{sect:growth}

The next result, combines G.\ Kuperberg's \cite{Kuperberg} result
showing that web invariants of signature $\sigma$ 
form a $\mathbb C(q)$-basis for $\mathrm{Inv}(V^\sigma)$ 
with M.\ Khovanov and G.\ Kuperberg's Theorem 2  in \cite{KK}.

\begin{thm}\cite{Kuperberg,KK}\label{thmkk}
Let $\sigma$ be a signature.
Web invariants associated with non-elliptic webs with signature $\sigma$
form a $\mathbb Z [v,v^{-1}]$-linear basis for 
$\mathrm{Inv}(V^\sigma).$
\end{thm}

Let $[W]\in\mathrm{Inv}(V^\sigma)$
be an invariant described by a non-elliptic web $W.$ 
Then $[W]$ has a distinguished state string
which parametrizes $W.$ 
This state string can be deduced
using the {\em minimal cut path algorithm}. This algorithm takes as input 
a non-elliptic web $W$ with its signature $\sigma$
and outputs a unique state string of $W,$ called a 
{\em dominant lattice path}, denoted by $J(W).$
Let us also briefly recall the inverse algorithm 
given by the {\em growth algorithm}, see \cite[Prop.1]{KK}.
The input of this algorithm is a
signature $\sigma$ and a state 
string $J$ of the same length. The output is
a web, with signature $\sigma,$ ``grown'' by inductively concatenating $Y,$ $H$ (obtained by
composing a $Y$ and a $\lambda$, see Equation \ref{def:lambda} for the latter) and $\mho$ pieces, following certain simple rules described in the above mentioned reference.
Importantly, it follows from these two algorithms that web invariants are indexed by 
their dominant lattice path $J(W).$
We will see that also dual canonical basis elements can be indexed by dominant lattice paths, see the discussion after Theorem \ref{Lusztig}.

In the following, the notation $J'<J$ indicates the {\em lexicographic order} 
on the set of state strings (of fixed length) in the alphabet $\{1>0>-1\}.$
\begin{thm}\label{thm3KK}\cite[Thm.2]{KK}
The web invariant $[W]\in\mathrm{Inv}(V^\sigma)$ 
expands as
$$
[W]=e_{J(W)}^\sigma+\sum_{J'<J(W)}c(\sigma,J(W),J')e_{J'}^\sigma
$$ 
for some coefficients $c(\sigma,J(W),J')\in\mathbb{Z}_{> 0}[v,v^{-1}].$
\end{thm}
As next, we recall how the
coefficients $c(\sigma,J(W),J')$ can be determined using the combinatorics of flow lines developed
in \cite{KK}. 
A {\em flow} $F$ on any tensor diagram $W$ is
an oriented subgraph of $W$ 
that contains exactly two out of three edges incident to each trivalent
vertex. The connected components of $F$ 
are called {\em flow lines}. 
The orientation of these flow lines is independent 
of the  coloring of the boundary of $W.$
Let us denote $W$ with a flow $F$ by $W_F.$
To $W_F$ a state string, called the {\em boundary state} 
of $F$, $J_F \in\{1,0,-1\}^n,$ can be associated.
The entries of $J_F$ can be read off from 
the boundary vertices of $W_F,$
according to the following convention.  
For the $i$-th boundary vertex of $W_F$ the entry $j_i$ of $J_F$ is 1
if a flow line in $W_F$ is oriented 
downwards at the edge $i$;
$j_i=-1,$ if a flow line in $W_F$ is oriented 
upwards at the edge incident to $i$ and
$j_i=0,$ otherwise. 
To these flows {\em weights} are associated 
according to the local chart given in Figure ~\ref{weightY}. 
In the tensor diagrams in Figure ~\ref{weightY} we omit the coloring of the 
boundary vertices, as the weight of the flow lines does not depend on it.
Moreover, changing the vertical orientation of the invariants $[Y_\bullet^{\circ\circ}]$
$[Y_\circ^{\bullet\bullet}]$ one obtains the invariant
$
[\lambda^{\circ}_{\bullet\bullet}]:V^\bullet\otimes V^\bullet\rightarrow V^\circ
$
defined by
\begin{equation}\label{def:lambda}
\begin{split}
e_1^{\bullet}\otimes e_{0}^{\bullet}&\mapsto ve_{1}^\circ\\
e_1^{\bullet}\otimes e_{-1}^{\bullet}&\mapsto ve_{0}^\circ\\
e_0^{\bullet}\otimes e_{-1}^{\bullet}&\mapsto ve_{-1}^\circ
\end{split}
\quad\quad
\begin{split}
e_0^{\bullet}\otimes e_{1}^{\bullet}&\mapsto e_{1}^\circ\\
e_{-1}^{\bullet}\otimes e_{1}^{\bullet}&\mapsto e_{0}^\circ\\
e_{-1}^{\bullet}\otimes e_{0}^{\bullet}&\mapsto e_{-1}^\circ
 \end{split}
\end{equation}
and zero otherwise, together with 
invariant $\lambda^{\circ}_{\bullet\bullet}$ defined dually.
Thus, changing the vertical orientation of tensor diagrams changes their local weights, as one sees
in Figure ~\ref{weightY}.

\begin{figure}[ht]
\center
\psfragscanon
\psfrag{1}{1}
\psfrag{a}{$v$}
\psfrag{b}{$v^{-1}$}
\psfrag{e}{$v^{-2}$}
\psfrag{f}{$v^2$}
  \includegraphics[width=12cm]{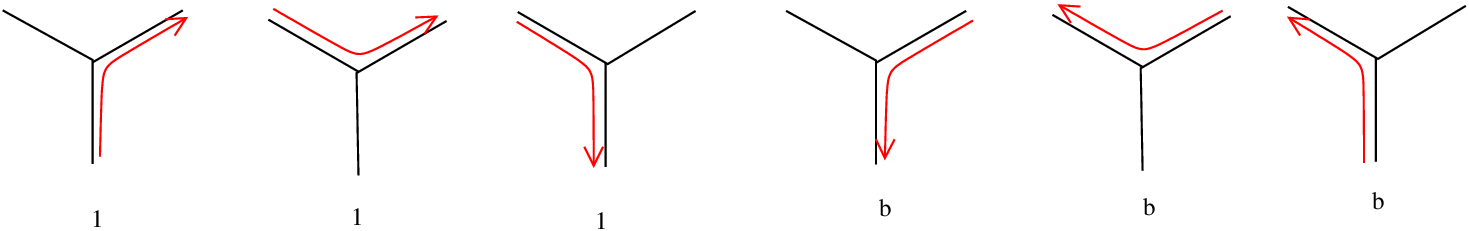}
 \includegraphics[width=12cm]{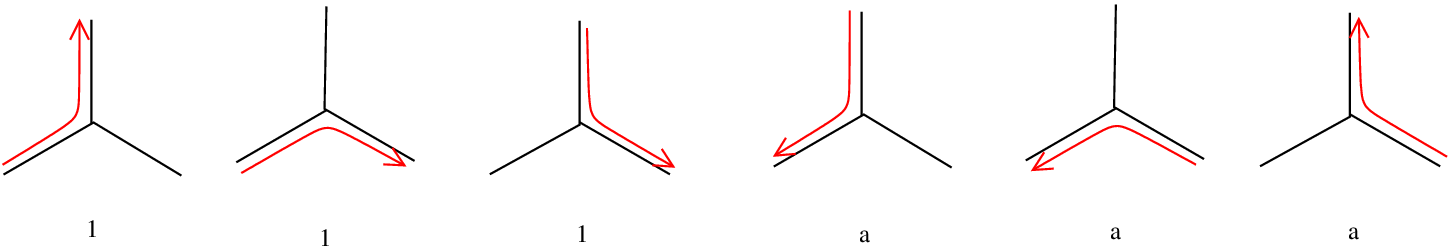}
  \includegraphics[width=12cm]{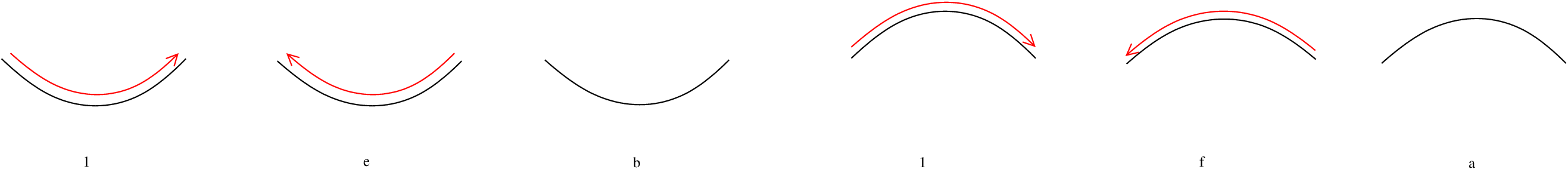}
 \caption{}
    \label{weightY}
\end{figure}
For a general web invariant $[W]\in\mathrm{Inv}(V^{\sigma_n})$ 
each state has some weight $v^n$, where $n$
is called the {\em exponent of the state}.
The sum of weights of 
all flows with boundary state 
$J_F$ determines the {\em weight of the state $J_F$} of $W_F$.
This weight is then the coefficient 
$c(\sigma, W, J(F))$ of $e_{J(F)}^\sigma$, in Theorem ~\ref{thm3KK}.
Since the  weight of any state in non-negative 
it follows that $c(\sigma, W, J(F))\in\mathbb{Z}_{>0}[v^{-1},v]$.


\section{Lusztig's dual canonical basis for $\mathrm{Inv}(V^{\sigma_n})$}\label{dualcan}

In this section, we provide an axiomatic description of Lusztig's dual canonical basis
for $\mathrm{Inv}(V^{\sigma_n})$ following closely Khovanov-Kuperberg's exposition given in \cite[\S.6]{KK}.
We then move on to present some useful operations on 
non-elliptic webs which preserve dual canonical invariants, see Proposition ~\ref{propemd}.
The following, will be a brief description of a second combinatorial approach used
to determine if a web invariant is dual canonical, see Theorem ~\ref{noredgraph}.
This combinatorial approach was developed by Louis-Hadrien Robert 
\cite{LHRII} and uses red-graphs. 
Finally, in Section ~\ref{subsec:backto} we discuss how $U_q(\mathfrak{sl}_3)$-invariants relate 
to $\mathrm{SL(V)}$-invariants.


\subsection{The axiomatic description of Lusztig's dual canonical basis}

As before, let $J=(j_1,j_2,\dots, j_n)\in\{1,0,-1\}^n$ be a state string. Let
$\sigma=(s_1,s_2,\dots, s_n)\in\{\circ,\bullet \}^n$  be a signature and let
$ e^\sigma_J=e_{j_1}^{s_1}\otimes e_{j_2}^{s_2}\otimes\dots\otimes e_{j_n}^{s_n}\in V^\sigma$ 
be the simple tensor corresponding to a given $J$ and $\sigma.$
For any $n$ and choice of $\sigma$, one defines inductively a $v$-antilinear 
endomorphism~$\Phi:V^{\sigma}\rightarrow V^{\sigma},$ as follows.
On $V^\bullet,$ resp.~  $V^\circ,$ the endomorphism $\Phi$ is defined by
$$\Phi(\sum c_ie_i^s)=\sum\bar{c_i}e_i^s$$
where the sum is over $i\in\{1,0,-1\},$ and $c_i\in\mathbb C(v).$ 
Here $\,\bar{}:\mathbb C(v)\rightarrow \mathbb C(v)$
is the $\mathbb C(v)$-algebra involution given by $\overline{v^n}=v^{-n}$
for all $n\in \mathbb Z.$ 
Next, assume that $\Phi$ is already defined on the tensor products $V^{\sigma_1}$ and $V^{\sigma_2},$ for two arbitrary signatures
$\sigma_1$ and $\sigma_2.$
On  the tensor product $V^{\sigma_1}\otimes V^{\sigma_2}$ one then defines $\Phi$ 
by setting
$$
\Phi\big(e_{J_1}^{\sigma_1}\otimes e_{J_2}^{\sigma_2}\big)=\overline{\Theta}\big(\Phi^{\sigma_1}(e_{J_1}^{\sigma_1})\otimes \Phi^{\sigma_1}(e_{J_2}^{\sigma_2})\big).
$$
Here  $J_1$ and $J_2$ are arbitrary state strings of the same length as $\sigma_1,$ resp.\ $\sigma_2,$ and
$\overline{\Theta}\in U_q(\mathfrak{sl}_3)\hat{\otimes} U_q(\mathfrak{sl}_3)$
is the bar-conjugate of the quasi-$R$-matrix 
defined in a completion of $U_q(\mathfrak{sl}_3)\otimes U_q(\mathfrak{sl}_3).$ 
The endomorphism $\Phi$ can be shown to be well defined on all tensor products of copies of $V^{\circ}$ and $V^{\bullet}.$
Using this endomorphism $\Phi,$ Khovanov-Kuperberg define in \cite[Thm.3]{KK}
Lusztig's dual canonical basis of $V^{\sigma},$ as follows. 

\begin{thm} \cite[Thm.3]{KK}\label{Lusztig}
For any signature $\sigma$ and state string $J,$ there is a unique element
$
\ell_{J}^{\sigma} \in V^{\sigma}
$
which is invariant under $\Phi$ and whose expansion in the tensor
product basis is
$$
\ell_{J}^{\sigma}=e_{J}^{\sigma}+\sum_{J'}c(\sigma,J,J')e^{\sigma}_{J'}  
$$
with $c(\sigma,J,J')\in v^{-1}\mathbb{Z}[v^{-1}].$
\end{thm}

If all coefficients $c(\sigma,J,J')$ in Theorem ~\ref{Lusztig} are such that
$c(\sigma,J,J')\in v^{-1}\mathbb{Z}[v^{-1}],$ one says that 
$J$ has {\em the negative exponent property}.
The set $\{ \ell_{J}^\sigma \},$ indexed by all state strings $J,$ 
is a $\mathbb{C}(v)$-linear basis for $V^{\sigma},$ 
called the {\em dual canonical basis} for $V^{\sigma}.$ 

Notice that not all dual canonical basis elements of $V^\sigma$ also
belong to $\mathrm{Inv}(V^\sigma).$ 
To obtain a $\mathbb C(v)$-linear basis for the invariant space
$\mathrm{Inv}(V^\sigma)$ one has to consider the subset of 
all $\{ \ell_{J}^{\sigma} \}$ indexed by dominant paths $J(W)$, 
see the comments below Theorem 3 in \cite{KK}.

By abuse of terminology, we say that the invariant $\ell_{J(W)}^{\sigma(W)}$ is dual canonical if it 
belongs to the dual canonical basis of $\mathrm{Inv}(V^{\sigma(W)}).$ We write then
$\ell(W)$ instead of $\ell_{J(W)}^{\sigma(W)}$.

\begin{rem}
In Theorem ~\ref{Lusztig} we allow any ordering 
in the summation, following M.\ Khovanov and G.\ Kuperberg's approach presented in \cite[Thm.\ 3]{KK}. 
\end{rem}
From Theorem ~\ref{thm3KK} we know that on non-elliptic webs there is a distinguished 
flow line which has overall weight $1$ and boundary state $J(W),$ \cite{KK}.
The boundary state of this flow is called the {\em leading state of $W$}. 
From Theorem ~\ref{Lusztig} one then deduces that $[W]$ is dual canonical if 
and only if every state different than the leading state of $W$ has the negative exponent property.


\subsection{Web basis and dual canonical invariants}\label{subkk}

\begin{figure}[t]
\center
\psfragscanon
\psfrag{-}{$-$}
 \includegraphics[width=8cm]{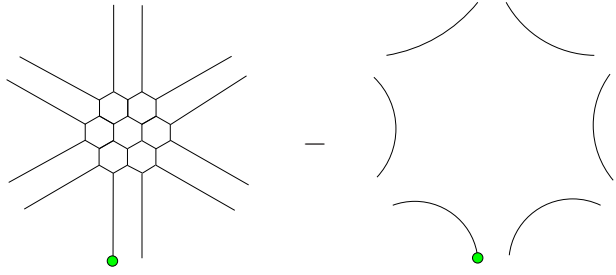}
 \caption{The invariant 
 $\ell(\mathrm{Thick}_2(W))\in\mathrm{Inv}((V^\circ \otimes V^\circ \otimes V^\bullet \otimes V^\bullet)^{\otimes3})$ 
 defined by a formal linear combination of non-elliptic webs.}\label{che2}
\end{figure}

As explained earlier,  the dual canonical basis and the web basis share a 
number of properties but are in general different from each other, as explained
in the next result.
\begin{thm}\cite[Thm.1]{KK} \label{thm.1}
Every web invariant in
$\mathrm{Inv}(V^{s_1} \otimes V^{s_2} \otimes \dots \otimes V^{s_n}),$
$s_i\in\{\circ,\bullet\},$
is dual canonical when $n \leq 12,$ except
for a single web invariant in
$$\mathrm{Inv}((V^\circ \otimes V^\circ \otimes V^\bullet \otimes V^\bullet)^{\otimes3})$$
and its counterparts given by cyclic permutation of tensor factors.
\end{thm}
The non-dual canonical web invariant in Theorem ~\ref{thm.1}
is defined by the second non-elliptic web represented on the left 
of Figure ~\ref{thick35} (and any rotation of it). In the notation developed in this paper, this non-elliptic web is
given by the full unclasp of $\mathrm{thick}_2(W),$ previously denoted by $\mathrm{Thick}_2(W).$
Let $B(W)$ be the full unclasp of the coefficient $b(W)$ associated with the hexagonal web $W$
as defined in Section ~\ref{coeffsub}. Let $$\ell(\mathrm{Thick}_2(W))=[\mathrm{Thick}_2(W)]-[B(W)]\in\mathrm{Inv}((V^\circ \otimes V^\circ \otimes V^\bullet \otimes V^\bullet)^{\otimes3})$$ be defined by the formal linear combination of non-elliptic webs shown in Figure ~\ref{cheb2:intro}.

\begin{thm}\cite[Thm.4]{KK} \label{thm.4}
The invariant $\ell(\mathrm{Thick}_2(W))$ is dual canonical.
\end{thm}
The invariant  $\ell(\mathrm{Thick}_2(W))$ then specializes to the invariant $[\mathrm{band}_2(W)]\in R_{\sigma(W)}(V)$ defined in Section ~\ref{subsect:band} at $q=1$ and after he partial restitution at neighboring and equally colored boundary points.
To see this, notice that both invariants are defined by the same linear combination of non-elliptic webs, after clasping substrings of equally colored endpoints.


\subsection{Operations preserving dual canonical invariants}\label{op}
Despite the fact that the dual canonical and web basis are generally different from each other, 
there are still many web-invariants which are dual canonical basis elements. 
In the next results, we present operations on non-elliptic webs defining dual canonical
invariants which preserve the dual canonical property. 

\begin{prop}\label{propemd}
Let $[D_1]$ and $[D_2]$ be dual canonical web-invariants.
Then any proper embedding of $D_1$ and $D_2$ into the half-plane 
defines a dual canonical web-invariant.
\end{prop}
The precise signature defining the invariant space corresponding to the invariants
given in  Proposition ~\ref{propemd} is read off from the boundary of the non-elliptic webs.
The proof of Proposition ~\ref{propemd} is a consequence of the next result.
\begin{prop}\cite[Prop. 2]{KK}\label{phi_in}
Every web invariant of signature $\sigma$ in $\mathrm{Inv}(V^\sigma)$ is invariant under $\Phi.$
\end{prop}

\begin{proof}
Denote by $D_1\cup D_2$ the image of a proper embedding of 
$D_1$ and $D_2.$  Since $D_1\cup D_2$ consists of two connected components,
one consisting of $D_1$ and one of $D_2,$ $D_1\cup D_2$ is again a non-elliptic web.
Then $[D_1\cup D_2]$ is $\Phi$-invariant by Proposition ~\ref{phi_in}. 
Moreover, all flows
of $D_1\cup D_2$ are obtained as a union of flows of
$D_1$ with flows of $D_2.$
The weights of these flows
are obtained by multiplying the weights of all boundary states
of $D_1$ with the weights of all boundary states of $D_2.$ 
Since $[D_1]$ and $[D_2]$ are dual canonical,
by assumption, all these weights are in $v^{-1}\mathbb Z[v^{-1}]$ 
or $1$ (which is the case only for the leading states on $D_1$ and $D_2$).
Hence, $[D_1\cup D_2]$ has the negative exponent property and the claim follows.
\end{proof}
\begin{cor}
Proposition ~\ref{propemd} extends by linearity to any invariant of signature $\sigma_n$ of $\mathrm{Inv}(V^{\sigma_n}).$
\qed
\end{cor}

\begin{cor}\cite[Corollary to Thm. ~\ref{Lusztig}]{KK}\label{corkk}
The tensor product of dual canonical invariants is a dual canonical invariant.
\end{cor}

To illustrate Proposition ~\ref{propemd}, consider the 
web invariant $[W\cup B\cup B]\in\mathrm{Inv}(V^{\sigma(W\cup B\cup B)})$ given as in Figure ~\ref{cup}. This invariant is defined by the full unclasping of the superimposition 
of hexagonal web $W$ and of two copies of $B=B(W),$ the coefficient associated 
to $W$ described in Section ~\ref{coeffsub}. 
In this example, the invariants $[W]$ and $[B]$ are both dual canonical, which can be seen for example 
with Theorem ~\ref{thm.1}.
Hence, the same is true for the invariant $[W\cup B\cup B]$ by Proposition ~\ref{propemd}.
\begin{figure}[ht]  
               \includegraphics[width=4cm]{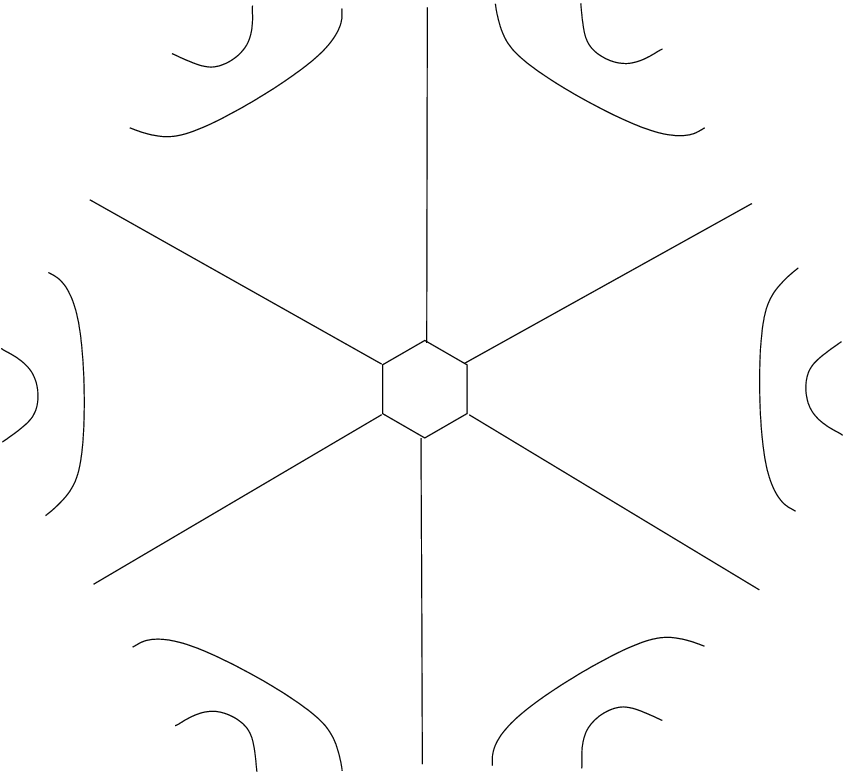}
               \caption{The invariant $[W\cup B\cup B].$}\label{cup}      
\end{figure}

Let us conclude this section by recalling a further 
useful operation on dual web invariants 
preserving the dual canonical property. 
\begin{prop}\cite[Prop.\ 3]{KK}\label{prop3}
Let $[W]$ be a dual canonical web invariant.
Let $W'$ be a non-elliptic web obtained from $W$ by
adding a $Y$ or a double $H$ at the boundary of $W.$
Then $[W']$ is a dual canonical web invariant.
\end{prop}


\subsection{Red graphs and dual canonical basis elements}\label{subsec:ref}

Red graphs are certain graphs defined in non-elliptic webs. These
can be used to determine if the corresponding web invariant is dual canonical.
The theory of red graphs was developed by Louis-Hadrien Robert
\cite{LHRII} and we will bring it together with the categorification of tensor diagrams
given by M.\ Mackaay, W.\ Pan and D.\ Tubbenhauer \cite{Mackaay} to deduce Theorem ~\ref{noredgraph}.

The following definitions are all taken from L.H.\ Robert's work \cite{LHRII}. Throughout the section, let $W$ be any non-elliptic web.
\begin{de}
A {\em red graph} for $W$ is a non-empty induced subgraph $G$ of the dual graph of $W$ 
such that:
\begin{itemize}
  \item  the vertices of $G$ correspond to some subsets of interior faces of $W$ diffeomorphic to discs; 
  \item  if $f_1,$ $f_2$ and $f_3$ are adjacent faces of $W$ sharing a vertex, 
  then at least one face is not a vertex of $G.$
\end{itemize}
\end{de}
Let $f$ be a vertex of a red graph $G$ for $W.$ 
\begin{de}Half edges of $W$ which are adjacent to $f$ and which do not bound other faces of $G$ are called
{\em gray half-edges} of $f$ in $G.$
\end{de}
Let the {\em external degree} of $f,$ denoted by $\mathrm{ed}(f)$ 
be the number of gray half-edges adjacent to $f$ which do 
not bound $f$ or another vertex of $G$. 
\begin{de}
Let $o$ be an orientation of a red graph $G$ of $W$.
The {\em level $i_o(f)$ of $f$ of $G$} is given by
$$i_o(f):= 2-\frac{1}{2}\mathrm{ed}(f)-\vert \{\textrm{edges of $G$ pointing to $f$}\}\vert.$$ 
The {\em level of $G$ } is given by the sum of the levels of all vertices of $G$, or equivalently, by the formula:
$$
 I(G)=2 \vert V(G) \vert- \vert E(G)\vert -\frac{1}{2} \sum_{f\in V(G)}\mathrm{ed}(f).
$$
\end{de}
\begin{de} A red graph is {\em admissible} if one can choose an orientation $o$ of $G$
such that for every vertex $f$ of $G$ one has
$i_o(f)\geq0$. Such an orientation is called {\em fitting}. In addition, an admissible  
red graph $G$ for $W$ is {\em exact} if 
$I(G)=0.$
\end{de}

\begin{de}
Let $G$ be a red graph for $W.$ A {\em pairing} of $G$ is a
partition of the gray half-edges of $G$ into subsets of 2 gray half-edges adjacent to the
same face of $W$, one pointing towards it, and the other pointing away from it. A red
graph together with a pairing is called a {\em paired red graph}.
\end{de}

\begin{figure}[ht]
\centering
\includegraphics[width=12cm]{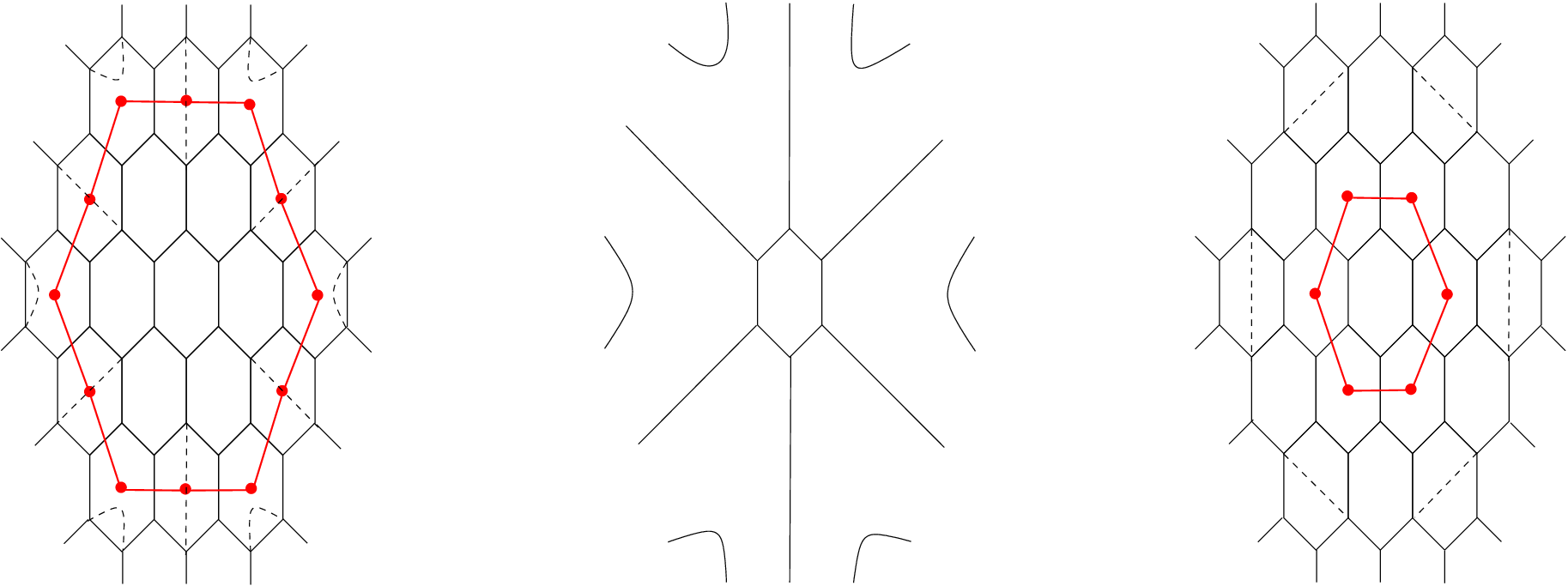}
\caption{Two exact paired red graphs and their $G$-reduction given by the non-elliptic web 
in the middle. Dashed lines indicate the pairings of the red graphs.}\label{Gredfig}
\end{figure}

\begin{de}\label{dered}
Let $G$ be a paired red graph for $W$. The
{\em $G$-reduction of $W$} is the web $W_G$ 
constructed as follows:  To every face of $W$ corresponding to a vertex of  $G$
\begin{itemize}
\item remove all edges adjacent to this face;
\item connect the gray half-edges of $G$ according to the pairing.
\end{itemize} 
\end{de}
The $G$-reduction of $W$, $W_G$,
always has the same signature as $W.$
Moreover, $W_G$ is considered
up to isotopy fixing the boundary and points with horizontal tangent.
Finally, notice that $W_G$ might be elliptic, even when $W$ is non-elliptic.

To illustrate the above definition, 
we give two examples
of paired red graphs in Figure ~\ref{Gredfig}. Both red graphs have exactly one pairing.
The $G$-reduction of the red graph on the right is elliptic.
Reducing with quantum skein relations yields $2^6$ terms, one of which
is the non-elliptic web drawn in
the middle of the figure. The unique $G$-reduction of the red graph on the left
yields the non-elliptic web in the middle of the figure.

\begin{thm}\label{noredgraph}
Let $[W]\in\mathrm{Inv}(V^\sigma)$ be a web invariant of signature $\sigma.$ If $W$ has
no exact red graph, then
$[W]$ is dual canonical in $\mathrm{Inv}(V^\sigma).$
\end{thm}
\begin{proof}
First, if $G$ is an exact red graph for $W$ then, by definition,
the level of $G$ is zero. With Theorem 3.10
and Theorem 3.11 in L.H.\ Robert's work
\cite{LHRII}, one deduces 
that the projective graded module corresponding to $W$ is indecomposable.
One then concludes with Theorem 5.31 in M.\ Mackaay, W.\ Pan and D.\ Tubbenhauer's work 
 \cite{Mackaay}. Notice that
the factor normalizing the grading, implies that 
the leading term of $[W]$ in $\mathrm{Inv}(V^\sigma)$ has weight 1.
\end{proof}

\subsection{Dual canonical invariants in $R_{\sigma'}(V)$}\label{subsec:backto}

In what follows, we give a more detailed description of the relation between the invariant ring
$R_{a,b}(V)$ and the direct sum of $U_q(\mathfrak{sl}_3)$-invariant spaces $\mathrm{Inv}(V^{\sigma})$.
To see this, consider again the decomposition of the ring $R_{a,b}(V)$
into multi linear components as given in Lemma \ref{decop}.
From this decomposition, one deduces that each multi linear component
can be obtained as a specialization at $q=1$ of an $U_q(\mathfrak{sl}_3)$-invariant 
space $\mathrm{Inv}(V^{\sigma})$, for an appropriate choice of $\sigma$
and passing to the quantum symmetric algebra obtained as a quotient of the tensor product $V^{\sigma}.$
For more details on this construction we invite the reader to compare with J.\ Brundan's work \cite[\S.5 and \S.6]{BRUNDAN200617}. Notice that there the author considers the $U_q(\mathfrak{gl}_n)$ case, but
we expect that the case of $U_q(\mathfrak{sl}_3)$-invariants follows in a similar way.

To better illustrate this construction, let us consider the case of $R_{0,b}(V)$. 
The above discussion then gives rise to the following commutative diagram:
\begin{displaymath}
 \xymatrix{   
 & \displaystyle \bigoplus_{p\in\mathbb N^b} \Big({V_q^\bullet}^{\otimes \vert  p \vert} \Big)^{U_q(\mathfrak{sl}_3(\mathbb{C}))} \ar@{->>}[d]\\
 \mathbb C_q[Gr_{3,b}]\ar@{~>}[d]^{q\rightarrow 1} \ar[r]^{\cong^{(1)}}& \displaystyle \bigoplus_{p\in\mathbb N^b}  \Big(S^p_q(V_q)\Big)^{U_q(\mathfrak{sl}_3(\mathbb{C}))}\ar@{~>}[d]^{q\rightarrow 1} \\
\mathbb C[Gr_{3,b}]\cong R_{0,b}(V)\ar[r]^{\cong^{(2)}} &\displaystyle \bigoplus_{p\in\mathbb N^b} \Big(S^p(V)\Big)^{U(\mathfrak{sl}_3(\mathbb{C}))}
             }
\end{displaymath}
where the linear isomorphism (1) essentially follows from Theorem 15 
in \cite{BRUNDAN200617} and the linear isomorphism (2) is given in Lemma \ref{decop}. In the above diagram
we add a subindex $_q$ to differentiate between the commutative and non-commutative setting.

Let $[D]\in\mathrm{Inv}(V^{\sigma(D)})$ be expressed as 
a $\mathbb Z[v,v^{-1}]$-linear combination of web invariants in
Kuperberg's basis. After specializing $q\rightarrow 1$, in the above sense,
the same linear combination of web invariants gives rise to an $\mathrm{SL(V)}$-invariant in
$R_{\sigma(D)}(V)$ which we denote again by $[D]$. In this situation, we say that $[D]\in \mathrm{Inv}(V^{\sigma(D)})$ {\em specializes} to $[D]\in R_{\sigma(D)}(V)$ at the classical limit of $q\rightarrow 1$. 
By abuse of terminology, we also say that $[D]\in \mathrm{Inv}(V^{\sigma(D)})$ {\em specializes} 
to $[D_c]\in R_{\sigma(D_c)}(V)$, if in addition, the invariant $[D_c]$ 
is obtained from $[D]$ after a partial restitution sending
the signature $\sigma(D)$ to $\sigma(D_c)$.

In the following examples, the only partial restitutions we consider are at
successive and equally colored substrings.

\begin{de}\label{spe}
Let $J(D)\in\{1,0,-1\}^n$ be a dominant lattice path and $\sigma(J(D))\in\{\circ,\bullet\}^n$ 
be the corresponding signature.
We say that $\ell(J(D))\in R_{\sigma(J(D))}(V)$
is {\em dual canonical} if $\ell(J(D))$ is the specialization
at the classical limit of $q\rightarrow 1$ of a
dual canonical invariant in $\mathrm{Inv}(V^{\sigma(J(D))})$ parametrized by  $J(D)$ and $\sigma(J(D))$
\end{de}


 \section{Chebyshev polynomials and the dual canonical basis}\label{sect:cheb2}
 
In Section ~\ref{Cheb_2}, we describe a family of Lusztig's dual canonical basis
elements in Kuperberg's web basis, see Theorem ~\ref{band2cd:2} below.
In Section ~\ref{Cheb_3} and Section ~\ref{Cheb_5}, 
we extend this analysis to 
dual canonical basis elements defined by the dominant lattice paths
of $\mathrm{Thick}_k(W)),$ $k\in\{3,5\}$,
for $W$ the hexagonal web. We then 
focus on their expression in Kuperberg's web basis.
Our main results are then given in Proposition ~\ref{a_1} 
and Proposition ~\ref{chebnotdual5}.
In Corollary ~\ref{spec1},  Corollary ~\ref{spec3} and Corollary ~\ref{contra2} 
we present new implications for $\mathrm{SL(V )}$-invariants.

\begin{figure}[t]
 \includegraphics[width=\textwidth]{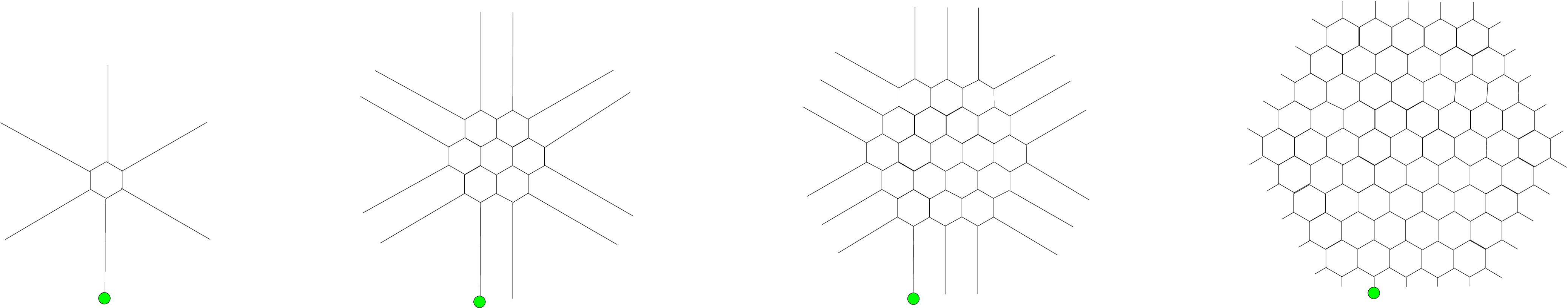}
 \caption{The web invariants $[\mathrm{Thick}_k(W)],$ $k\in\{1,2,3,5\}.$
}\label{thick35}
\end{figure}


\subsection{The case $\ell(\mathrm{Thick}_2(W_n))$ }\label{Cheb_2}

Let $W_n$ be any single-face non-elliptic web with an internal face bounded by at least 6 edges. 
Let $M_n$ be the minimal single-face non-elliptic web in $W_n$ (which might coincide with $W_n$).
Let $B(W_n)$ be the full unclasping of the non-elliptic web $b(W_n)$
associated with $W_n,$ see Section ~\ref{coeffsub}. 
The next result can be deduced for example from Theorem ~\ref{noredgraph}.
\begin{lemma}\label{[B(W)]_is_dc}
The web-invariants 
$[W_n]\in \mathrm{Inv}(V^{\sigma(W_n)})$ and
$[B(W_n)]\in\mathrm{Inv}(V^{\sigma(B(W_n))})$ 
are dual canonical. \qed
\end{lemma}

\begin{rem}
The invariants described in the previous result are special cases
of so called {\em superficial web invariants},
investigated in L.H.\ Robert's work \cite{LHRI}. 
\end{rem}

The proof of the next Theorem ~\ref{band2cd:2} combines Khovanov-Kuperberg's 
flow line approach with the combinatorics of red graphs of Theorem ~\ref{noredgraph}.

\begin{thm}\label{band2cd:2}
The invariant $$[\mathrm{Thick}_2(W_n)]-[B(W_n)]\in\mathrm{Inv}(V^{\sigma(\mathrm{Thick}_2(W_n))})$$
is a dual canonical basis element for all $n\geq6.$ 
\end{thm}
\begin{proof}
When $n=6$ this is Theorem ~\ref{thm.4}. For all other values of $n$
we show that $\mathrm{Thick}_2(W_n)$ has exactly two flow lines
with positive weight: one with boundary state $J(\mathrm{Thick}_2(W_n))$, 
the second with boundary state $J(B(W_n)).$  

We proceed by distinguishing two cases.
First,  we restrict our attention to $\mathrm{Thick}_2(M_n).$ The thickening procedure implies
that all bounded faces in $\mathrm{Thick}_2(M_n),$ except the most internal 
one, are bounded by six edges, see Figure ~\ref{flowthick2} for an example. 
Removing any $H$ piece from the boundary of $\mathrm{Thick}_2(M_n),$ produces
a non-elliptic web $M_n'$ having the property that all bounded faces of $M_n'$ are 
adjacent to the unbounded face. Hence, $M_n'$ has no exact red 
graph and is a dual canonical basis element by Theorem ~\ref{noredgraph}.
We then follow the same steps as in the proof of 
Theorem 4 in \cite{KK} and observe that a hypothetical state $x$ with the 
non-negative exponent property must
either have weight $v$ or $1$ on the $H$ we removed. In the latter case,
the state $x$ must restrict to the leading state of $M_n'$
and differ from the leading state of $M_n.$  
This however is impossible, as there is only one way to extend such an $x$ on $H.$ 
Hence, no such $H$ in $\mathrm{Thick}_2(M_n)$ can have weight 1 and they must all 
have weight $v.$ There is then only one way to complete such a flow to a positive flow
of $\mathrm{Thick}_2(M_n),$ forcing $x$ to be as indicated
in Figure ~\ref{flowthick2}. Notice that  the order used 
to construct $\mathrm{Thick}_2(M_n)$ with the growth algorithm does not matter, \cite[Lemma 1]{KK}.

We then conclude by noticing
that the boundary state of $x$ is the leading state of $B(M_n).$ 
Moreover, $B(M_n)$ is dual canonical by Lemma ~\ref{[B(W)]_is_dc}.
Therefore, the difference $\mathrm{Thick}_2(M_n)-B(M_n)$ has exactly
one flow with positive weight and $[\mathrm{Thick}_2(M_n)]-[B(M_n)]$
has the non-negative exponent property. 

Second, we consider the case where $\mathrm{Thick}_2(W_n)$ is 
given by $\mathrm{Thick}_2(M_n)$ with some $\mathrm{Thick}_2(T_i)$
attached to the boundary, where $T_i$ are some planar tree diagrams. 
Then the state $x$ on $\mathrm{Thick}_2(M_n)$ extends 
to two different positive states on $\mathrm{Thick}_2(W_n),$ if there are at
least two distinct states on some $\mathrm{Thick}_2(T_i),$ beginning in 1 and ending 
in -1 on the two legs connecting to $\mathrm{Thick}_2(M_n).$
But each $[\mathrm{Thick}_2(T_i)]$ is dual canonical for all $i,$ 
since there is no exact red graph in $\mathrm{Thick}_2(T_i).$
Therefore, there is only one such state on $\mathrm{Thick}_2(T_i)$ with positive weight, namely the
leading state of each $\mathrm{Thick}_2(T_i).$

It follows that the state $x$ can be 
extended uniquely to a positive state of $\mathrm{Thick}_2(W_n).$
Moreover, the boundary state of this flow is $J(B(W_n)),$
as it is the leading state on each $\mathrm{Thick}_2(T_i)$ and
the positive state on $\mathrm{Thick}_2(M_n).$
The state $J(B(W_n))$ parametrizes the non-elliptic web 
$B(W_n),$ which is dual canonical by Lemma ~\ref{[B(W)]_is_dc}.

To complete the proof, observe that $[\mathrm{Thick}_2(W_n)]-[B(W_n)]$ is invariant under
$\Phi$ by Proposition ~\ref{phi_in}. 

\begin{figure}[ht]
 \includegraphics[width=10cm,height=5cm]{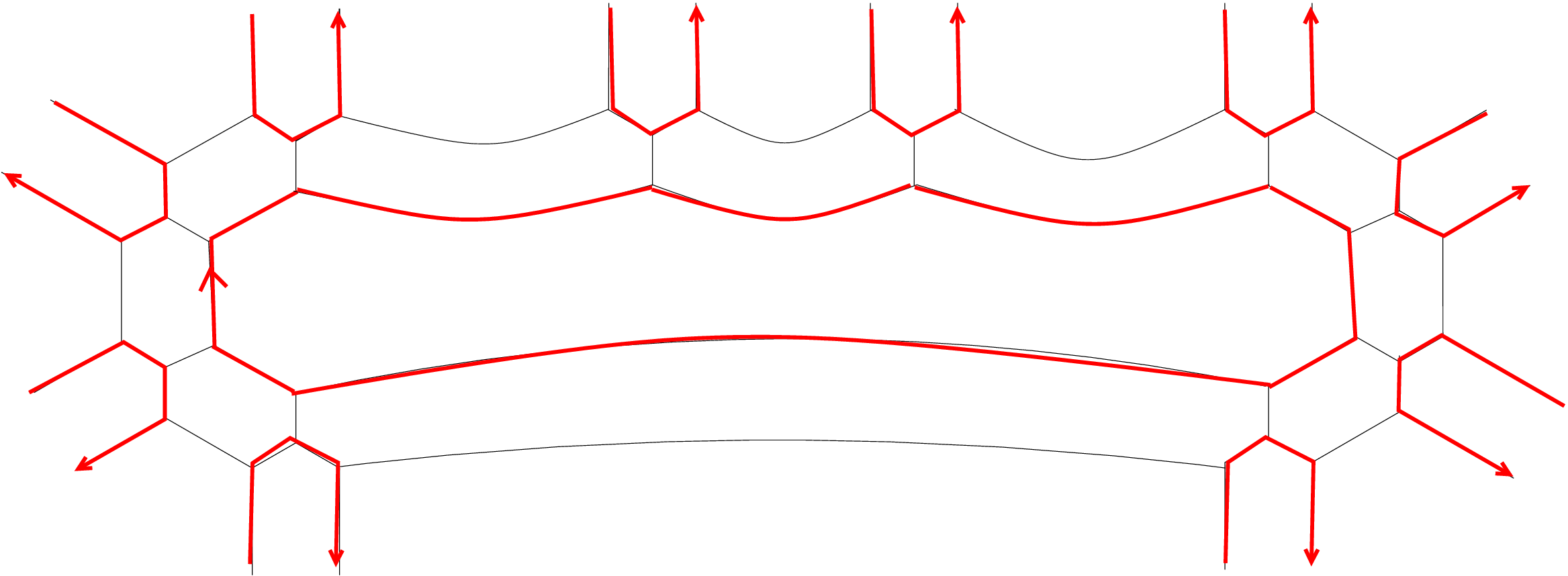}
 \caption{Flow on $\mathrm{Thick}_2(M_n)$ with boundary state $J(B(M_n))$ and overall weight 1.} \label{flowthick2}
\end{figure}
\end{proof}

In the next result, let $\sigma_1=\sigma(W_n)$ and $\sigma_2=\sigma(\mathrm{Thick}_2(W_n)).$
\begin{cor}\label{spec1}
The invariants 
$[\mathrm{band}_2(W_n)]\in R_{\sigma_1}(V) \textrm{ and }
[\mathrm{Band}_2(W_n)]\in R_{\sigma_2}(V)$
are dual canonical.
\end{cor}

\begin{proof}
By Theorem ~\ref{band2cd:2} we know that $\mathrm{Thick}_2(W_n)-B(W_n)$ 
defines a dual canonical invariant in $\mathrm{Inv}(V^{\sigma_2})$. 
But the difference $\mathrm{Thick}_2(W_n)-B(W_n)$ also defines the multilinear invariant
$ [\mathrm{Band}_2(W_n)]=[\mathrm{Thick}_2(W_n)-B(W_n)]$ in $R_{\sigma_2}(V)$.
Clasping all successive equally colored substrings of $\sigma_2$ yields the invariant
\begin{align*}
[\mathrm{band}_2(W_n)]=[\mathrm{thick}_2(W_n)]-[b(W_n)] \in R_{\sigma_1}(V)
\end{align*}
described as in Theorem ~\ref{thmC2:2} and the claim follows.
\end{proof}
This result is special, since the invariants $[\mathrm{Thick}_2(W_n)]$ and $\ell(\mathrm{Thick}_2(W_n))$
in $\mathrm{Inv}(V^{\sigma_2})$
differ only by one term. Moreover, the coefficient of $[B(W_n)]$ is simply $1$.


\subsection{The case $\ell(\mathrm{Thick}_3(W))$}\label{Cheb_3}
Consider the $U_q(\mathfrak{sl}_3)$-invariant $[\mathrm{Thick}_3(W)]$
defined by the non-elliptic web in Figure ~\ref{thick35}. 
The dominant lattice path and 
signature of  $[\mathrm{Thick}_3(W)]$ are as in Lemma \ref{latticepath3}.
Let $\ell(\mathrm{Thick}_3(W))$ be the dual canonical basis element
indexed by the same lattice path.
In Proposition ~\ref{a_1} we consider the
decomposition of $\ell(\mathrm{Thick}_3(W))$ in Kuperberg's web basis. 
To do so, we first use the combinatorics of red graphs and
determine the web invariants arising in the linear combination.
To find the coefficients in this decomposition we use the flow lines.

\begin{figure}[ht]
\centering
\psfragscanon
\psfrag{-b}{$-2$}
\psfrag{-e}{$\pm [R]$}
              \includegraphics[scale=0.28]{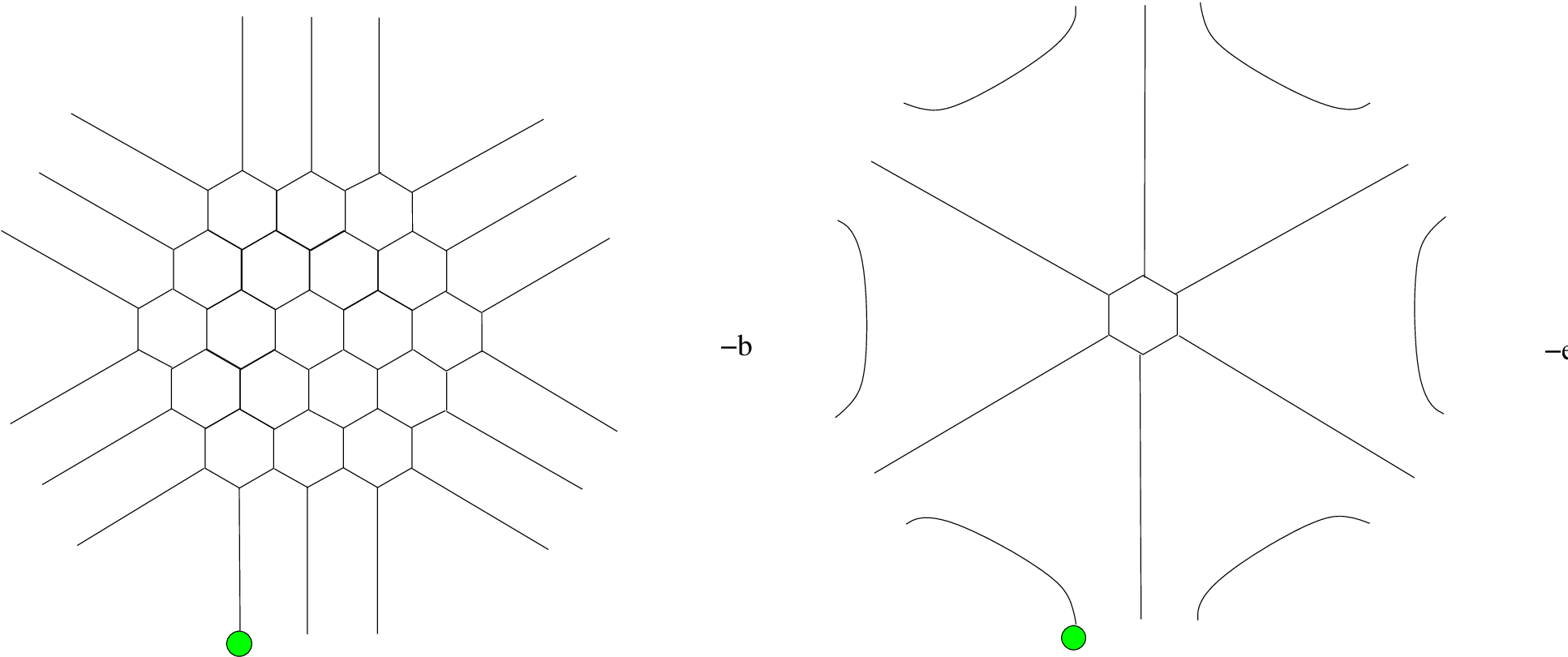}
               \caption{The invariant $[\mathrm{Band}_3(W)]=[\mathrm{Thick}_3(W)]-2[W\cup B]\pm[R].$ }\label{noncan3}
\end{figure}

Throughout, let $W$ be the hexagonal web.
Let $B$ be the full unclasping of $B(W)$ defined by the non-elliptic web
consisting of a six-tuple of $\mho$ shaped tensor diagrams, 
as drawn on the right of Figure ~\ref{che2}. 
The next Lemma ~\ref{latticepath3} provides the dominant 
lattice path of the non-elliptic webs $\mathrm{Thick}_3(W)$ and $W\cup B$ illustrated in Figure ~\ref{noncan3}.
The first part of the claim can be computed with the minimal cut algorithm, 
described in Section ~\ref{sect:growth}
specifying the first boundary vertex as in Figure ~\ref{noncan3} (green in the figure). 
The signatures are cyclic words that can be read off directly from the 
boundary points of the corresponding webs 
following the clockwise order.
In the represented tensor diagrams
we omitted the specific coloring of the vertices since it is irrelevant for
what follows.

\begin{lemma}\label{latticepath3}\leavevmode
\begin{itemize}
\item The dominant lattice path of $\mathrm{Thick}_3(W),$ resp.\ of $W\cup B,$ are:
\vspace{0.2cm}
\begin{center}
\begin{tabular}{r c l}
$J(\mathrm{Thick}_3(W))$&=&(1\ 111\ 011\ 000 \ -100\ -1-1-1\ -1-1) \\
$J(W\cup B)$&=&(1\ -111\ -111\ -101\ -101\ -1-11\ -1-1).\\
\end{tabular}
\end{center}
\vspace{0.2cm}
\item The signature of $\mathrm{Thick}_3(W),$ resp.\ $W\cup B,$ are both equal to:
\begin{align*}
\sigma_3&=(s_1,s_2,s_2,s_2,s_1,s_1,s_1,s_2,s_2,s_2,s_1,s_1,s_1,s_2,s_2,s_2,s_1,s_1)
\end{align*}
$s_1\neq s_2\in\{\circ,\bullet\}.$
\end{itemize} \qed
\end{lemma}

In the applications we have in mind, it will be convenient to distinguish certain 
particularly simple web invariants from the others, see Definition ~\ref{Yatbound} below.
These web invariants have the property that they vanish with 
an appropriate clasping. 

\begin{de}\label{Yatbound} 
A non-elliptic web $W_n$ with signature $\sigma$ has a {\em $Y$ at the boundary}
if the signature $\sigma=(s_1,s_2,\dots, s_n)$ and the dominant path $J(W_n)$ of $W_n$ 
(both viewed as cyclic words, here)
have substrings of the form $$\Big((s_i,s_{i+1}), (1,0)\Big); \Big((s_{i},s_{i+1}),(0,-1)\Big);\Big((s_{i},s_{i+1}), (1,-1)\Big)$$
with $s_i=s_{i+1}\in\{\circ,\bullet\}.$
\end{de}
The terminology used in Definition ~\ref{Yatbound} agrees with the one 
used in the growth algorithm, defined in \cite[Section 5]{KK}.

\begin{lemma}\label{terms3}
The dual canonical invariant
$\ell(\mathrm{Thick}_3(W))\in\mathrm{Inv}(V^{\sigma_3})$ 
decomposes in Kuperberg's web basis as:
\begin{align}\label{dec3}
\ell(\mathrm{Thick}_3(W))&= [\mathrm{Thick}_3(W)]\pm a_1 [W\cup B] 
\pm \sum_{i}a_i[L_i] \pm \sum_{j}a_j[R_j]
\end{align}
where $a_1, a_i, a_j\in\mathbb Z[v,v^{-1}]$ and $[W\cup B],[L_{-}],[R_{-}]$ are all different web invariants. 
Moreover, all $L_-$ are obtained as $G$-reductions of $\mathrm{Thick}_3(W)$ and have a $Y$ at the boundary.
\end{lemma}
\begin{proof}
The coefficients are in $a_1,a_i,a_j\in\mathbb Z[v,v^{-1}]$ by
Theorem ~\ref{thmkk}.
Moreover, the first summand in Equation \eqref{dec3} is the unique non-elliptic
web invariant parametrized by the dominant path
$J(\mathrm{Thick}_3(W))$ and the signature $\sigma_3.$

The summands $a_1 [W\cup B]\pm \sum_{i}a_i[L_i]$
are obtained using the combinatorics of red graphs.
To find these web invariants, we first identify all possible exact 
paired red graphs, $G,$ in $\mathrm{Thick}_3(W).$ 
Without loss of generality, it is enough to consider $G$'s
consisting only of one cycle and with no additional 
trees attached to the cycle.
For these $G$'s there is only one possible pairing.
Second, we analyze all possible $G$-reductions of 
$\mathrm{Thick}_3(W)$ that can occur.
To do this, we vary both the position of $G$ and augment the number of vertices 
of the cycle. We then notice that altogether there are 65 non-elliptic webs obtained as
$G$-reductions of $\mathrm{Thick}_3(W)$ for all possible 
$G$'s in $\mathrm{Thick}_3(W)$: 
two of which are isotopic to $W\cup B,$ the web shown in
the middle of Figure ~\ref{Gredfig}.
All other 63 diagrams have a $Y$ at the boundary.
To see that the coefficients 
$a_1\neq 0, a_i\neq 0,$ it is enough to find the corresponding flow lines on
$\mathrm{Thick}_3(W)$.
\end{proof}

The next claim is an immediate consequence of Theorem ~\ref{noredgraph}.

\begin{lemma}\label{termsdc}
The web invariants $[L_{-}]$ and 
$[W\cup B]$ in the decomposition of Lemma ~\ref{terms3} 
are dual canonical.\qed
\end{lemma}

The main obstacle in completing the
decomposition of Proposition \ref{terms3} is that we don't know if
all invariants in Equation \eqref{dec3} are 
obtained as $G$-reductions of exact red graphs.
We believe that this should be true, since
that's the case in many examples 
including $\ell(\mathrm{Thick}_2(W_n))$ for all single-face non-elliptic webs $W_n,$ 
as we saw in Theorem ~\ref{band2cd:2} and superficial web invariants, 
as it follows from work of L.H.\ Robert \cite{LHRI}.
See also the example given in Figure 20 in L.H.\ Robert's other contribution presented in \cite{LHRI}. 

In the next result, we provide an accurate value of the coefficient of the simple tensor 
indexed by the dominant path $J(W\cup B)$ and signature $\sigma_3$
given in Lemma ~\ref{latticepath3}.

\begin{prop}\label{a_1} 
The invariant $[\mathrm{Thick}_3(W)]\in\mathrm{Inv}(V^{\sigma_3})$
decomposes in the tensor product basis as
\begin{align}\label{eq:prop1}
\begin{split}
[\mathrm{Thick}_3(W)]=&e_{J(\mathrm{Thick}_3(W))}^{\sigma_3} + (2+p(v))e_{J(W\cup B)}^{\sigma_3}\\
&+\sum_{\substack{J'< J(\mathrm{Thick}_3(W))\\ J'\neq J(W\cup B)}}c(\sigma_3,J(\mathrm{Thick}_3(W)), J') e_{J'}^{\sigma_3} 
\end{split}
\end{align}
for $p(v)\in v^{-1}\mathbb Z_{>0}[v^{-1}]$ and $c(\sigma_3,J(\mathrm{Thick}_3(W)), J')\in\mathbb Z_{>0}[v,v^{-1}].$
\end{prop}

\begin{proof}
Consider the expansion given in Lemma ~\ref{terms3}. Given Theorem ~\ref{thm3KK}, 
to prove the claim, all we need is to show 
$$c(\sigma_3,J(\mathrm{Thick}_3(W)), J(W\cup B))=2+p(v).$$
For this, consider a disc $D$ with as many marked points
on its boundary as entries in $\sigma_3.$ 
Assume that these boundary points are colored according to $\sigma_3$ 
and indexed by $J(W\cup B)$ with one point distinguished as first.
Then we determine all possible maximal collections of non-crossing oriented
arcs in $D$ starting in $1$ and ending in $-1.$ We find that there are precisely
50 different such collections.
We then consider all possible embeddings of these
maximal collections of non-crossing 
oriented arcs into the non-elliptic web $\mathrm{Thick}_3(W),$ such
that their images are flow lines with boundary state $J(W\cup B).$
We then complete these flow lines to flows maximizing their weight.
This can be done only by adding as many clockwise oriented closed loops as possible.
We then compute the overall weight of these flows and find that only two flows are such that their 
overall weight is not in  $v^{-1}\mathbb Z_{>0}[v^{-1}].$ These flows are represented
in Figure ~\ref{flow3}. In each case, the overall weight of each flow is $1,$ as we explain more carefully below. 
All other flows have negative weight; hence they define a polynomial in $v^{-1}\mathbb Z_{>0}[v^{-1}].$
Notice that the coefficients are polynomials with positive integers follows already from the definition of
flow lines.

To determine the overall weight of each flow, one can directly use the local rules
of Figure ~\ref{flow3} at each trivalent vertex, since we have a representation
of $\mathrm{Thick}_3(W)$ consisting only of $\lambda$ and $Y$ pieces.
In Figure ~\ref{flow3}, we color each flow line according to its weight, 
so that black flow lines have weight 1, red flow lines have weight $v$
and blue flow lines have weight $v^{-1}.$
Multiplying the weights of each flow line implies that each flow in Figure ~\ref{flow3}
has weight $v^{4}.$ The overall weight of each flow is then obtained by attaching
each non-elliptic web up to a horizontal line using $\mho$-pieces. 
The linear order is determined by the first vertex of $J(W\cup B).$ 
At this stage one creates nine points with vanishing tangent. Their weight 
can be determined with Figure ~\ref{weightY}, and the claim follows. 
\end{proof}

\begin{figure}[ht]
               \includegraphics[width=5cm,height=5cm]{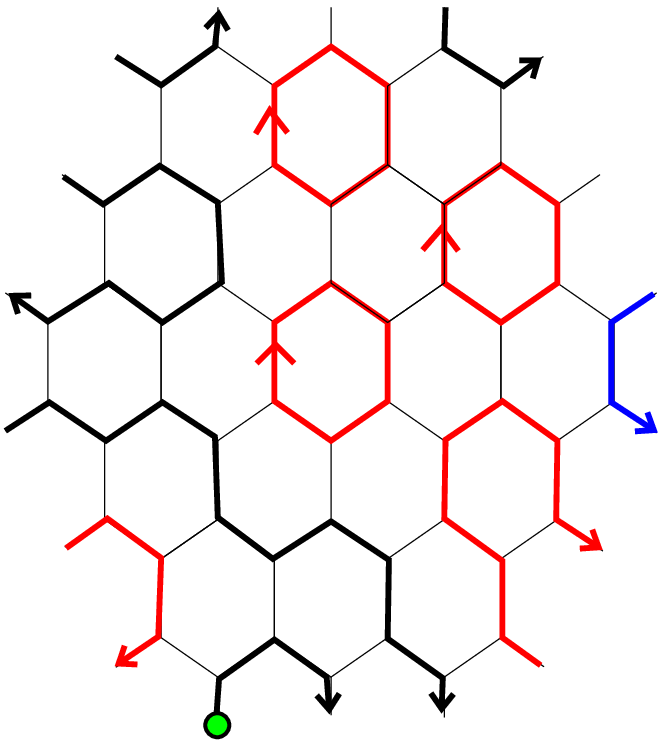}
               \includegraphics[width=5cm,height=5cm]{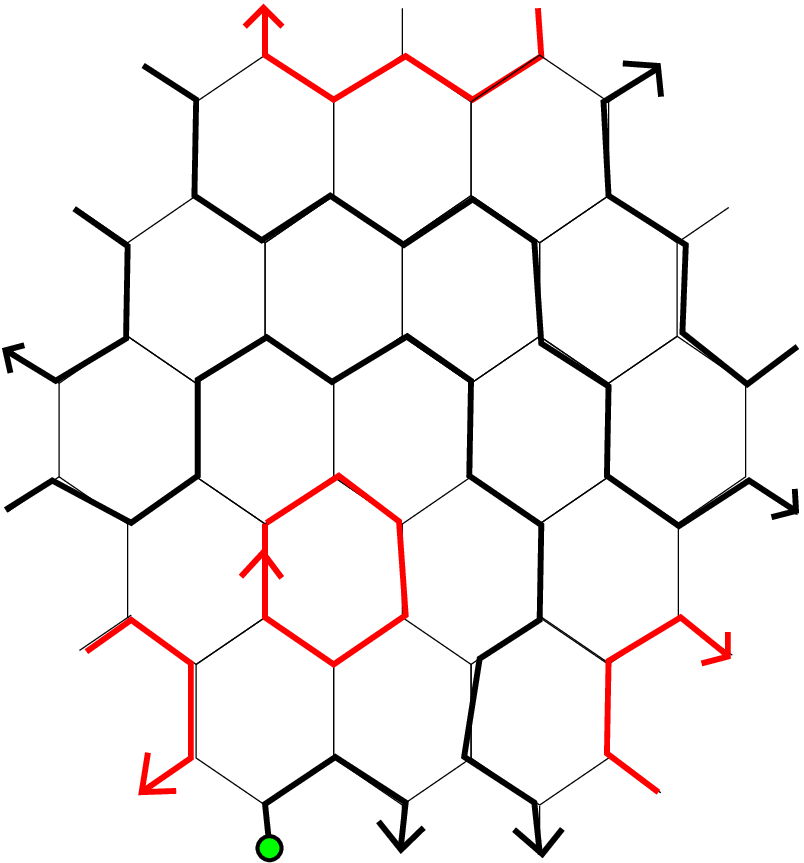}
               \caption{Two flows of $\mathrm{Thick}_3(W)$ with boundary state $J(B\cup W)$
                and overall weight 1.}\label{flow3}
\end{figure}
In the next result, let $\sigma_1=\sigma(W)$ and consider
Equation \eqref{dec3} in Lemma ~\ref{terms3} above.
\begin{cor}\label{spec3}
The invariant $[\mathrm{band}_3(W)]\in R_{\sigma_1}(V)$ is dual canonical
if and only if the terms $a_j$ are such that $\displaystyle\lim_{q\rightarrow 1} a_j=0$ for all $j,$ 
or all $[R_j]$ are zero, or all $R_j$ have a $Y$ 
at the boundary.
\end{cor}
\begin{proof}
In Lemma ~\ref{terms3} the coefficient of
$[W\cup B]$ satisfies $a_1=-2,$ as there are exactly two flows of positive
weight and boundary state $J(W\cup B)$ in $\mathrm{Thick}_3(W),$ by what we saw in the proof of
Proposition ~\ref{a_1}.
Hence, we have on the one side that
\begin{align}\label{eq1}
\ell(\mathrm{Thick}_3(W))= [\mathrm{Thick}_3(W)]-2[W\cup B]\pm \sum_{i}a_i[L_i] \pm \sum_{j}a_j[R_j]
\end{align}
in $\mathrm{Inv}(V^{\sigma_3})$. Moreover, in Equation \eqref{eq1} all $[L_i]$ are different from
$[W\cup B],$ since all non-elliptic webs have a $Y$ at the boundary, and are dual canonical, by Lemma ~\ref{termsdc}.
On the other side, we have the invariant 
\begin{align}\label{eq2}
[\mathrm{band}_3(W)]=[\mathrm{thick}_3(W)]-2[W][b(W)]
\end{align}
in $R_{\sigma_1}(V),$ as given in Theorem ~\ref{thmC2:2}.
Moreover, the terms 
$\sum_{i}a_i[L_i]$
all vanish after clasping $\sigma_3$ at all successive and equally colored substrings.  
From this it follows that Equation \eqref{eq2} is a specialization of Equation \eqref{eq1} at $q\rightarrow 1,$ if and only if the conditions in the claim are satisfied.
\end{proof}
\begin{rem}
$[\mathrm{Band}_3(W)]\in R_{\sigma_3}(V)$
is not dual canonical, since 
the expansion of $[\mathrm{Band}_3(W)]$ in the web basis does not
agree with the expansion of $\ell(\mathrm{Thick}_3(W))$.
\end{rem}


\subsection{The case of $\ell(\mathrm{Thick}_5(W))$}\label{Cheb_5}

In this section, we focus on how $\ell(\mathrm{Thick}_5(W))$ decomposes
in Kuperberg's basis, following the approach used in Section ~\ref{Cheb_3}.

The next result can be deduced as Lemma ~\ref{latticepath3}.

\begin{figure}[t]
\center
\psfragscanon
\psfrag{- a}{$-3$}
\psfrag{- b}{$+4$}
\psfrag{- e}{$\pm R$}
\center
 \includegraphics[scale=0.28]{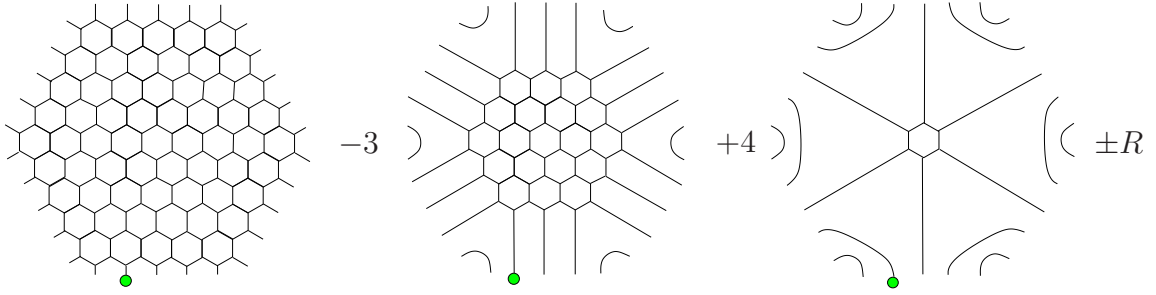}
\caption{The invariant $[\mathrm{Band}_5(W)]$ expressed in Kuperberg's basis as
$[\mathrm{Thick}_5(W)]-3[\mathrm{Thick}_3(W)\cup B]+4[W\cup B\cup B]\pm[R].$}\label{noncan5}
\end{figure}

\begin{lemma} \label{latticepath5} \leavevmode
\begin{itemize}
\item The dominant lattice path of $\mathrm{Thick}_5(W),$ resp.\ of $W\cup B\cup B,$ are:
\vspace{0.2cm}
\begin{center}
\begin{tabular}{ r c l}
$J(\mathrm{Thick}_5(W))$&=&(11\ 11111\ 00111\ 00000 \ -1-1000\ -1-1-1-1-1 \ -1-1-1) \\
$J(W\cup B\cup B)$&=&(11\ -1-1111 \ -1-1111\ -1-1011\ -1-1011\ -1-1-111\ -1-1-1).\\
\end{tabular}
\end{center}
\vspace{0.2cm}
\item The signatures of $\mathrm{Thick}_5(W),$ resp.\ $W\cup B\cup B,$ are both equal to:
\begin{align*}
\sigma_5&=(s_1,s_1, \underbrace{s_2,\dots,s_2}_{5-\textrm{times}},\underbrace{s_1,\dots,s_1}_{5-\textrm{times}},\underbrace{s_2,\dots,s_2}_{5-\textrm{times}},
\underbrace{s_1,\dots,s_1}_{5-\textrm{times}},\underbrace{s_2,\dots,s_2}_{5-\textrm{times}}, s_1,s_1,s_1)
\end{align*}
for $s_1\neq s_2\in\{\circ,\bullet\}.$
\end{itemize}
\qed
\end{lemma}

\begin{lemma}\label{terms5}
The dual canonical basis element
$\ell(\mathrm{Thick}_5(W))\in\mathrm{Inv}(V^{\sigma_5}),$ 
decomposes in Kuperberg's web basis as:
\begin{align*}
\ell(\mathrm{Thick}_5(W))&=[\mathrm{Thick}_5(W)]\pm c_1[\mathrm{Thick}_3(W)\cup B]\pm
c_2[W\cup B\cup B]\\
&\pm
 \sum_{k}c_k[L_{k}]\pm \sum_{j}c_j[R_{j}]
\end{align*}
where $c_1, c_2, c_k, c_j\in\mathbb Z[v,v^{-1}]$  and
$[\mathrm{Thick}_3(W)\cup B], [W\cup B\cup B], [L_{-}], [R_{-}]$ 
are all different web invariants.
Moreover, all $L_-$ are obtained as $G$-reductions of $\mathrm{Thick}_5(W)$ and have a $Y$ at the boundary.
\end{lemma}

\begin{proof}
We proceed as in the proof of Lemma ~\ref{terms3}.
That is, we analyze all possible $G$-reductions of\ $\mathrm{Thick}_5(W)$ that can occur.  

Varying the position of $G$ we notice that, up to isotopy fixing the boundary, 
$G$-reductions of $\mathrm{Thick}_5(W)$ split into two classes: 
Those with a $Y$ piece at the boundary and those without. It can be checked that the elements of the
latter are precisely $\mathrm{Thick}_3(W)\cup B$
and ${W\cup B\cup B}.$ 

All other $G$-reductions $L_-$ arising have a $Y$ piece at the boundary. These $L_-$ might be elliptic.
However, the $Y$ at the boundary remains after 
using skein relations.
\end{proof}
The next result can be deduced
observing that all the non-elliptic webs $L_-$ may have red graphs, and may 
be decomposed further iterating the above procedure. As the $Y$ at the boundary remains 
unchanged, the claim follows.
\begin{cor}\label{cortocorr5}
The $G$-reductions of $L_-$ never 
yield $W\cup B,$ resp.\ $\mathrm{Thick}_3(W)\cup B$ nor ${W\cup B\cup B}.$ \qed
\end{cor}

By Lemma ~\ref{terms5}, we know that $[\mathrm{Thick}_5(W)]$ is not a dual 
canonical basis element. In Proposition ~\ref{chebnotdual5} below we provide a lower estimate  
for the positive weight in the
coefficient of  the simple tensor $e_{J(W\cup B\cup B)}^{\sigma_5}$ indexed by the dominant path
$J(W\cup B\cup B)$ given in Lemma ~\ref{terms5}.

\begin{prop}\label{chebnotdual5} 
The invariant 
$[\mathrm{Thick}_5(W)]\in\mathrm{Inv}(V^{\sigma_5})$ 
decomposes in the tensor product basis as
\begin{align}\label{eq:cheb5}
\begin{split}
[\mathrm{Thick}_5(W)]=&e_{J(\mathrm{Thick}_5(W))}^{\sigma_5}  + (6+p(v) ) e_{J(W\cup B\cup B)}^{\sigma_5}\\
&+
\sum_{J'< J(\mathrm{Thick}_5(W))}c(\sigma_5,J(\mathrm{Thick}_5(W)), J') e_{J'}^{\sigma_5} 
\end{split}
\end{align}
for  $p(v)\in v^{-1}\mathbb Z_{>0}[v^{-1}]$ and $c(\sigma_5,J(\mathrm{Thick}_5(W)), J')\in\mathbb Z_{>0}[v,v^{-1}].$
\end{prop}
\begin{proof}
All we need to show is that
$$
c(\sigma_5,J(\mathrm{Thick}_5(W)), J(W\cup B\cup B))\geq (6+p(v)).
$$
For this it is enough to exhibit 6 different
flows on the non-elliptic web 
$\mathrm{Thick}_5(W)$ with boundary state $J(W\cup B\cup B)$
each with overall weight 1.
The flows illustrated in Figure ~\ref{flowcheb} satisfy this claim
as it can be seen proceeding as in the proof of Proposition ~\ref{a_1}. We colored the flows
in Figure ~\ref{flowcheb} according to their local weight (black means that the 
overall weight is 1, blue indicates that the overall weight is $v^{-1}$
while red means that the flow line has overall weight $v$).
The inequality then follows from observing that weights in state sums never cancel. 
\end{proof}
\begin{figure}[ht]  
               \includegraphics[width=5cm,height=5cm]{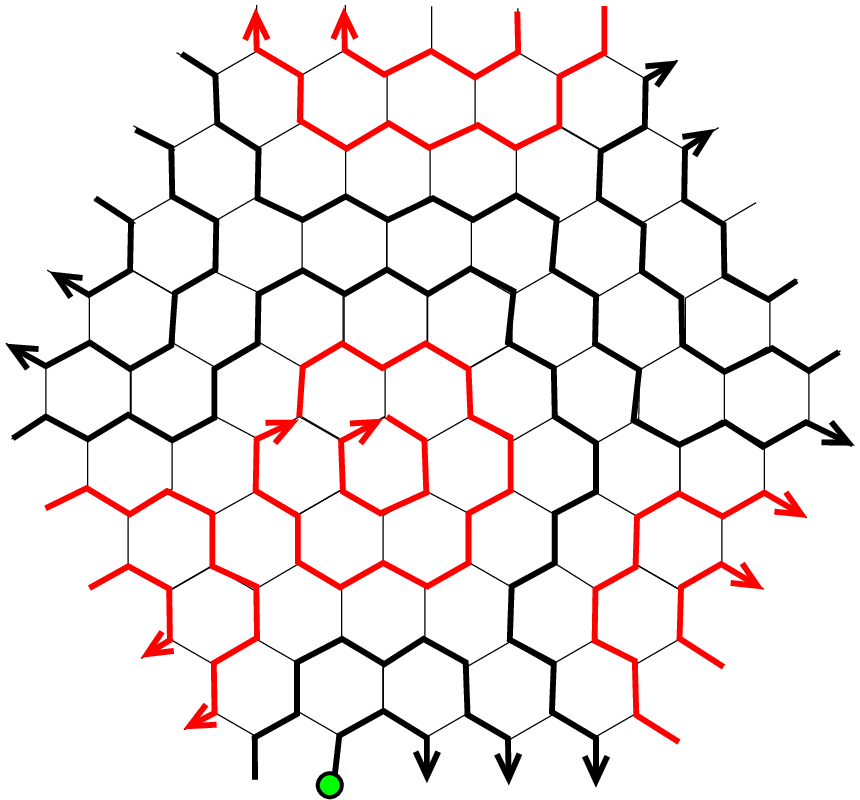}
               \includegraphics[width=5cm,height=5cm]{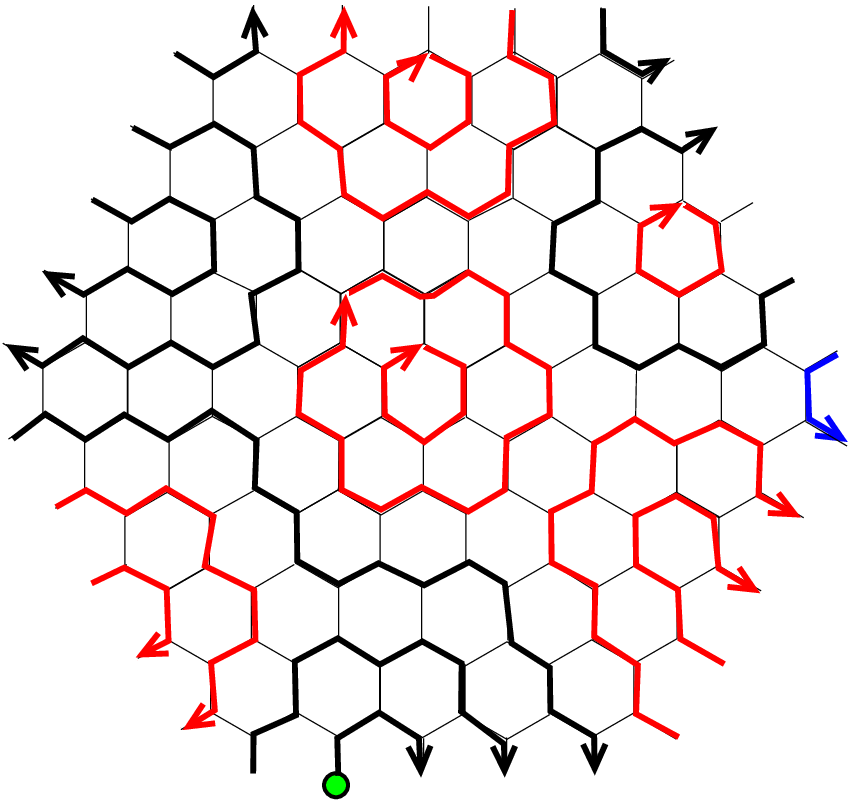}
               \includegraphics[width=5cm,height=5cm]{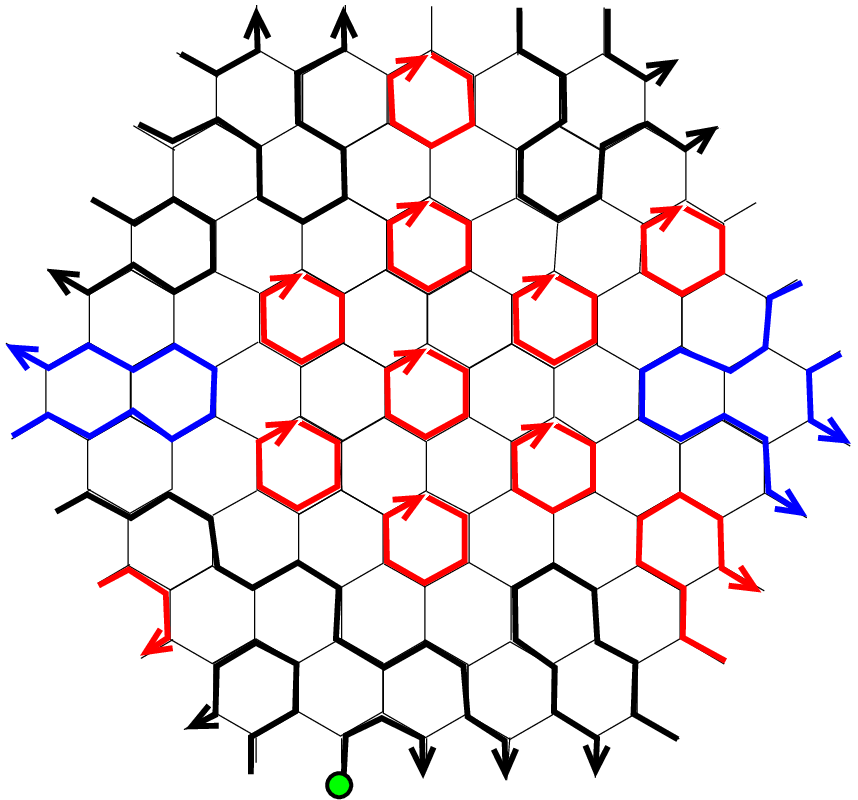}              
                \bigskip              
               \includegraphics[width=5cm,height=5cm]{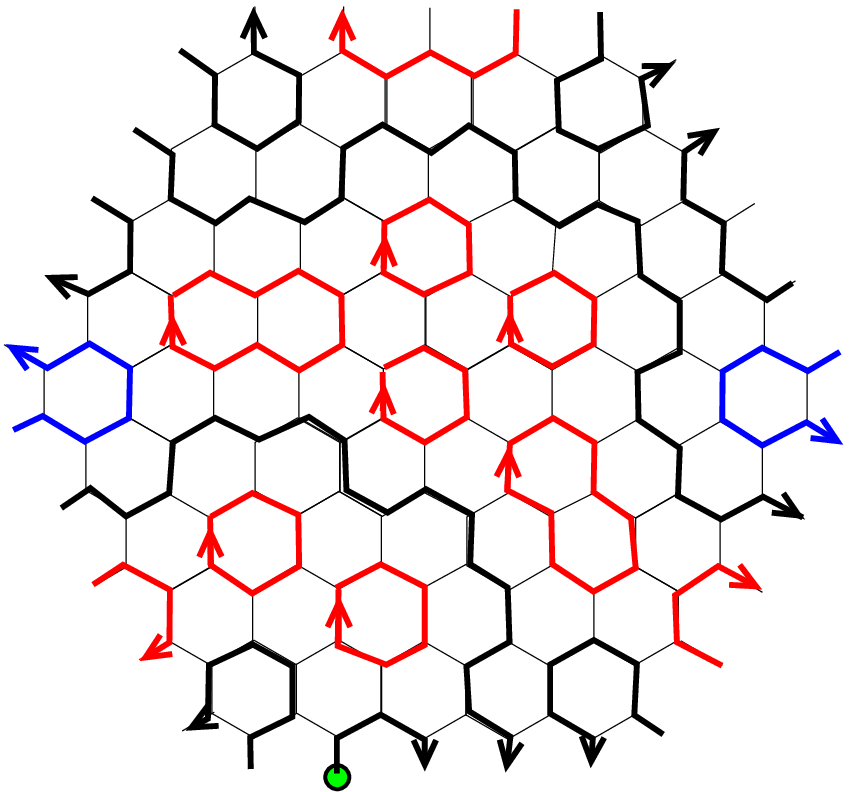}
               \includegraphics[width=5cm,height=5cm]{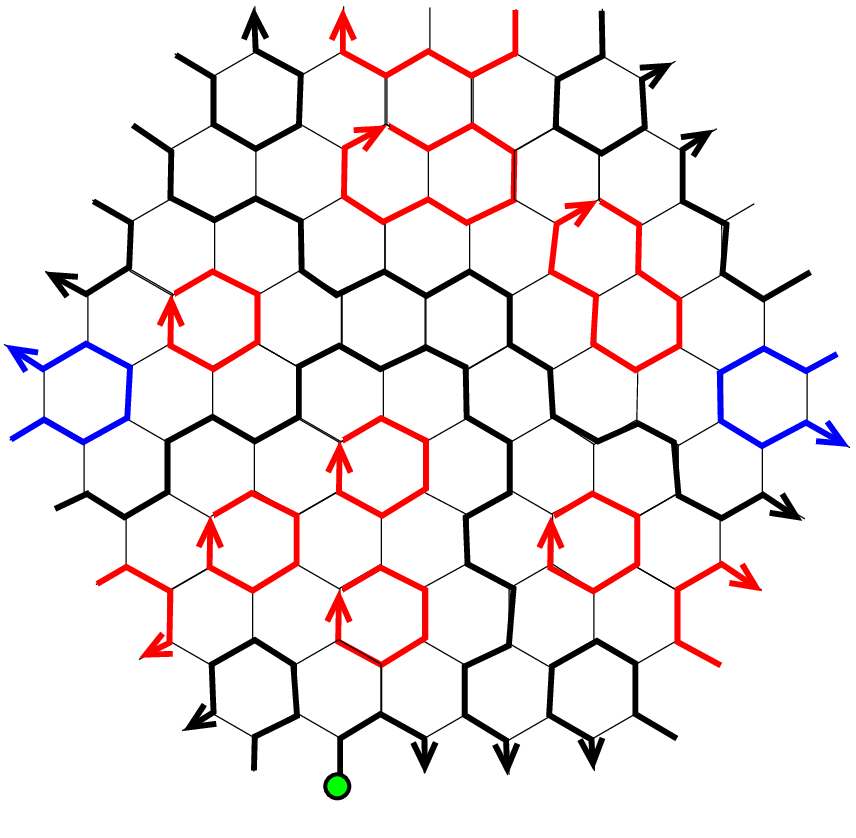}         
               \includegraphics[width=5cm,height=5cm]{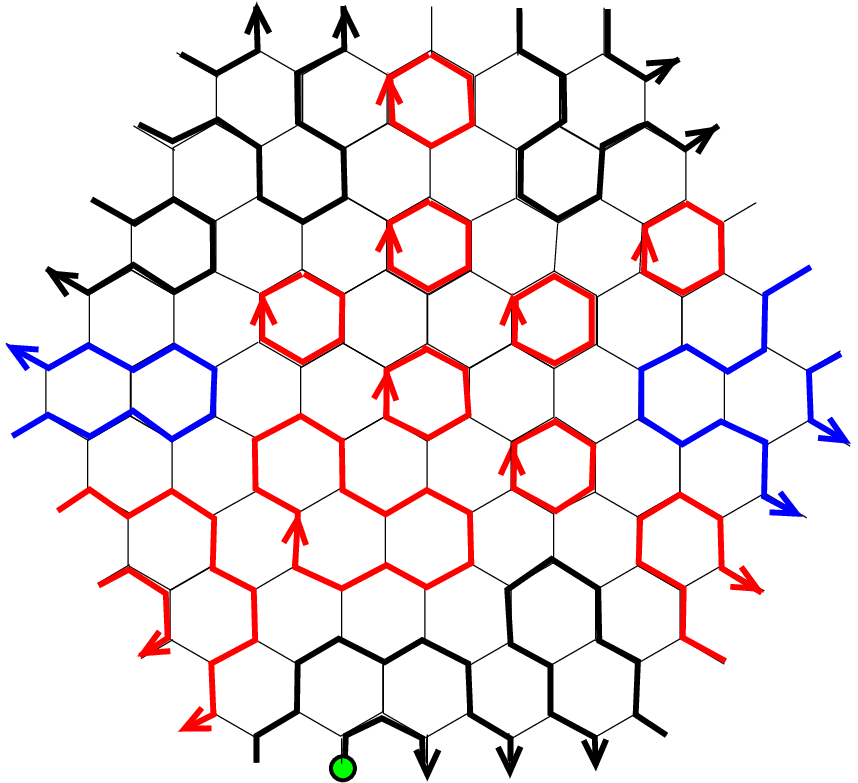}            
               \caption{Different flows on $\mathrm{Thick}_5(W)$ with boundary state 
               given by  $J(W\cup B\cup B)$ and positive overall
               weight.}\label{flowcheb}      
\end{figure}

The previous result shows that the
lower estimate for the coefficient of $e_{J(W\cup B\cup B)}^{\sigma_5}$ 
in the expansion given in Equation \eqref{eq:cheb5}  is
greater than the coefficient of $[b(W)]^2U_1=[b(W)]^2[W]\in R_{\sigma(W)}(V):$
$$
\binom{5}{2}-\binom{5}{1}=5
$$
given in the expansion of Corollary ~\ref{corcheb2}.
The proof of Proposition ~\ref{chebnotdual5} also shows that as $k$ and the number of vertices in $W$ increase, the more likely it is that flow lines have positive weight. To see this, note that more closed flow lines 
with positive weight (red oriented loops in the Figure ~\ref{flowcheb}) 
can be placed inside a larger tensor diagram.

In view of Corollary \ref{contra2}, let $\sigma_5$ be as
in Lemma ~\ref{latticepath5} and let $\sigma_1=\sigma(W)$.
Let $[\mathrm{band}_k(W)]\in R_{\sigma_1}(V)$, $k\in\{3,5\}$, be defined
by clasping the non-elliptic webs of Figure ~\ref{noncan3}, resp.\ Figure~\ref{noncan5}.
\begin{cor}\label{contra2}
Assume $\ell(\mathrm{Thick}_5(W))\in\mathrm{Inv}(V^{\sigma_5})$ has integer coefficients 
when expanded in Kuperberg's web basis.
Then the invariants  
$[\mathrm{band}_3(W)]$ and
$[\mathrm{band}_5(W)]$
of $R_{\sigma_1}(V)$ 
can not be simultaneously dual canonical basis element.
\end{cor}
\begin{proof}
Assume that both $[\mathrm{band}_3(W)]$ and $[\mathrm{band}_5(W)]$
belong to the specialization of Lusztig's dual canonical
basis for $R_{\sigma_1}(V)$ at $q=1$. 

Consider $[\ell(\mathrm{Thick}_3(W))]$ expressed in Kuperberg's basis
and let $B$ the non-elliptic web given on the right of Figure ~\ref{che2}.
Since $B$ is homotopic to the boundary of the disc containing 
the summands defining $[\ell(\mathrm{Thick}_3(W))]$ we can consider
the invariant $[\ell(\mathrm{Thick}_3(W))\cup B]$
obtained by
inserting in each summand the six $\mho$ pieces, 
after tuples of successive equally colored boundary vertices. 
Thus $[\ell(\mathrm{Thick}_3(W))\cup B]$ has
signature $\sigma_5$, obtained from $\sigma_3$ by inserting
$\sigma(\mho)$ as follows:
$$
\sigma=(s_1,\underbrace{s_1,s_2}_{\sigma(\mho)},
s_2,s_2,s_2,\underbrace{s_2,s_1}_{\sigma(\mho)},
\dots ,
s_2,s_2,s_2,\underbrace{s_2,s_1}_{\sigma(\mho)}
s_1,s_1).
$$
The invariant $[\ell(\mathrm{Thick}_3(W))\cup B]$ is then
dual canonical by Proposition ~\ref{propemd}. To see this, we observe that
$[\ell(\mathrm{Thick}_3(W))]$ is dual canonical by assumption 
and $[B]$ is dual canonical, again by Proposition ~\ref{propemd}.

Next, we express $\ell(\mathrm{Thick}_5(W))$ 
using Lemma ~\ref{terms3} as follows:
\begin{align*}
\begin{split}\label{eq:cheb5}
\ell(\mathrm{Thick}_5(W))
=&[\mathrm{Thick}_5(W)]+c_1[\mathrm{Thick}_3(W)\cup B]+
c_2[W\cup B\cup B]+\sum_{j}c_j[R_{j}]\\
=&[\mathrm{Thick}_5(W)]+c_2[W\cup B\cup B]+\sum_{j}c_j[R_{j}]\\
+&c_1\bigg([\ell(\mathrm{Thick}_3(W))\cup B]+ a_1 [W\cup B\cup B]
+\sum_{i}a_i[R_{i}\cup B]\bigg)\\
=&[\mathrm{Thick}_5(W)]
+(c_1a_1+c_2)[W\cup B\cup B]\\
+&c_1[\ell(\mathrm{Thick}_3(W))\cup B]+\sum_{k}b_k[R_{k}],
\end{split}
\end{align*}
where $J(R_j), J(R_{i}\cup B), J(R_k)$ are all different from $J(W\cup B\cup B)$ and $c_-,a_-\in\mathbb Z,$ by assumption.
Then since $\ell(\mathrm{Thick}_5(W))$ has non-negative exponent property: 
$$c({\sigma_5},J(\mathrm{Thick}_5(W)), J(W\cup B\cup B))=-(c_1a_1+c_2)=5.$$
Here the last equality follows since
we assume that $[\mathrm{band}_5(W)]$ 
is dual canonical and $\ell(\mathrm{Thick}_5(W))\in\mathrm{Inv}(V^{\sigma_5})$ has integer coefficients.
Therefore, we have that $c_1=-4$ and $c_2=3$ and $a_1=2$ in the first line in Equation \ref{eq:cheb5},
compare with Table ~\ref{chebtable}.
But this is impossible, since we proved in 
Proposition~~\ref{chebnotdual5} that 
$$c({\sigma_5},J(\mathrm{Thick}_5(W)), J(W\cup B\cup B))\geq 6,$$
and weights never cancel in $\mathrm{Inv}(V^{\sigma_5})$. 
\end{proof}
Corollary \ref{contra2} shows
that the simplest quantization of the band power
operation precludes the invariants $[\mathrm{band}_3(W)]$ 
and $[\mathrm{band}_5(W)]$ in $R_{\sigma_1}(V)$ 
to be simultaneously dual canonical. 

We went further and asked if it is possible to extend the band operation to the
$U_q(\mathfrak{sl}_3)$-invariant setting. For this we vertically superimposed 
copies of $W,$ using an appropriate notion of 
``clasping'' endpoints, and solved the two types of crossings using 
the quantum skein relations defined by G.\ Kuperberg in \cite{Kuperberg}. 
Expressing the corresponding invariant in Kuperberg's web basis, 
we find again the recursions of Theorem ~\ref{thmC2:2}, up to an overall 
scalar given by a power of the indeterminate $v$. Therefore, also for this band operation the statement of the above Corollary ~\ref{contra2} holds.

\section{Discussion and outlook}\label{out}

In this work we showed how a class of $\mathrm{SL(V)}$-invariant polynomial functions of
$R_{a,b}(V)$ can be described using recursion formulas, see Theorem ~\ref{thmC1.2}  and Theorem ~\ref{thmC2:2}.
The formulas we find agree with the recursions arising in
the context of surface cluster algebras.
To better understand this relationship,
one has to go back to the work of S.\ Fomin and P.\ Pylyavskyy, \cite{FP}.
From their results one deduces that there are values $a, b$ 
such that the ring $R_{a,b}(V)$ 
contains a sub-cluster algebra  structure of affine type $A_1^{(1)},$ see Figure 20 in \cite{FP}.
For this structure, the $\mathrm{SL(V)}$-invariant $[W]$ 
defined by the hexagonal web $W,$
can be expressed as $[W]=z_0z_3-z_1z_2\in R_{a,b}(V),$ where all
$z_i$ are cluster variables belonging to the two clusters $\{z_0,z_1\}$ and $\{z_2,z_3\}$ 
of $R_{a,b}(V),$ up to coefficient variables.
This description of $[W]$ matches the topological description of the variable $z$
established in several papers, see Section \ref{sect:intro}.
In addition, an important role is played by the
invariant $[b(W)]$ appearing in our recursions.
More precisely, this $\mathrm{SL(V)}$-invariant can be thought of as a generalization of
the coefficients described in \cite{MSW}.
To see this it suffices to remove the internal trivalent 
vertices of $b(W),$ using skein relations, and compare
with Definition 6.1. in \cite{FP}.
 
In light of these results and the evidences of B.\ Leclerc \cite{Leclerc} and 
of P.\ Lampe \cite{MR2817684}, 
A.\ Berenstein and A.\ Zelevinsky
\cite{MR3180605} as well as the results of  
M.~Ding and F.\ Xu 
\cite{MR2916330} it is generally 
expected that Chebyshev polynomials $(U_k)_{k\in\mathbb Z_{\geq 0}}$ in the variable $[W]$
belong the specialization at $q=1$ of Lusztig's dual canonical basis for $R_{a,b}(V)$.
However, when we relate our recursions to the dual canonical basis of the invariant space
$\mathrm{Hom}_{U_q(\mathfrak{sl}_3)}(V^\sigma,\mathbb C(q^\frac{1}{2}))$ 
for an appropriated tensor product $V^{\sigma}$
we find an example suggesting that there is a disagreement at the level of coefficients,
see Proposition ~\ref{chebnotdual5}.

A possible explanation for this inconsistency might be that $R_{a,b}(V)$ often carries 
simultaneously several cluster algebra structures, as explained in \cite{FP}
and it remains unclear how these generalize to the 
$U_q(\mathfrak{sl}_3)$-invariant setting.
Moreover, uncertainty still exists about the
role of the coefficient variable $[b(W)]$ and its relation
to the coefficient variables for the given quantum cluster algebra structures.
 
Despite these difficulties, a possible new conjecture to formulate appears to
be the following:

\begin{conj}\label{conj}
There is an invariant $[L]\in R_{a,b}(V)$ 
with the following properties:
\begin{itemize}  
\item $[L]$ is defined by a single tensor diagram $L$ 
(not necessarily planar) with a single internal face bounded $n\geq 6$ sides.
\item There is a quantum cluster sub-algebra, in the sense of A.~Berenstein and A.~Zelevinsky \cite{MR2146350}, of
Kronecker type in $\mathrm{Inv}(V^{\sigma_L})$ such that $[L]$ decomposes as
$vz_0z_3-v^3z_1z_2$ for $z_0, z_1, z_2, z_3$ 
quantum cluster variables in the two clusters $\{z_0,z_1\},$ and $\{z_2,z_3\}$.
\item The band operation of $[L]$ in $R_{a,b}(V)$ is defined. 
\end{itemize}
Then for all $k\geq 1,$  $[\mathrm{Band}_k(L)]$ is dual canonical at $q\rightarrow1$ and after clasping.
\end{conj}
Further work also needs to be
done to establish if other recursions can be used to 
compare the different bases for $R_{a,b}(V).$ For instance, 
it would be important to clarify which basis of $R_{a,b}(V)$
contains the Chebyshev polynomials $(T_k)_{k\in\mathbb Z_{\geq 0}}$ in the variable $[W]$. 
More research could also be conducted to determine
the role of Chebyshev recursions in related invariant spaces. 


\end{document}